\documentclass{amsart}
\title{Khovanov Homology and Rational Unknotting}
\author{Damian Iltgen}
\address{Department of Mathematics, University of Regensburg, Germany}
\email{\myemail{damian@iltgen.ch}}
\author{Lukas Lewark}
\address{Department of Mathematics, University of Regensburg, Germany}
\email{\myemail{lukas@lewark.de}}
\urladdr{\url{http://www.lewark.de/lukas/}}
\author{Laura Marino}
\address{IMJ-PRG, Universit\'{e} de Paris, France}
\email{\myemail{laura.marino@imj-prg.fr}}
\usepackage{thmtools,enumerate,graphicx,xcolor,enumitem,afterpage,float,amsmath,amssymb,amscd,tabu,diagbox,sseq,amssymb,epsfig,psfrag,microtype,thm-restate,url,stmaryrd,mathtools,tikz-cd,amsthm,caption,soul,placeins,skak,verbatim,transparent}
\usepackage[pagebackref]{hyperref}
\usepackage[utf8]{inputenc}
\usepackage[nameinlink]{cleveref}
\usepackage{subfig}
\usepackage{nicematrix}
\usepackage{bbm}
\usepackage{multirow}
\let\cref\Cref
\crefname{subsection}{section}{sections}
\Crefname{subsection}{Section}{Sections}
\Crefname{enumi}{}{}
\Crefname{equation}{}{}
\definecolor{darkblue}{RGB}{0,0,96}
\definecolor{gray}{RGB}{127,127,127}
\definecolor{darkred}{RGB}{160,0,0}
\definecolor{lightyellow}{RGB}{255,255,128}
\definecolor{myred}{rgb}{1,0.85,0.85}
\definecolor{myblue}{rgb}{0.75,0.75,1}
\hypersetup{colorlinks={true},linkcolor={black},citecolor={black},filecolor={black},urlcolor={black},pdfauthor={\authors},pdftitle={\shorttitle}}
\newcommand{\myemail}[1]{\href{mailto:#1}{#1}}
\newcommand{\qua}{\hskip 0.4em \ignorespaces}
\def\arxiv#1{\relax\ifhmode\unskip\qua\fi
\href{http://arxiv.org/abs/#1}%
{\tt arXiv:\penalty -100\unskip#1}}

\def\ZB#1{\relax\ifhmode\unskip\qua\fi
\href{https://zbmath.org/?q=an:#1}{\tt zb#1}}
\def\xox#1{\csname xx#1\endcsname}

\numberwithin{equation}{section}
\declaretheorem[numberwithin=section,name={Lemma}]{lem}
\newtheorem{theorem}[lem]{Theorem}
\newtheorem{corollary}[lem]{Corollary}
\newtheorem{prop}[lem]{Proposition}

\theoremstyle{definition}
\newtheorem{dfn}[lem]{Definition}

\newtheorem{remark}[lem]{Remark}

\newtheorem{example}[lem]{Example}
\newtheorem{question}[lem]{Question}
\DeclareMathAlphabet{\mathpzc}{OT1}{pzc}{m}{it}
\DeclareMathOperator{\Kom}{Kom}
\DeclareMathOperator{\Cob}{Cob}
\DeclareMathOperator{\Mat}{Mat}

\DeclareMathOperator{\Id}{Id}
\DeclareMathOperator{\univ}{univ}
\DeclareMathOperator{\Hom}{hom}
\DeclareMathOperator{\ob}{ob}
\DeclareMathOperator{\id}{id}

\newcommand{\Z}{\mathbb{Z}}

\newcommand{\Q}{\mathbb{Q}}
\newcommand{\R}{\mathbb{R}}

\newcommand*{\img}[1]{%
\raisebox{-.3\baselineskip}{%
\includegraphics[ height=\baselineskip, width=\baselineskip, keepaspectratio, ]{figures/#1}%
}}%
 
\newcommand{\myqed}{\pushQED{\qed}\qedhere}
\graphicspath{{figures/}}
\definecolor{annotation}{rgb}{.2,.7,.2}
\definecolor{firstmap}{rgb}{0,0,1}
\definecolor{secondmap}{rgb}{.7,.7,1}
\definecolor{homotopy}{rgb}{1,.1,.1}

\renewcommand*{\backrefalt}[4]{%
\tiny
\ifcase #1 %
No citations.%
\or
Cited on page~#2.%
\else
Cited on pages~#2.%
\fi
}
\newcommand{\refinsec}[2]{\texorpdfstring{\cref{#2}}{#1 \ref*{#2}}}
\newcommand{\modl}{\smash{\raisebox{.3ex}{$\scriptscriptstyle /$}l}}

\let\stdthebibliography\thebibliography
\let\stdendthebibliography\endthebibliography
\renewenvironment*{thebibliography}[1]{%
  \stdthebibliography{CGL{\etalchar{+}}21}}
  {\stdendthebibliography}

\begin{document}
\thispagestyle{empty}
\begin{abstract}
Building on work by Alishahi--Dowlin, 
we extract a new knot invariant $\lambda \geq 0$ from universal Khovanov homology.
While $\lambda$ is a lower bound for the unknotting number, in fact more is true:
$\lambda$ is a lower bound for the proper rational unknotting number (the minimal
number of rational tangle replacements preserving connectivity necessary to relate a knot to the unknot).
Moreover, we show that for all~$n\geq 0$, there exists a knot $K$ with~$\lambda(K) = n$.
Along the way, following Thompson, we compute the Bar-Natan complexes of rational tangles.

\end{abstract}
\maketitle
\section{Introduction}\label{sec:intro}
The most famous geometric application of Khovanov's categorification of the Jones polynomial \cite{KH1} comes from the Rasmussen invariant, a knot concordance homomorphism giving a strong lower bound for the slice genus \cite{ras3}.
The Rasmussen invariant may be read off the grading of the limit of Lee's spectral sequence \cite{Lee}, which starts at Khovanov homology.
Recently, Alishahi and Dowlin \cite{zbMATH07005602} discovered that there is further geometric information contained in Lee's spectral sequence.
Namely, the number of the page on which that sequence collapses is a lower bound for the unknotting number~$u$.
This new lower bound behaves rather differently from the Rasmussen invariant, e.g.~it is not invariant under concordance, and not additive under the connected sum of knots.

In this paper, we use Alishahi and Dowlin's methods on a different variation of Khovanov homology to define a new knot invariant $\lambda$ taking non-negative integer values. This invariant is greater than or equal to all the previously known variations of Alishahi--Dowlin's bound appearing in \cite{zbMATH07005602,zbMATH07178864,CGLLSZ,gujral}. Before giving the definition of~$\lambda$, let us state our two main results.

\begin{theorem}\label{thm:rationalreplacement}
For all knots~$K$, one has $\lambda(K) \leq u_q(K)$.
\end{theorem}
\begin{theorem}\label{thm:lambdahigh}
For every $n\geq 0$ there exists a knot $K$ such that~$\lambda(K) = n$.
\end{theorem}%
Here, the \emph{proper rational unknotting number} $u_q(K)$ is defined as follows.
\begin{dfn}\label{dfn:ratioreplacement}
Two knots $K,K'$ are related by a \emph{rational replacement} if $K'$ may be obtained from $K$
by replacing a rational tangle $T$ in $K$ with another rational tangle~$T'$.
If the arcs of $T$ and $T'$ connect the same tangle end points, we say that
the rational replacement is \emph{proper}. Now, $u_q(K)$ is defined
as the minimal number of proper rational replacements relating $K$ to the unknot.
\end{dfn}
\Cref{fig:introratioreplacement} shows examples of rational replacements,
and an application of \cref{thm:rationalreplacement}.
\begin{figure}[t]
		\centering
		\def\svgwidth{0.7\textwidth}
   {
\begingroup%
  \makeatletter%
  \providecommand\color[2][]{%
    \errmessage{(Inkscape) Color is used for the text in Inkscape, but the package 'color.sty' is not loaded}%
    \renewcommand\color[2][]{}%
  }%
  \providecommand\transparent[1]{%
    \errmessage{(Inkscape) Transparency is used (non-zero) for the text in Inkscape, but the package 'transparent.sty' is not loaded}%
    \renewcommand\transparent[1]{}%
  }%
  \providecommand\rotatebox[2]{#2}%
  \newcommand*\fsize{\dimexpr\f@size pt\relax}%
  \newcommand*\lineheight[1]{\fontsize{\fsize}{#1\fsize}\selectfont}%
  \ifx\svgwidth\undefined%
    \setlength{\unitlength}{1054.78352224bp}%
    \ifx\svgscale\undefined%
      \relax%
    \else%
      \setlength{\unitlength}{\unitlength * \real{\svgscale}}%
    \fi%
  \else%
    \setlength{\unitlength}{\svgwidth}%
  \fi%
  \global\let\svgwidth\undefined%
  \global\let\svgscale\undefined%
  \makeatother%
  \begin{picture}(1,0.68356642)%
    \lineheight{1}%
    \setlength\tabcolsep{0pt}%
    \put(0,0){\includegraphics[width=\unitlength,page=1]{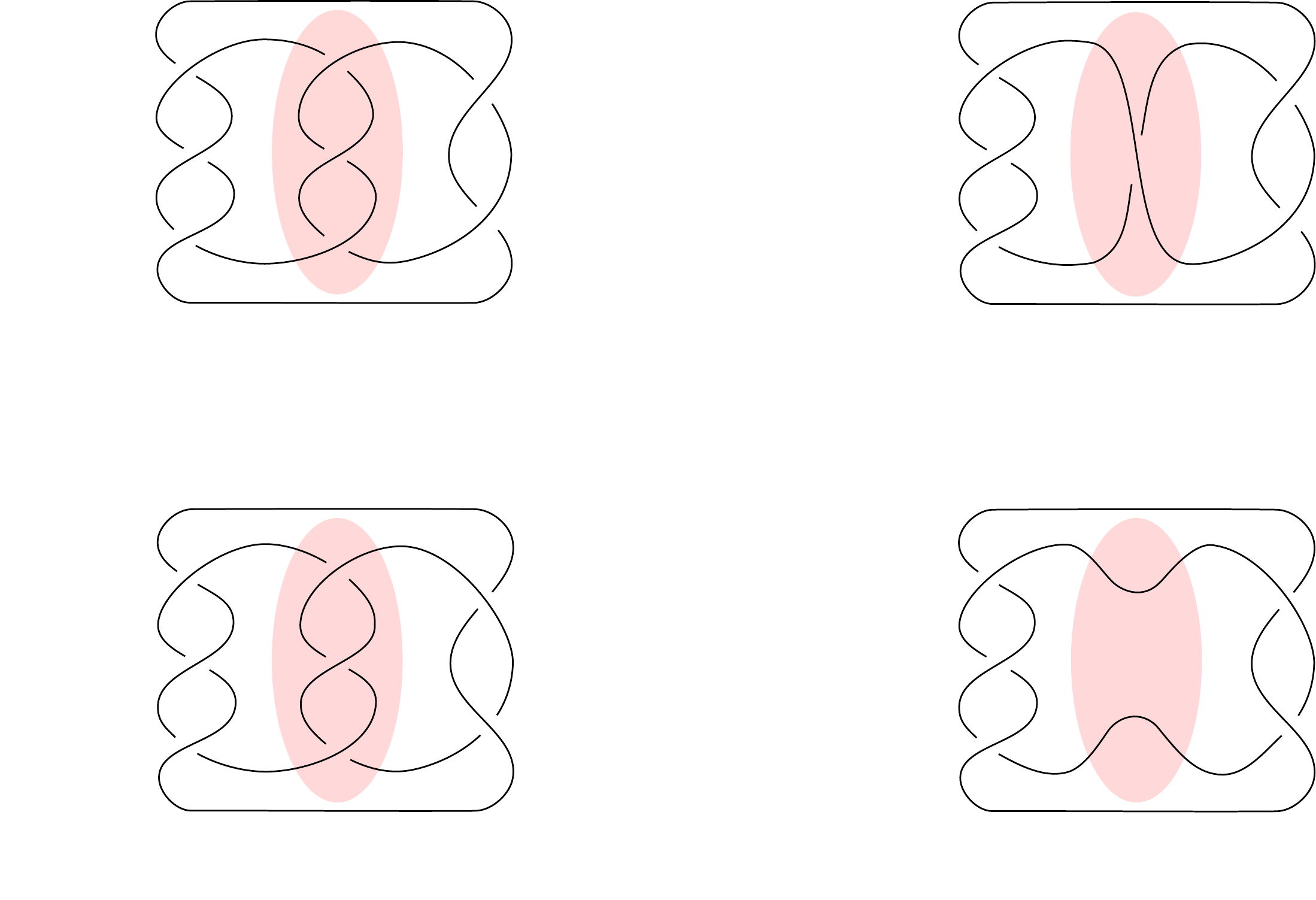}}%
    \put(0.1326647,0.3946474){\makebox(0,0)[lt]{\lineheight{1.25}\smash{\begin{tabular}[t]{l}$P(3,3,2)=8_5$\end{tabular}}}}%
    \put(0.79809713,0.39405239){\makebox(0,0)[lt]{\lineheight{1.25}\smash{\begin{tabular}[t]{l}unknot\end{tabular}}}}%
    \put(0.046378,0.00884351){\makebox(0,0)[lt]{\lineheight{1.25}\smash{\begin{tabular}[t]{l}$P(3,3,-2)=T_{3,4}=8_{19}$\end{tabular}}}}%
    \put(0.79745398,0.00875863){\makebox(0,0)[lt]{\lineheight{1.25}\smash{\begin{tabular}[t]{l}unknot\end{tabular}}}}%
    \put(0,0){\includegraphics[width=\unitlength,page=2]{rationalreplacement.pdf}}%
    \put(-0.00304517,0.55122038){\makebox(0,0)[lt]{\lineheight{1.25}\smash{\begin{tabular}[t]{l}(i)\end{tabular}}}}%
    \put(-0.00144314,0.16736927){\makebox(0,0)[lt]{\lineheight{1.25}\smash{\begin{tabular}[t]{l}(ii)\end{tabular}}}}%
    \put(0,0){\includegraphics[width=\unitlength,page=3]{rationalreplacement.pdf}}%
  \end{picture}%
\endgroup%
}
\caption{In (i), an example of a proper rational replacement ($1/3$ by $-1$ in the language of \cref{dfn:ratioreplacement2}), showing that the $P(3,3,2)$ pretzel knot has proper rational unknotting number~1. In (ii), an example of a non-proper rational replacement ($1/3$ by~$0$), showing that the $P(3,3,-2)$ pretzel knot, which is also the $T_{3,4}$ torus knot, has rational unknotting number 1. Since $\lambda(T_{3,4}) = 2$, it follows from \cref{thm:rationalreplacement} that there is no \emph{proper} rational replacement transforming the $T_{3,4}$ pretzel knot into the unknot, i.e.~$T_{3,4}$ has proper rational unknotting number at least 2 (and in fact equal to 2).}
\label{fig:introratioreplacement}
\end{figure}
Since a crossing change is merely a special case of a proper rational replacement, we find that
$
u_q(K) \leq u(K)
$
holds for all knots~$K$.
So \cref{thm:rationalreplacement} can be seen as a strengthening of the inequality $\mathfrak{u}_X(K) \leq u(K)$,
where $\mathfrak{u}_X(K)$ denotes one of the bounds constructed by Alishahi--Dowlin \cite{zbMATH07005602}. Indeed, we have
\[
\mathfrak{u}_X(K) \leq \lambda(K) \leq u_q(K) \leq u(K).
\]
We shall see that none of these inequalities are equalities,
and in fact, the gaps between the four involved invariants can be arbitrarily large.

In contrast to \cref{thm:lambdahigh}, let us note that currently $\mathfrak{u}_X(K) \leq 3$ holds for all knots $K$ for which $\mathfrak{u}_X$ has been computed.
The essentially only known knot with $\mathfrak{u}_X(K) = 3$ was recently found by Manolescu and Marengon \cite{zbMATH07144520}.
Still, it seems a reasonable conjecture that $\mathfrak{u}_X$ is unbounded,
though the proof of that conjecture might require more complicated knots and methods of computation
than our proof of the unboundedness of $\lambda$ in \cref{thm:lambdahigh}.

In the remainder of the introduction, let us provide details and background about Khovanov homologies, the invariant $\lambda$ and rational unknotting.
\subsection{Rational replacements and rational unknotting}
Rational unknotting has previously been
considered by Lines~\cite{zbMATH00894845} and McCoy~\cite{zbMATH06497273},
and proper rational unknotting in a recent paper by McCoy and Zentner~\cite{newmccoy}.
In those papers, rational unknotting is obstructed via the double branched cover,
relying on the so-called Montesinos trick: if two knots $K$ and $J$ are related by a rational replacement,
then their double branched covers $M_K, M_J$ are related by a surgery.
So one may obstruct the existence of certain rational replacements by obstructing the existence of certain surgeries.

The arguably easiest upshot of that method is the following:
the minimal number of generators of $H_1(M_K; \Z)$ is a lower
bound for the rational unknotting number of~$K$. For example, this implies that the connected sum of $n$ trefoil knots
has (proper and non-proper) rational unknotting number equal to~$n$. On the other hand, one may easily compute that
$\lambda$ of the connected sum of $n\geq 1$ trefoil knots equals~$1$. This may be taken as a first sign
that our lower bound $\lambda$ is quite different from the lower bounds for $u_q$ obtainable from the double branched cover.
In the instance of connected sums of trefoils, the $\lambda$ bound is weaker. However, there are also knots for which $\lambda$ is the strongest lower bound for $u_q$ that we know of. Let us give an explicit example.
\begin{question}
In \cref{ex:t56}, we compute $\lambda(T_{5,6}) = 3$, from which it follows that $u_q(T_{5,6}) \geq 3$.
Is there any other way of showing $u_q(T_{5,6}) \geq 3$?
\end{question}

Note that the gap between the proper rational unknotting number $u_q$ and the (classical) unknotting number $u$ may be arbitrarily high.
For example, $u_q(K) = 1$ clearly holds for all two-bridge knots~$K$;
but $u(K)$ of two-bridge knots can take any value, which can e.g.~be shown using
the signature bound $|\sigma(K)/2| \leq u(K)$.
This also demonstrates that
$|\sigma|/2$ is not a lower bound for the proper rational unknotting number.

In \cref{subsec:tangles}, we will fix our conventions regarding tangles and tangle diagrams.
We will then revisit \cref{dfn:ratioreplacement}, and give a more refined definition of rational replacement in \cref{dfn:ratioreplacement2}. As an aside, let us also remark that in the definition of the proper rational unknotting number~$u_q(K)$,
the proper rational replacements relating $K$ and the unknot are sequential: happening one after the other.
However, by a standard transversality argument (see e.g.~\cite{scharlemann}) one can show that for every knot~$K$,
there exist $u_q(K)$ many simultaneous rational replacements, i.e.~rational replacements taking place in pairwise disjoint balls.
\subsection{A simple universal Khovanov homology}
An algebra $A$ over a ring $R$ equipped with an $(A,A)$-bilinear map $\Delta\colon A \to A\otimes A$ and an $R$-linear map $\epsilon\colon A\to R$, such that $(\epsilon\otimes \text{Id})\circ\Delta= \Id$, is called a \emph{Frobenius algebra},
and the tuple $\mathcal{F} = (R,A,\Delta,\epsilon)$ is called a \emph{Frobenius system}.
We will only consider so-called \emph{rank two Frobenius systems}.
These are Frobenius systems $\mathcal{F}$ with an $X\in A$ such that $A$ is freely generated by $1$ and $X$ as an $R$-module. Moreover, all our Frobenius algebras will be equipped with a filtration or a grading, such that $1$ and $X$ are homogeneous elements of degree $0$ and~$-2$, respectively. We call this the \emph{quantum grading}.
Every such rank two Frobenius system $\mathcal{F}$ yields
a variation of Khovanov homology theory, i.e.~a way to associate with all
diagrams $D$ of a link $L$ a chain complex $C_{\mathcal{F}}(D)$ of $R$-modules, such that $C_{\mathcal{F}}(L)$
is well-defined up to homotopy equivalence~\cite{KH2}.
For links with a marked component, or for knots, there is an action of $A$ on $C_{\mathcal{F}}(D)$,
which is well-defined up to homotopy, so we may consider $C_{\mathcal{F}}(D)$ as a chain complex of $A$-modules, which are free.

Khovanov's original homology theory corresponds to the Frobenius algebra $\Z[X]/(X^2)$ over~$\Z$.
On the other hand, the theory coming from the Frobenius algebra $A_{\univ} = R_{\univ}[X] / (X^2 - hX - t)$ over $R_{\univ} = \Z[h,t]$
is called \emph{universal}, since for all rank two Frobenius algebras~$\mathcal{F}$, the chain complex $C_{\mathcal{F}}(D)$
is determined by $C_{\univ}(D)$~\cite{KH2}.
Recently, Khovanov and Robert defined another theory called $\alpha$-homology,
which is also universal in the sense above \cite{frob2}.
To define~$\lambda$, we will use a third universal theory, which we call $\Z[G]$-homology.%
\footnote{Here $G$ is simply a symbolic variable, and so $\Z[G]$ is the one-variable polynomial ring over the integers (not some group ring).}
The universality of this theory is due to Naot \cite{zbMATH05118580,naotphd}.
This $\Z[G]$-theory associates with a diagram $D$ of a knot $K$ the reduced Khovanov chain complex coming from the Frobenius algebra $R[X] / (X^2 + GX)$ with~$R = \Z[G]$. We denote this chain complex by $\llbracket D\rrbracket$ (well-defined up to isomorphism) or $\llbracket K\rrbracket$ (well-defined up to homotopy equivalence).
Our reason to use $\Z[G]$-homology is that it is the simplest of the three mentioned universal theories,
in so far as the ground ring is a polynomial ring in only one, instead of two variables.
Let us explicitly state how $\Z[G]$-homology determines $\mathcal{F}_{\univ}$-homology
(this is implicit in the work of Naot \cite{zbMATH05118580,naotphd}).
\begin{restatable}{theorem}{thmequiv}\label{thm:equiv}
Endow $A_{\univ} = \Z[h,t][X] / (X^2 - hX - t)$
with the structure of a $\Z[G]$-module by letting $G$ act as $2X - h$.
Then for every knot~$K$,
\[
C_{\univ}(K) \simeq \llbracket K\rrbracket \otimes_{\Z[G]} A_{\univ}\{1\}.
\]
\end{restatable}
Here, $\{\,\cdot\,\}$ denotes a shift in quantum degree.
\begin{corollary}\label{cor:equiv}
For every knot~$K$, $C_{\univ}(K)$ is homotopy equivalent to a chain complex of free shifted $A_{\univ}$-modules, with differentials consisting only of integer multiples of powers of~$2X - h$.\qed
\end{corollary}
\Cref{thm:equiv} and \Cref{cor:equiv} can be understood to say that $\Z[G]$-homology encodes the same amount of information present in $\mathcal{F}_{\univ}$- and $\alpha$-homology in a more compact way.
In particular, the original reduced Khovanov homology over~$\Z$ of $K$ as defined in \cite{KH3} may be obtained from $\llbracket K\rrbracket$ simply by setting~$G = 0$, i.e.~by tensoring with $\Z[G]/(G) \cong \Z$.
The original unreduced Khovanov homology over~$\Z$ is also determined by $\Z[G]$-homology, see \cref{cor:zg2kh}.

Let us give three examples of $\Z[G]$-complexes $\llbracket\,\cdot\,\rrbracket$ of knots.
For the unknot~$U$, $\llbracket U\rrbracket$ is simply homotopy equivalent to one copy of $\Z[G]$ supported in homological degree 0. For the trefoil~$T_{2,3}$, we have a homotopy equivalence
\[
\llbracket T_{2,3}\rrbracket\ \simeq
\begin{tikzcd}
_0\Z[G]\{2\} \ar[r] & 0 \ar[r] & \Z[G]\{6\} \ar[r,"G"] & \Z[G]\{8\},
\end{tikzcd}
\]
where the subscript to the left denotes homological degree.
Finally, for the $T_{3,4}$ torus knot, we have
\[
\llbracket T_{3,4}\rrbracket\ \simeq
\begin{tikzcd}[column sep=small]
_0\Z[G]\{6\} \ar[r] & 0 \ar[r] & \Z[G]\{10\}\ar[r,"G"] & \Z[G]\{12\}
\ar[r,"0"] & \Z[G]\{12\} \ar[r,"G^2"] & \Z[G]\{16\}.
\end{tikzcd}
\]
 Many more examples will be given in \cref{sec:lambdaintro,sec:lambdacalc}.

We will give an introduction to Khovanov homologies and the proof of \Cref{thm:equiv} in \cref{sec:zg}.
The idea of the proof is to show that $\Z[G]$-homology of a knot $K$ determines the \emph{Bar-Natan complex} \cite{BN1} of the 2-ended tangle obtained by cutting $K$ open at some point. The Bar-Natan complex in turn is known to be universal and determine $\mathcal{F}_{\univ}$-homology. The Bar-Natan complex is discussed in \Cref{sec:zg}, and also in \Cref{subsec:bn} below.

\subsection{The definition of $\lambda$}
Let us now introduce this paper's protagonist:
\begin{restatable}{dfn}{dfnlambda}\label{dfn:lambda}
For a knot~$K$, let $\lambda(K)$ be the minimal integer $k\geq 0$ such that there exist
ungraded chain maps (i.e.~chain maps that do not need to respect the homological or the quantum degree,
cf.~\cref{dfn:ungraded})
\[
\begin{tikzcd}[ampersand replacement=\&]
\llbracket K\rrbracket \ar[r,"f",bend left=10] \&
\llbracket U\rrbracket \ar[l,"g",bend left=10]
\end{tikzcd}
\]
and homotopies
$
g\circ f \simeq G^k\cdot \text{id}_{\llbracket K\rrbracket},
f\circ g \simeq G^k\cdot \text{id}_{\llbracket U\rrbracket}$.
\end{restatable}
It is not obvious that for a given knot, $f,g,k$ as in \cref{dfn:lambda} exist at all.
So, for the time being, we scrupulously set $\lambda(K) = \infty$ if they do not. It will then be a consequence
of \cref{thm:rationalreplacement} that this case does not occur, since $\lambda(K) \leq u_q(K) \leq u(K) < \infty$.
To get acquainted with calculating~$\lambda$, the reader is invited
to convince themself that $\lambda(U) = 0, \lambda(T_{2,3}) = 1, \lambda(T_{3,4}) = 2$.

Our \cref{dfn:lambda} is based on the work of Alishahi and Dowlin, who use analogous maps $f$, $g$ in the proof that their invariant $\mathfrak{u}_X$ is a lower bound for the unknotting number~$u$. The invariant $\mathfrak{u}_X(K)$ is defined as the maximal $X$-torsion order of the homology of $K$ coming from the Frobenius algebra $\mathcal{F}_{\text{Lee}} = \Q[X,t] / (X^2 - t)$ over~$\Q[t]$,
i.e.~the minimal $n$ such that $X^n H_{\text{Lee}}(K)$ is torsion-free.
At first glance, the definition of $\mathfrak{u}_X$ and $\lambda$ appear to be rather different; but on a closer look, one finds that $\mathfrak{u}_X(K) = \lambda_X(K)$, where $\lambda_X(K)$ is the minimal $k\geq 0$ such that there exist ungraded chain maps
\[
\begin{tikzcd}[ampersand replacement=\&]
C_{\text{Lee}}(K) \ar[r,"f",bend left=10] \&
C_{\text{Lee}}(U) \ar[l,"g",bend left=10]
\end{tikzcd}
\]
and homotopies
$
g\circ f \simeq (2X)^k\cdot \text{id}_{C_{\text{Lee}(K)}},
f\circ g \simeq (2X)^k\cdot \text{id}_{C_{\text{Lee}(U)}}$.

In this sense, $\lambda$ is a direct generalization of~$\mathfrak{u}_X$, obtained from the reduced homology coming from the Frobenius algebra $\mathcal{F}_{\Z[G]}$ instead of from the unreduced homology coming from~$\mathcal{F}_{\text{Lee}}$. But why don't we instead of $\lambda$ consider~$\mathfrak{u}_G(K)$, defined (see \cref{def:maxGtorsion}) as the maximal $G$-torsion order of $\Z[G]$-homology of~$K$? There are two reasons. Firstly, $\lambda$ is not equal to $\mathfrak{u}_G$ for all knots; the proof of the equality $\lambda_X = \mathfrak{u}_X$ does not carry over from $\Q[X,t] / (X^2 - t)$ to~$\Z[G]$, because it relies on $\Q[X,t] / (X^2 - t) \cong \Q[X]$ being a PID, which $\Z[G]$ is not. In fact, $\lambda(K) \geq \mathfrak{u}_G(K)$ holds for all knots~$K$.
Secondly, $\mathfrak{u}_G$ displays some unorthodox behavior; for example, the value of $\mathfrak{u}_G(-K)$ is not determined by the value of~$\mathfrak{u}_G(K)$, where $-K$ denotes the mirror image of~$K$. Again, the ring $\Z[G]$ not being a PID is to blame for this.

More details about the invariant $\lambda$ and related invariants can be found in \cref{subsec:furtherlambda} below, and in \cref{sec:lambdaintro}.
\subsection{The Bar-Natan complex of rational tangles}\label{subsec:bn}
The proof of \cref{thm:rationalreplacement}, given at the end of \cref{sec:rational},
proceeds by relating the Bar-Natan complexes of different rational tangles.
So, first we need to compute the Bar-Natan complexes of rational tangles.
For this, we heavily rely on Thompson's computation \cite{thompson} of \emph{dotted} Bar-Natan complexes of rational tangles.
Thompson's proof adapts mutatis mutandis to the
(more general) version of Bar-Natan complexes that we are using.
We also use Kotelskiy--Watson--Zibrowius' theorem \cite[Theorem~1.1]{KWZ}
that Bar-Natan's category of 4-ended tangles and cobordisms
is equivalent to a category coming from a quiver with two vertices and four edges,
which yields a quite simple calculus for chain complexes of 4-ended tangles.

As a result, we obtain a recursive algorithm that takes as input a rational tangle $T$ corresponding to a rational number $p/q\in \Q\cup\{\infty\}$, and returns as output a chain complex in the homotopy class of the Bar-Natan complex of~$T$. That chain complex is relatively small, consisting of $|p| + |q|$ many objects and $|p| + |q| - 1$ non-trivial morphisms. We refer the reader to \cref{sec:rational} and \cref{thm:thompson} for details.
\subsection{Further properties and generalizations of $\lambda$}\label{subsec:furtherlambda}
From Kronheimer--Mrowka's result that Khovanov homology detects the unknot \cite{km3}, it follows that $\lambda$ does, too.
\begin{restatable}{prop}{propdetects}\label{prop:lambdadetectsunknot}
The $\lambda$-invariant detects the unknot, i.e.~$\lambda(K) = 0$ holds if and only if $K$ is the unknot.
\end{restatable}
We will see that the value of $\lambda(K\# J)$ for the connected sum $K \# J$ is not determined by the values of $\lambda(K)$ and~$\lambda(J)$. However, we can say the following.
\begin{prop}\label{prop:lambdabasicproperties}
\begin{enumerate}[label=(\roman*)]
\item $\lambda(K\# J) \leq \lambda(K) + \lambda(J)$ for all knots~$K, J$.
\item $\lambda$ does not change under taking mirror images, or orientation reversal.
\end{enumerate}
\end{prop}
Let us call a knot $K$ \emph{thin} if 
its reduced integral Khovanov homology consists of free modules supported in a single $\delta$-degree
(see \cref{subsec:thin} for the definition of $\delta$ and further details).
\begin{restatable}{prop}{propthin}\label{prop:lambdaofthin}
For all non-trivial thin knots~$K$, we have~$\lambda(K) = 1$.
\end{restatable}
In particular $\lambda(K) = 1$ holds for all non-trivial quasi-alternating knots, since those knots are thin in the above sense \cite{manozs}. This leads to applications such as the following.
\begin{example}
In \cref{ex:t56}, we will compute $\lambda(T_{5,6}) = 3$.
It follows that there is no proper rational replacement relating $T_{5,6}$
to a quasi-alternating knot (compare this to~\cite[Example~10]{CGLLSZ}).
\end{example}
In the definition of~$\lambda$, replacing $U$ by an arbitrary second knot~$J$
yields the definition of a function $\lambda(K,J) \geq 0$ that is symmetric and obeys the triangle inequality: $\lambda(\,\cdot\,,\,\cdot\,)$ is a pseudometric on the set of isotopy classes of knots.
In fact, we can even further extend the definition of~$\lambda$, and define $\lambda$ for pairs of tangles.
This leads to a pseudometric on the set of equivalence classes of tangles in a fixed ball, with fixed base point, connectivity, and number of components, see \cref{prop:lambdametric}.
\Cref{sec:lambdaintro} provides details, and also the proofs of the above propositions.
\subsection{A comparison of $\lambda$ with previously known invariants}
Alishahi and Dowlin's article \cite{zbMATH07005602} appeared at the same time as an article by 
Alishahi \cite{zbMATH07178864}, in which, similarly to~$\mathfrak{u}_X$, a lower bound $\mathfrak{u}_h$
for the unknotting number was constructed using the Frobenius algebra $\mathbb{F}_2[X,h]/(X^2 + hX)$ over~$\mathbb{F}_2[h]$.
Then, further papers followed:
Caprau--González--Lee--Lowrance--Sazdanović--Zhang
generalized Alishahi and Dowlin's work for $\Q$ to the fields $\mathbb{F}_p$ for all odd primes~$p$ \cite{CGLLSZ}.
Gujral \cite{gujral}, using $\alpha$-homology, defined an invariant $\nu$ which can be seen to equal our invariant~$\mathfrak{u}_G$, and showed that it provides a lower bound for the ribbon distance between knots;
this was a generalization of earlier work by Sarkar \cite{zbMATH07195384}.
Here, the \emph{ribbon distance} between two smoothly concordant knots $K$ and $J$ is the minimal $k$ such that there is a sequence $K = K_1, \ldots, K_n = J$ of knots, such that each consecutive $K_i, K_{i+1}$ are related by a ribbon concordance in either direction with at most $k$ saddles. This leads to the following question (see \cref{subsec:furtherlambda} or \cref{dfn:lambdageneral} for the definition of~$\lambda(K,J)$).
\begin{question}
Is $\lambda(K, J)$ less than or equal to the ribbon distance of $K$ and $J$ for all pairs of knots~$K$,~$J$?
\end{question}
The previously defined invariants mentioned above will be discussed in more detail in \cref{subsec:torsionorders}.
By construction, $\lambda$ is greater than or equal to all of them (the price to pay is that $\lambda$ is generally slightly harder to compute).
Let us explicitly emphasize that this observation combined with \cref{thm:rationalreplacement} implies that
all of those previously defined invariants are also lower bounds for the proper rational unknotting number.

The construction principle that underlies $\lambda$ and the other above-mentioned invariants from Khovanov homology goes back to work by Alishahi and Eftekhary, who applied it to knot Floer homology \cite{arXiv1610.07122v1,zbMATH07305772}. They obtained a lower bound for the unknotting number, as well as lower bounds for other quantities, such as the minimal number of negative-to-positive crossing changes in any unknotting sequence of a knot. Further knot Floer torsion order invariants were defined by Juhász, Miller and Zemke \cite{zbMATH07311838}, who find lower bounds for even more topological quantities, such as the bridge index, the band-unlinking number, etc. Still, the following question remains open.
\begin{question}
Is one of the knot Floer torsion order invariants a lower bound for the proper rational unknotting number?%
\footnote{After the article at hand first had appeared as a preprint, Eftekhary answered this
question in the positive~\cite{arXiv2203.09319}.}
\end{question}

\subsection{Computations}
Computations of $\Z[G]$-homology are theoretically possible by hand using Bar-Natan's divide-and-conquer approach \cite{BN2}. Nevertheless, to proceed efficiently,
we use the program \texttt{khoca} \cite{khoca} (originally written for \cite{lewarklobb})
to compute $\Z[G]$-complexes of knots.\footnote{Note that \texttt{javakh} \cite{javakh}, while very fast, apparently only calculates Morrison's `universal homology', which corresponds to $\Q[G]$-homology. Currently, the program \texttt{kht++} \cite{khtpp} also only simplifies complexes over fields, not over the integers.}
As input, \texttt{khoca} accepts diagrams of a knot~$K$, e.g.~in PD notation. From \texttt{khoca}'s output, one may read off a chain complex of $\Z[G]$-modules in the homotopy class of~$\llbracket K\rrbracket$.
For further simplification, \texttt{khoca}'s output may be fed into the new
program \texttt{homca} \cite{homca}, which attempts to decompose $\llbracket K\rrbracket$ as a
direct sum of simpler chain complexes.
From these simpler pieces, one may typically calculate $\lambda$ by hand.
See \cref{ex:t56} for an application of this strategy to the $T_{5,6}$ torus knot.
For small knots, we find the following.
\begin{restatable}{prop}{lambdasmallknots} \label{prop:lambdasmallknots}
	For all knots up to $10$ crossings we have~$\lambda = 1$, except for the knots $8_{19}$, $10_{124}$, $10_{128}$, $10_{139}$, $10_{152}$, $10_{154}$, $10_{161}$, where~$\lambda = 2$.
\end{restatable}
This proposition and further calculations for small knots are discussed in \cref{subsec:lambdasmall}.
\subsection{Structure of the paper}
In \cref{sec:zg}, we introduce $\Z[G]$-homology and the other variations of Khovanov homology needed in this article, and prove \cref{thm:equiv} (parts of that proof have been moved to \cref{sec:appendix}).
\cref{sec:lambdaintro} contains properties and generalizations of the $\lambda$-invariant
and other closely related invariants (in particular, the proofs of the propositions mentioned in \cref{subsec:furtherlambda} above).
In \cref{sec:lambdacalc}, we calculate the $\lambda$ invariant for various families of knots, thereby proving \cref{thm:lambdahigh}.
\cref{sec:rational} includes a discussion of the Bar-Natan complexes of rational tangles, and the proof of \cref{thm:rationalreplacement}.
\subsection*{Acknowledgments}
The authors thank Claudius Zibrowius for pointing us towards the papers of Naot and Thompson,
which were essential for this text.
We thank all participants of our \emph{Khovanov homology reading course},
Felix Eberhart, Marco Moraschini, Lars Munser, Jos\'e Pedro Quintanilha, and Paula Tru\"ol
for the exciting times.
The second author thanks Raphael Zentner for stimulating discussions about rational unknotting,
and Paul Turner for inspiring conversations about integral Lee theory.
Thanks to Dirk Schütz for sharing his calculations of Rasmussen invariants.
The authors thank Eaman Eftekhary, Duncan McCoy, Dirk Schütz and Claudius Zibrowius for comments on drafts of this article.
Thanks to Daniel Schäppi for discussions that lead to \cref{rmk:nonuniq}.
The first and second author gratefully acknowledge the support by the Emmy Noether Programme of the DFG, project no.~412851057.
The authors thank the anonymous referee for their suggestions.

\section{$\Z[G]$-homology and other variations of Khovanov homology}\label{sec:zg}
The aim of this section is to lay the foundations of $\mathbb{Z}[G]$-homology and discuss how it fits in the general picture of Khovanov homology. We will in particular show that the $\mathbb{Z}[G]$-theory is equivalent to Khovanov's universal theory described in \cite{KH2}, which establishes $\mathbb{Z}[G]$-homology as a simpler alternative of equal strength. Note that the $\mathbb{Z}[G]$-theory is not new; it was previously described by Naot \cite{zbMATH05118580,naotphd}\footnote{Naot's notation is $\Z[H]$ instead of~$\Z[G]$. $H$ is for `handle', while $G$ is for `genus'...}. Other topics of discussion in this section include Frobenius systems and topological quantum field theories (TQFTs for short), as well as homology of reduced type. Throughout this section we assume familiarity with Bar-Natan's theory for tangles and cobordisms \cite{BN1}.

\subsection{Tangles and tangle diagrams}\label{subsec:tangles}
Let us start by giving precise definitions of tangles and tangle diagrams, and discuss their relationship.
\begin{dfn}
A \emph{tangle} $T$ is a proper smooth 1-submanifold of a closed oriented 3-ball~$B$.
The points in $T\cap \partial B$ are called \emph{end points} of~$T$. Every tangle is \emph{$2n$-ended},
i.e.~has $2n$ end points, for some~$n\geq 0$.
Throughout this text we will consider oriented tangles, unless explicitly mentioned otherwise.
Two tangles in the same 3-ball $B$
with the same set of $2n$ end points in $\partial B$
are \emph{equivalent} if there is an orientation-preserving homeomorphism of~$B$, fixing the boundary pointwise, mapping one tangle to the other, and preserving the orientation of the tangles if they are oriented.
\end{dfn}
Note that a $2n$-ended tangle consists of $n$ arcs and a finite number of circles.
\begin{dfn}
The \emph{connectivity} of a $2n$-ended tangle $T$ with arcs $\alpha_1,\ldots,\alpha_n\subset T$
is the set $\{\partial \alpha_1, \ldots,\partial\alpha_n\}$.
\end{dfn}
For example, for tangles in a fixed ball $B$ with a fixed set of $0,2,4,6,\ldots$ end points on~$\partial B$,
there are $1,1,3,15,\ldots$ possible connectivities\footnote{With $2n>0$ end points, there are $(2n-1)!!$ connectivities, where $!!$ denotes the double factorial.}.
Bleiler \cite{zbMATH03960559} called this notion `string attachments', but for its brevity we prefer the term
connectivity, which is also used in \cite{2004.10837,KWZ2}.
Note that proper rational replacements (see \cref{dfn:ratioreplacement}) are precisely those rational replacements that preserve connectivity.
\begin{dfn}
A \emph{tangle diagram} $D$ is an immersed proper smooth 1-submanifold of a closed 2-disk~$E$,
such that all self-intersections are transverse double points, endowed with over-under information at each such double point.
Similarly as for tangles, every tangle diagram has an even number of \emph{end points}~$D\cap \partial E$,
and we call two tangle diagrams in the same disk $E$ with the same set of end points \emph{equivalent} if there
is an orientation-preserving homeomorphism of~$E$, fixing the boundary pointwise, mapping one tangle diagram to the other, preserving over-under information, and preserving orientation if the tangle diagram is oriented.
\end{dfn}
\begin{remark}
All tangles in a ball that is embedded into the 3-sphere arise as intersections of that ball with a link that is transverse to the ball's boundary sphere. Similarly, all tangle diagrams in a disk embedded into the plane arise as intersections of that disk with link diagrams that are transverse to the disk's boundary circle.
\end{remark}
How are tangle diagrams with $2n$ end points in two \emph{different} disks $E_1, E_2$ related?
Let us consider homeomorphisms $\varphi\colon E_1\to E_2$ that are orientation-preserving and end point-preserving, i.e.~sending end points to end points. If two such homeomorphisms $\varphi, \varphi'$ are end point-preservingly isotopic (i.e.~isotopic along end point-preserving maps), then they send a tangle diagram $D\subset E_1$ to two equivalent tangle diagrams $\varphi(D), \varphi'(D) \subset E_2$.
By Alexander's trick, the isotopy class of a homeomorphism $E_1\to E_2$ is determined by the isotopy class of its restriction to the boundary. So, there are $2n$ end point-preserving isotopy classes of homeomorphisms~$E_1\to E_2$, each giving a way to identify equivalence classes of $2n$-ended tangle diagrams in two different disks.
If one considers tangle diagrams with \emph{base points}, i.e.~one distinguished end point,
then requiring that $\varphi$ sends base point to base point determines $\varphi$ uniquely up to isotopy.

The situation is more complicated, however, for tangles in different balls~$B_1, B_2$.
As before, the end point-preserving isotopy classes of homeomorphisms $\varphi\colon B_1\to B_2$ are determined by the end point-preserving isotopy classes of homeomorphisms $\partial B_1\to \partial B_2$.
Those are in (non-canonical) one-to-one correspondence with the elements of the mapping class group of the $2n$-punctured sphere (see e.g.~\cite{fm} for an introduction to mapping class groups). For~$2n \geq 4$, there are non-trivial mapping classes fixing some boundary point; so, in contrast to the situation for tangle diagrams, base pointed tangles with four or more end points in different balls cannot be identified in a canonical fashion.

This also has consequences for the tangle diagrams of a tangle, which one may obtain by projection.
\begin{dfn}\label{dfn:diagramoftangle}
Let $T\subset B$ be a $2n$-ended tangle. Let $\varphi$ be a homeomorphism
from~$B$ to the unit ball $B_0 \subset \mathbb{R}^3$,
mapping the end points of $T$ on $\partial B$ to $\{(\cos(k\pi/n), \sin(k\pi/n),0) \mid 0\leq k < 2n\}$.
If the projection $\mathbb{R}^3\to\mathbb{R}^2, (x,y,z)\mapsto (x,y)$ sends $\varphi(T)$ to a tangle diagram $D_T$ in the unit disk in the $xy$-plane, we call $D_T$ a \emph{tangle diagram of $T$}.
\end{dfn}
A fixed homeomorphism $\varphi$ sends equivalent tangles $T, T'$ to tangle diagrams $D_T, D_{T'}$ related by Reidemeister moves and tangle diagram equivalence. But if one does not specify~$\varphi$, this is no longer true, and the equivalence class of $T$ is no longer determined by~$D_T$. We will revisit this in \cref{subsec:lambdatangles}.

For now, let us focus on the special case $n = 1$ (note that for~$n = 0$, tangles without end points are just links in a ball, and of no interest beyond that). Two-ended tangles and tangle diagrams will be the objects we use to construct $\Z[G]$-homology in this section.
We have the following one-to-one correspondences:
\begin{equation}\label{eq:2endlinks}
	\parbox{.27\textwidth}{isotopy classes of base pointed oriented links $L\subset S^3$} \quad\overset{1:1}\longleftrightarrow \quad \parbox{.40\textwidth}{equivalence classes of $2$-ended tangles $T$ in a fixed ball with fixed end points~$x,y$, with the arc of $T$ oriented from $x$ to~$y$.}
\end{equation}
\begin{equation}\label{eq:2enddiagrams}
	\parbox{.27\textwidth}{isotopy classes of base pointed oriented link diagrams} \quad\overset{1:1}\longleftrightarrow \quad \parbox{.40\textwidth}{equivalence classes of $2$-ended tangle diagrams $D$ in a fixed disk with fixed end points~$x,y$, with the arc of $D$ oriented from $x$ to~$y$.}
\end{equation}\medskip

Let us describe how to get from $L$ to $T$ and vice versa in \cref{eq:2endlinks}.
The complement of an open ball neighborhood of the base point of $L$ is a closed ball $B$ containing a 2-ended tangle~$B\cap L$. There are two non-isotopic homeomorphisms sending end points to end points between $B$ and another fixed ball; these two correspond to the two elements of the mapping class group of the twice-punctured sphere. By specifying the orientation of the arc on the right-hand side of \cref{eq:2endlinks}, we eliminate this ambiguity.
In the other direction, a fixed ball containing a 2-ended tangle $T$ may be embedded into~$S^3$, and the two end points of $T$ may be joined by an arc outside of the embedded ball, producing a link~$L\subset S^3$.
The correspondence \cref{eq:2enddiagrams} can be shown in a similar way.

So, from now on, we will work with the notions of base pointed link (diagrams) and 2-ended tangle (diagrams) interchangeably. Moreover, we may assign tangle diagrams to given 2-ended tangles, without the ambiguities arising for tangles with more end points.

Note that this setup is also well-suited to describe the connected sum operation:
given two base pointed oriented links~$L$,~$L'$,
glue their corresponding 2-ended tangles together in a way that is compatible with the orientation of the arcs. This produces another oriented 2-ended tangle, which corresponds
to the \emph{connected sum}~$L \# L'$.
\subsection{Categorical Framework}
Let us now fix some notions from category theory, following Bar-Natan \cite{BN1}. Whenever we refer to ``category'' in this paper, we assume that the category is small, i.e.~its classes of objects and morphisms are actually sets. A category $\mathcal{C}$ is called \emph{pre-additive} if it has the following additional structure: for any two given objects $\mathcal{O}, \mathcal{O}^\prime \in \ob(\mathcal{C})$, the set $\hom_\mathcal{C}(\mathcal{O}, \mathcal{O}^\prime)$ is an Abelian group and the composition of morphisms is bilinear. An arbitrary category $\mathcal{C}$ can be made pre-additive by allowing formal $\Z$-linear combinations in every set of morphisms $\hom_\mathcal{C}(\mathcal{O}, \mathcal{O}^\prime)$ and by extending composition of morphisms in the natural bilinear way. Given a pre-additive category~$\mathcal{C}$, we denote by $\Mat(\mathcal{C})$ the additive closure and by $\Kom(\mathcal{C})$ the category of complexes over $\mathcal{C}$ (in the sense of Bar-Natan \cite{BN1}).
We call a pre-additive category $\mathcal{C}$ \emph{graded} if it carries the following additional structure:
\begin{enumerate}
	\item For any two objects $\mathcal{O}, \mathcal{O}^\prime \in \ob(\mathcal{C})$, $\hom_\mathcal{C}(\mathcal{O}, \mathcal{O}^\prime)$ forms a graded Abelian group such that composition of morphisms respects the grading and such that all identity maps are of degree zero.
	\item There is a $\Z$-action $(m, \mathcal{O}) \mapsto \mathcal{O}\{m\}$ on the objects of~$\mathcal{C}$, called \emph{grading shift by $m$}. Note that this action changes gradings of morphisms, but not the set of morphisms itself (i.e.\ $\hom_\mathcal{C}(\mathcal{O}, \mathcal{O}^\prime) = \hom_\mathcal{C}(\mathcal{O}\{m\}, \mathcal{O}^\prime\{n\})$, but if $f \in \hom_\mathcal{C}(\mathcal{O}, \mathcal{O}^\prime)$ has degree~$d$, then $f \in \hom_\mathcal{C}(\mathcal{O}\{m\}, \mathcal{O}^\prime\{n\})$ has degree~$d-m+n$).
\end{enumerate}
If a pre-additive category $\mathcal{C}$ only satisfies the first point above, we can extend the category to have the second point as well: simply add ``artificial'' objects $\mathcal{O}\{m\}$ for any $\mathcal{O} \in \ob(\mathcal{C})$ and $m \in \Z$ and define the $\Z$-action in the obvious way. Let us call this construction the \emph{graded closure}.
If $\mathcal{C}$ is a graded category, the additive closure $\Mat(\mathcal{C})$ and the category of complexes $\Kom(\mathcal{C})$ can be considered as graded categories as well.

Next, let us describe the categories we are going to work with to construct $\Z[G]$-homology.

\begin{itemize}
	\item $\Cob^3(2n)$: The objects of $\Cob^3(2n)$ are crossingless unoriented $2n$-ended tangle diagrams $D_T$ (possibly empty if~$n = 0$) in some disk that is fixed throughout, together with an enumeration of every circle appearing in $D_T$ (see \cref{fig:cobobject} below).
		\begin{figure}[b]
			\centering
			\def\svgwidth{0.5\textwidth}
\begingroup%
  \makeatletter%
  \providecommand\color[2][]{%
    \errmessage{(Inkscape) Color is used for the text in Inkscape, but the package 'color.sty' is not loaded}%
    \renewcommand\color[2][]{}%
  }%
  \providecommand\transparent[1]{%
    \errmessage{(Inkscape) Transparency is used (non-zero) for the text in Inkscape, but the package 'transparent.sty' is not loaded}%
    \renewcommand\transparent[1]{}%
  }%
  \providecommand\rotatebox[2]{#2}%
  \newcommand*\fsize{\dimexpr\f@size pt\relax}%
  \newcommand*\lineheight[1]{\fontsize{\fsize}{#1\fsize}\selectfont}%
  \ifx\svgwidth\undefined%
    \setlength{\unitlength}{225.73773193bp}%
    \ifx\svgscale\undefined%
      \relax%
    \else%
      \setlength{\unitlength}{\unitlength * \real{\svgscale}}%
    \fi%
  \else%
    \setlength{\unitlength}{\svgwidth}%
  \fi%
  \global\let\svgwidth\undefined%
  \global\let\svgscale\undefined%
  \makeatother%
  \begin{picture}(1,0.33241687)%
    \lineheight{1}%
    \setlength\tabcolsep{0pt}%
    \put(0,0){\includegraphics[width=\unitlength,page=1]{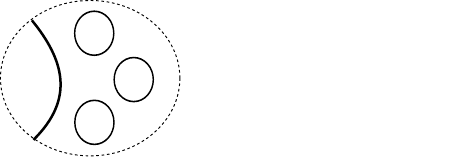}}%
    \put(0.25756274,0.25537718){\color[rgb]{0,0,0}\makebox(0,0)[lt]{\lineheight{1.25}\smash{\begin{tabular}[t]{l}$2$\end{tabular}}}}%
    \put(0.25552277,0.05728002){\color[rgb]{0,0,0}\makebox(0,0)[lt]{\lineheight{1.25}\smash{\begin{tabular}[t]{l}$1$\end{tabular}}}}%
    \put(0.34039116,0.15478367){\color[rgb]{0,0,0}\makebox(0,0)[lt]{\lineheight{1.25}\smash{\begin{tabular}[t]{l}$3$\end{tabular}}}}%
    \put(0,0){\includegraphics[width=\unitlength,page=2]{cobobject.pdf}}%
    \put(0.8223784,0.25537622){\color[rgb]{0,0,0}\makebox(0,0)[lt]{\lineheight{1.25}\smash{\begin{tabular}[t]{l}$1$\end{tabular}}}}%
    \put(0.82033842,0.05727903){\color[rgb]{0,0,0}\makebox(0,0)[lt]{\lineheight{1.25}\smash{\begin{tabular}[t]{l}$3$\end{tabular}}}}%
    \put(0.90520685,0.15478272){\color[rgb]{0,0,0}\makebox(0,0)[lt]{\lineheight{1.25}\smash{\begin{tabular}[t]{l}$2$\end{tabular}}}}%
  \end{picture}%
\endgroup%

			\caption{Two non-equal but isomorphic objects in~$\Cob^3(2n)$.}
			\label{fig:cobobject}
		\end{figure}
Morphisms are $2$-dimensional cobordisms (orientable, possibly disconnected surfaces) between two such diagrams~$D_T$,~$D_{T^\prime}$, considered up to boundary-fixing isotopy. The identity is given by the product cobordism, and composition is done by concatenating cobordisms. We turn $\Cob^3(2n)$ into a pre-additive category as described above.
		For better readability we will frequently keep the enumeration implicit and omit it in our diagrams.
		\begin{figure}[b]
			\centering
			\def\svgwidth{0.7\textwidth}
			{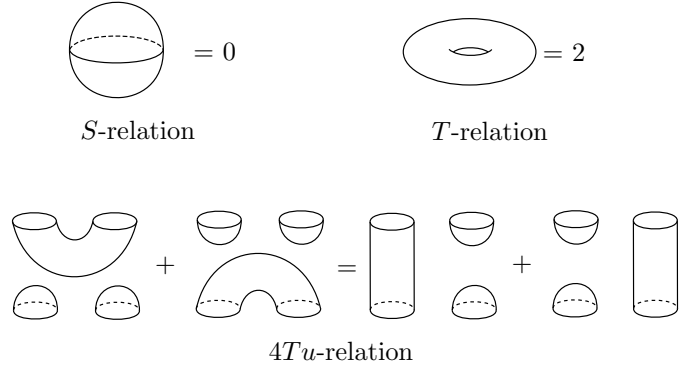}%
			\caption{The defining relations for~$\Cob^3_{\modl}(2n)$.}
			\label{fig:relations}
		\end{figure}
	\item $\Cob^3_{\modl}(2n)$: By modding out the \emph{local relations}~$S$,~$T$, and $4Tu$ on the morphisms of~$\Cob^3(2n)$, we obtain the category $\Cob^3_{\modl}(2n)$ (see \cref{fig:relations}).
	\item $\mathcal{E}$: There is only one object in~$\mathcal{E}$, namely the crossingless diagram $D_{T_0}$ of the trivial $2$-ended tangle~$T_0$. Morphisms are connected cobordisms up to boundary-fixing isotopy. The identity is given by the product cobordism (a ``curtain'', see \cref{def:curtain}), and composition is done by concatenating cobordisms. We turn $\mathcal{E}$ into a pre-additive category as described above.
	\item $\mathcal{M}_R$: Let $R$ be a graded ring. We write $R\{m\}$ for $R$ with grading shifted by~$m \in \Z$, i.e.\ $R\{m\}_n = R_{n-m}$. Let $\mathcal{M}_R$ be the category whose objects are graded $R$-modules isomorphic to a direct sum $\oplus_{i=1}^n R\{m_i\}$, and whose morphisms are graded homomorphisms between $R$-modules. We turn $\mathcal{M}_{R}$ into a graded category by introducing the shift operation
		\begin{equation*}
			\Big(\bigoplus_{i=1}^n R\{m_i\}\Big)\{n\} \coloneqq \bigoplus_{i=1}^n R\{m_i + n\},\quad n \in \Z.
		\end{equation*}
\end{itemize}

\begin{remark}
	Note that our definition of the category $\Cob^3(2n)$ (resp.~$\Cob^3_{\modl}(2n)$) differs from Bar-Natan's definition in \cite{BN1}: we require that the objects in~$\Cob^3(2n)$, i.e.\ crossingless tangle diagrams, come with an enumeration of the circles in the diagram. This enumeration will be needed in subsequent sections in order to obtain well-defined TQFTs.
    It is worthwhile to note that while the enumeration enlarges the set of objects in~$\Cob^3(2n)$, it does \emph{not} introduce any new isomorphism classes of objects compared to Bar-Natan's definition. Indeed, the morphisms of~$\Cob^3(2n)$ are unaffected by the enumeration, so any two differently enumerated tangle diagrams with same underlying un-enumerated tangle diagram are isomorphic via the product cobordism. In fact, the functor that forgets the enumeration of the circles is an equivalence of categories.
\end{remark}

Connected cobordisms between the trivial $2$-ended tangle diagram $D_{T_0}$ and itself will have a special role throughout this paper, so let's give them a proper name.

\begin{dfn}\label{def:curtain}
	A connected cobordism of genus $k$ between the trivial $2$-ended tangle diagram and itself will be called a \emph{curtain of genus $k$}.
\end{dfn}

\cref{fig:curtaingenusone} shows a curtain of genus one.

\begin{figure}[t]
	\centering
	\includegraphics[width=0.3\textwidth]{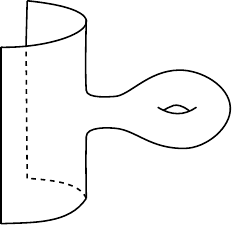}
	\caption{A curtain of genus one.}
	\label{fig:curtaingenusone}
\end{figure}

Let $C \in \Hom_{\Cob^3(2n)}(D_{T_1}, D_{T_2})$ be a morphism between two tangle diagrams $D_{T_1}$ and~$D_{T_2}$. We can turn $\Hom_{\Cob^3(2n)}(D_{T_1}, D_{T_2})$ into a graded Abelian group by setting
\begin{equation*}
	\deg C \coloneqq \chi(C) - n,
\end{equation*}
where $\chi(C)$ is the Euler characteristic of~$C$. Consequently, we can extend $\Cob^3(2n)$ to become a graded category (cf.~\cite{BN1}). Since the three local relations $S, T,$ and $4Tu$ are degree-homogeneous, $\Cob^3_{\modl}(2n)$ is graded, too. Last but not least, we use the same notion of degree to turn $\mathcal{E}$ into a graded category as well. For the sake of simplicity, we will use the same notation for the graded versions of~$\Cob^3(2n)$, $\Cob^3_{\modl}(2n)$ and~$\mathcal{E}$.

\subsection{Definition of the $\Z[G]$-complex}\label{subsec:defzgcomplex}

Given a $2n$-ended tangle $T$ with diagram~$D_T$, Bar-Natan \cite{BN1} showed how to obtain from the cube of resolutions of $D_T$ a chain complex $[D_T]$ living in $\Kom(\Mat(\Cob^3_{\modl}(2n)))$ and well-defined up to isomorphism.
For $2$-ended tangles, this complex is an invariant up to homotopy equivalence for equivalent tangles. If $T$ is obtained from a link $L$ with base point by the correspondence \cref{eq:2endlinks}, Bar-Natan showed that one can obtain the original Khovanov homology of $L$ from the complex $[T]$ (\cite{BN1}, Section 9).

There is an isomorphism in the category $\Mat(\Cob^3_{\modl}(2))$ known as \emph{delooping}\footnote{Delooping was first described by Bar-Natan \cite{BN2}, with a different version given later by Naot \cite{zbMATH05118580}. Our version is closest to Naot's, with the exception that we don't use dotted cobordisms.}, which can be used to reduce the complexity of the complex~$[T]$. It is described in \cref{fig:delooping} below.

\begin{figure}[th]
	\centering
	\captionsetup{width=.8\linewidth, font=small}
	\def\svgwidth{.5\textwidth}
\begingroup%
  \makeatletter%
  \providecommand\color[2][]{%
    \errmessage{(Inkscape) Color is used for the text in Inkscape, but the package 'color.sty' is not loaded}%
    \renewcommand\color[2][]{}%
  }%
  \providecommand\transparent[1]{%
    \errmessage{(Inkscape) Transparency is used (non-zero) for the text in Inkscape, but the package 'transparent.sty' is not loaded}%
    \renewcommand\transparent[1]{}%
  }%
  \providecommand\rotatebox[2]{#2}%
  \newcommand*\fsize{\dimexpr\f@size pt\relax}%
  \newcommand*\lineheight[1]{\fontsize{\fsize}{#1\fsize}\selectfont}%
  \ifx\svgwidth\undefined%
    \setlength{\unitlength}{387.80964048bp}%
    \ifx\svgscale\undefined%
      \relax%
    \else%
      \setlength{\unitlength}{\unitlength * \real{\svgscale}}%
    \fi%
  \else%
    \setlength{\unitlength}{\svgwidth}%
  \fi%
  \global\let\svgwidth\undefined%
  \global\let\svgscale\undefined%
  \makeatother%
  \begin{picture}(1,0.49904735)%
    \lineheight{1}%
    \setlength\tabcolsep{0pt}%
    \put(0,0){\includegraphics[width=\unitlength,page=1]{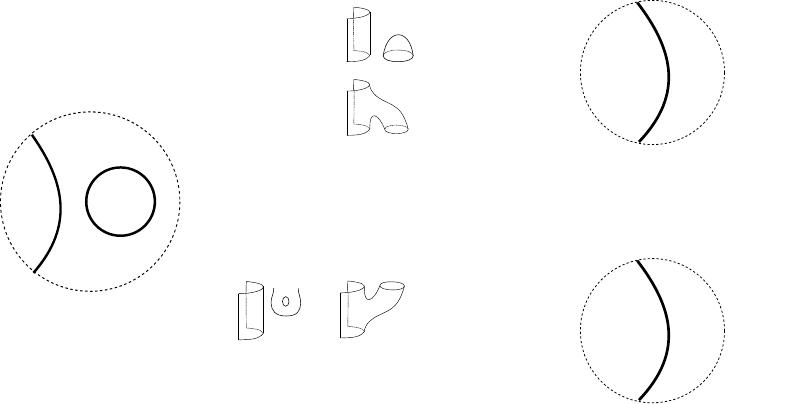}}%
    \put(0.38957827,0.10404666){\color[rgb]{0,0,0}\makebox(0,0)[lt]{\lineheight{1.25}\smash{\begin{tabular}[t]{l}-\end{tabular}}}}%
    \put(0,0){\includegraphics[width=\unitlength,page=2]{deloopingv2.pdf}}%
    \put(0.90775635,0.40523022){\color[rgb]{0,0,0}\makebox(0,0)[lt]{\lineheight{1.25}\smash{\begin{tabular}[t]{l}$\{-1\}$\end{tabular}}}}%
    \put(0.90775635,0.07549376){\color[rgb]{0,0,0}\makebox(0,0)[lt]{\lineheight{1.25}\smash{\begin{tabular}[t]{l}$\{+1\}$\end{tabular}}}}%
  \end{picture}%
\endgroup%

	\caption{Delooping.}
	\label{fig:delooping}
\end{figure}

It is a simple exercise to check that the morphisms depicted in \cref{fig:delooping} yield two mutually inverse isomorphisms in~$\Mat(\Cob^3_{\modl}(2))$. In other words, we have an isomorphism of objects 
\begin{equation*}
	\img{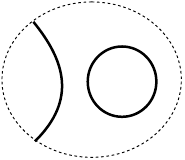} \cong \img{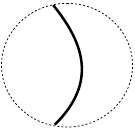}\{-1\} \oplus \img{strand.pdf}\{+1\}.
\end{equation*}

Given a $2$-ended tangle $T$ with diagram~$D_T$, we can use delooping to successively resolve every circle appearing in the complex~$[D_T]$. This yields an isomorphic complex whose chain objects consist solely of sums of grading shifted copies of~$D_{T_0}$, giving us a connection between the categories $\Kom(\Mat(\Cob^3_{\modl}(2)))$ and $\Kom(\Mat(\mathcal{E}))$. In fact:

\begin{restatable}{prop}{rpropequivcat}\label{prop:equivcat}
	The functor $B\colon \Mat(\mathcal{E}) \rightarrow \Mat(\Cob^3_{\modl}(2))$ given by inclusion is an equivalence of categories.
\end{restatable}

The proof requires quite a bit of work and is postponed to \cref{sec:appendix}.

As $B$ is an equivalence of categories, there is a functor $I\colon \Mat(\Cob^3_{\modl}(2)) \rightarrow \Mat(\mathcal{E})$ such that $I \circ B$ and $B \circ I$ are naturally isomorphic to the identity functors $\Id_{\Mat(\mathcal{E})}$ and $\Id_{\Mat(\Cob^3_{\modl}(2))}$, respectively. This functor can be constructed by using delooping, though not in a natural way: there is an ambiguity in the order of which circles get delooped.\footnote{This problem can be resolved by introducing the convention of always delooping the \emph{last} circle with respect to the enumeration. However, since we don't need a natural inverse we aren't going to introduce such a convention.} Observe that if $I$ is constructed in this way, then $I \circ B = \Id_{\Mat(\mathcal{E})}$ while $B \circ I$ is still only naturally isomorphic to~$\Id_{\Mat(\Cob^3_{\modl}(2))}$. The functor $I$ induces an equivalence of categories

\begin{equation*}
	\widehat I\colon \Kom(\Mat(\Cob^3_{\modl}(2))) \rightarrow \Kom(\Mat(\mathcal{E})).
\end{equation*}

Now let $G$ be a formal variable and consider the ring~$\Z[G]$. We equip $\Z[G]$ with a grading by setting $\deg 1 = 0$ and~$\deg G = -2$, and we consider the category~$\mathcal{M}_{\Z[G]}$. There is a functor $F\colon \mathcal{E} \rightarrow \mathcal{M}_{\Z[G]}$ sending the object $D_{T_0}\{m\} \in \ob(\mathcal{E})$ to the $\Z[G]$-module $\Z[G]\{m\}$ and a cobordism of genus $k$ to the linear map given by multiplication with~$G^k$. This functor extends by linearity to a functor $F\colon \Mat(\mathcal{E}) \rightarrow \mathcal{M}_{\Z[G]}$, which is in fact an isomorphism. Moreover, it induces yet another functor
\begin{equation*}
	\widehat F\colon \Kom(\Mat(\mathcal{E})) \rightarrow \Kom(\mathcal{M}_{\Z[G]}).
\end{equation*}

Let us make the following definition.

\begin{dfn}\label{def:zgcomplex}
	The \emph{$\Z[G]$-complex} of a $2$-ended tangle~$T$, denoted by ~$\Omega(T)$, is defined as the chain complex
	\begin{equation*}
		\Omega(T) \coloneqq \widehat F(\widehat I([T])) \in \Kom(\mathcal{M}_{\Z[G]}),
	\end{equation*}
	where $[T]$ is the Bar-Natan complex of~$T$. If $L$ is a link with base point and $T_L$ its corresponding $2$-ended tangle, we define the \emph{$\Z[G]$-complex of $L$} as $\Omega(L) \coloneqq \widehat F(\widehat I([T_L]))$.
\end{dfn}

By construction, the homotopy class of $\Omega(T)$ is an invariant for $2$-ended tangles. The construction is summarized in the schematic below.

\begin{center}
\begin{tikzcd}
	\textrm{$2$-ended tangle $T$} \arrow{d} \\
	\textrm{cube of resolutions of $T$} \arrow{d}{\textrm{\cite{BN1}}} \\
	\phantom{}[T]\in \Kom(\Mat(\Cob^3_{\modl}(2))) \arrow{d}{\widehat I, \text{ \cref{prop:equivcat}}} \\
	\widehat I([T]) \in \Kom(\Mat(\mathcal{E})) \arrow{d}{\widehat F} \\
	\Omega(T) \coloneqq \widehat F(\widehat I([T])) \in \Kom(\mathcal{M}_{\Z[G]}).
\end{tikzcd}
\end{center}

\subsection{Frobenius Systems and TQFTs}\label{subsec:frobsystqft}

In \cite{KH2}, Khovanov describes a rank~$2$ Frobenius system $\mathcal{F}_{\textrm{univ}}$ (denoted by $\mathcal{F}_5$ in \cite{KH2}) which is universal in the sense that any other rank two Frobenius system can be obtained from it by a base change and a twist. 
Consequently, the chain complex obtained by applying the $\mathcal{F}_{\textrm{univ}}$-TQFT to the cube of resolutions is sometimes called \emph{universal Khovanov complex}. Before we go into more detail, let us first recall the definition of a Frobenius system (as in \cite{KH2}).

\begin{dfn}
    A Frobenius system $\mathcal{F} = (R, A, \Delta, \epsilon)$ is a $4$-tuple consisting of a graded commutative unitary ring $R$ and a graded or filtered free $R$-module $A$ equipped with a commutative algebra structure (multiplication~$m$, unit~$\iota$) and a cocommutative coalgebra structure (comultiplication~$\Delta$, counit~$\epsilon$) that are related by the so-called \emph{Frobenius identity}:
    \begin{equation*}
    	\Delta \circ m = (\Id \otimes m) \circ (\Delta \otimes \Id).
    \end{equation*}
    The maps defining the (co-)algebra structure are required to be homogeneous of a certain degree.
\end{dfn}

We will only be interested in rank $2$ Frobenius systems, i.e.\ Frobenius systems where $A \cong R1 \oplus RX$ as $R$-modules for some~$X \in A$, and for such systems we will always use the grading $\deg 1 = 0$ and~$\deg X = -2$.

Given a rank $2$ Frobenius system $\mathcal{F} = (R, A, \Delta, \epsilon)$, we can define a functor (a~TQFT) $\mathcal{F}\colon \Cob^3_{\modl}(2) \rightarrow \mathcal{M}_A$ as follows:\footnote{We abuse notation and use $\mathcal{F}$ to denote the Frobenius system, the corresponding TQFT, and sometimes even the Frobenius algebra of the system.}
\begin{enumerate}
	\item On objects, $\mathcal{F}$ acts in the following way:
		\begin{equation*}
			\mathcal{F}\Big(\img{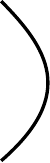} \sqcup \underbrace{\img{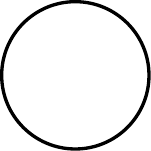} \cdots \img{circle.pdf}}_{n \textrm{ times}}\Big) = \underline{A}\{1\} \otimes_R \underbrace{A\{1\} \otimes_R \cdots \otimes_R A\{1\}}_{n \textrm{ times}}.
		\end{equation*}
		Here, the special strand corresponds to the first tensor factor while the other factors are ordered according to the enumeration of the circles. The underline indicates the action of $A$ on the tensor product $\underline{A}\{1\} \otimes_R A\{1\}^{\otimes n}$, turning it into an $A$-module.
	\item The morphisms of $\Cob^3_{\modl}(2)$ can be expressed as sums of compositions of disjoint unions of the following elementary cobordisms (details can be found in \cite{KH1}):

		\begin{equation*}
		\begin{split}
			\raisebox{-0.6\baselineskip}{\includegraphics[width=0.1\textwidth]{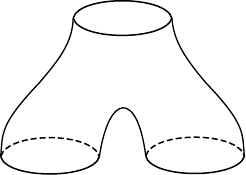}}\quad \raisebox{-0.6\baselineskip}{\includegraphics[width=0.1\textwidth]{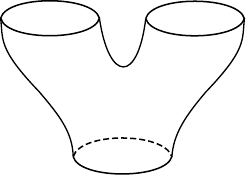}}\quad \raisebox{-0.6\baselineskip}{\includegraphics[width=0.04\textwidth]{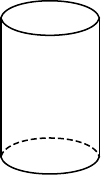}}\quad \raisebox{-0.6\baselineskip}{\includegraphics[width=0.05\textwidth]{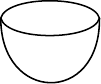}}\quad \raisebox{-0.6\baselineskip}{\includegraphics[width=0.05\textwidth]{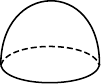}}\quad
			\raisebox{-0.6\baselineskip}{\includegraphics[width=0.07\textwidth]{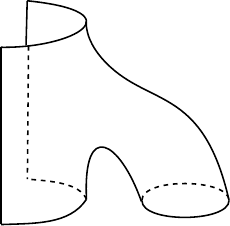}}\quad \raisebox{-0.6\baselineskip}{\includegraphics[width=0.07\textwidth]{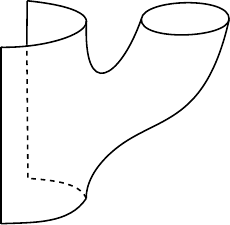}}\quad \raisebox{-0.6\baselineskip}{\includegraphics[width=0.03\textwidth]{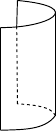}}
		\end{split}
		\end{equation*}\smallskip

		\noindent Hence it is enough to define $\mathcal{F}$ on these elementary cobordisms:\smallskip
		\begin{align*}
			& \mathcal{F}\bigg(\raisebox{-0.6\baselineskip}{\includegraphics[width=0.1\textwidth]{m.pdf}}\bigg) = m\colon A\{1\} \otimes A\{1\} \rightarrow A\{1\} && \mathcal{F}\big(\raisebox{-0.3\baselineskip}{\includegraphics[width=0.05\textwidth]{unit.pdf}}\big) = \iota\colon R\{1\} \rightarrow A\{1\} \\ &
			\mathcal{F}\bigg(\raisebox{-0.6\baselineskip}{\includegraphics[width=0.1\textwidth]{Delta.pdf}}\bigg) = \Delta\colon A\{1\} \rightarrow A\{1\} \otimes A\{1\} && \mathcal{F}\big(\raisebox{-0.3\baselineskip}{\includegraphics[width=0.05\textwidth]{counit.pdf}}\big) = \epsilon\colon A\{1\} \rightarrow R\{1\} \\ &
			\mathcal{F}\bigg(\raisebox{-0.6\baselineskip}{\includegraphics[width=0.04\textwidth]{id.pdf}}\bigg) = \Id\colon A\{1\} \rightarrow A\{1\} && \mathcal{F}\bigg(\raisebox{-0.6\baselineskip}{\includegraphics[width=0.03\textwidth]{idSpecial.pdf}}\bigg) = \Id\colon \underline{A}\{1\} \rightarrow \underline{A}\{1\} \\ &
			\mathcal{F}\bigg(\raisebox{-0.6\baselineskip}{\includegraphics[width=0.07\textwidth]{mSpecial.pdf}}\bigg) = m\colon \underline{A}\{1\} \otimes A\{1\} \rightarrow \underline{A}\{1\} && \mathcal{F}\bigg(\raisebox{-0.6\baselineskip}{\includegraphics[width=0.07\textwidth]{DeltaSpecial.pdf}}\bigg) = \Delta\colon \underline{A}\{1\} \rightarrow \underline{A}\{1\} \otimes A\{1\}
		\end{align*}
\end{enumerate}\smallskip

Given a $2$-ended tangle $T$ with diagram~$D_T$, we can apply $\mathcal{F}$ to its cube of resolutions and obtain a Khovanov-type chain complex in the usual way. The resulting complex is sometimes called \emph{unreduced}. We will denote it by either $C_\mathcal{F}(T)$ or~$C_\mathcal{F}(D_T)$, which is justified since different choices of the diagram $D_T$ yield homotopy equivalent complexes.

\begin{remark}\label{rem:tqft}
	\begin{enumerate}
		\item By our definition, a TQFT yields a complex over the category~$\mathcal{M}_A$. However, this category is not Abelian, so in order to take homology one has to move from $\Kom(\mathcal{M}_A)$ to $\Kom(A\textrm{-Mod})$ using inclusion, where $A\textrm{-Mod}$ denotes the category of graded $A$-modules. We accept this inconvenience because $\mathcal{M}_A$ is a simpler category and complexes are our main working tool (in contrast to homology).
		\item If we forget about the base point, then the TQFT above can be viewed as $\mathcal{F}\colon \Cob^3_{\modl}(0) \rightarrow \mathcal{M}_R$ which yields the usual Khovanov complex corresponding to the Frobenius system~$\mathcal{F}$.
	\end{enumerate}
\end{remark}

We will mainly be interested in the following two Frobenius systems:

\begin{dfn}\label{def:f5system}
	The Frobenius system $\mathcal{F}_{\textrm{univ}} = (R_{\textrm{univ}}, A_{\textrm{univ}}, \Delta, \epsilon)$ is defined as follows:
	\begin{equation*}
		R_{\textrm{univ}} = \mathbb{Z}[h, t],\quad A_{\textrm{univ}} = R_\textrm{univ}[X]/(X^2-hX-t).
	\end{equation*}
	\begin{alignat*}{2}
		& \epsilon(1) = 0,\quad && \Delta(1) = 1\otimes X + X\otimes 1 - h1\otimes 1 \\ &
		\epsilon(X) = 1,\quad && \Delta(X) = X\otimes X + t1\otimes 1.
	\end{alignat*}
	We equip $\mathcal{F}_{\textrm{univ}}$ with a grading by setting $\deg X = \deg h = -2$ and~$\deg t = -4$.
\end{dfn}

\begin{dfn}\label{def:ZGsystem}
	The Frobenius system $\mathcal{F}_{\mathbb{Z}[G]} = (R_{\mathbb{Z}[G]}, A_{\mathbb{Z}[G]}, \Delta, \epsilon)$ is defined as follows:
	\begin{equation*}
		R_{\mathbb{Z}[G]} = \mathbb{Z}[G],\quad A_{\mathbb{Z}[G]} = R_{\mathbb{Z}[G]}[X]/(X^2+GX).
	\end{equation*}
	\begin{alignat*}{2}
		& \epsilon(1) = 0,\quad && \Delta(1) = 1\otimes X + X\otimes 1 + G1\otimes 1 \\ &
		\epsilon(X) = 1,\quad && \Delta(X) = X\otimes X.
	\end{alignat*}
	We equip $\mathcal{F}_{\mathbb{Z}[G]}$ with a grading by setting $\deg X = \deg G = -2$.
\end{dfn}

It is easy to see that $\mathcal{F}_{\mathbb{Z}[G]}$ can be obtained from $\mathcal{F}_{\textrm{univ}}$ by a base change that sends $t \mapsto 0$ and~$h \mapsto -G$.
To simplify notation, we will write $C_\textrm{univ}$ for the unreduced complex $C_{\mathcal{F}_\textrm{univ}}$ and $C_{\mathbb{Z}[G]}$ for $C_{\mathcal{F}_{\mathbb{Z}[G]}}$.

Using the Frobenius system $\mathcal{F}_{\mathbb{Z}[G]}$, we can define a second type of chain complex as follows.

\begin{dfn}\label{def:redzgcomplex}
	Let $T$ be a $2$-ended tangle with diagram $D_T$ and $C_{\mathbb{Z}[G]}(D_T)$ the corresponding $\mathcal{F}_{\mathbb{Z}[G]}$-complex. The \emph{reduced} $\mathcal{F}_{\mathbb{Z}[G]}$-complex of $T$ is defined as
	\begin{equation*}
		\llbracket T \rrbracket \coloneqq C_{\mathbb{Z}[G]}(D_T) \otimes_{A_{\mathbb{Z}[G]}} A_{\mathbb{Z}[G]}/(X)\{-1\} \in \Kom(\mathcal{M}_{\mathbb{Z}[G]}).
	\end{equation*}
\end{dfn}

The notation $\llbracket T \rrbracket$ is justified since different choices of the diagram $D_T$ yield homotopy equivalent complexes (we will sometimes use $\llbracket D_T \rrbracket$ nevertheless). Observe that reducing has the following effect on summands in~$C_{\mathbb{Z}[G]}$:
\begin{equation*}
	\underline{A_{\mathbb{Z}[G]}}\{1\} \otimes_{R_{\mathbb{Z}[G]}} A_{\mathbb{Z}[G]}\{1\}^{\otimes n} \overset{\textrm{reduce}}\longrightarrow R_{\mathbb{Z}[G]} \otimes_{R_{\mathbb{Z}[G]}} A_{\mathbb{Z}[G]}\{1\}^{\otimes n}.
\end{equation*}
In particular, the reduced complex has no longer an $A_{\mathbb{Z}[G]}$-module structure. Also note that the first factor is no longer affected by a shift in grading.\footnote{This is needed in order for the dual of the reduced complex to correspond to the reduced complex of the mirror image. Here, we use the usual convention that the signs of the homological and quantum grading in the dual are switched.} We will see in the next section that the reduced $\mathcal{F}_{\mathbb{Z}[G]}$-complex $\llbracket T \rrbracket$ is in fact isomorphic to the $\mathbb{Z}[G]$-complex $\Omega(T)$ defined in \cref{subsec:defzgcomplex}. 

\subsection{Equivalence of the $\mathcal{F}_{\textrm{univ}}$- and $\mathbb{Z}[G]$-theory}\label{subsec:equivtheories}

Let $\mathcal{F} = (R, A, \Delta, \epsilon)$ be a rank $2$ Frobenius system. The corresponding TQFT gives us a functor
\begin{equation*}
	\alpha = \alpha_\mathcal{F}\colon \Kom(\Mat(\Cob^3_{\modl}(2))) \rightarrow \Kom(\mathcal{M}_A).
\end{equation*}
Let us now show that the $\mathcal{F}_{\textrm{univ}}$- and the $\mathbb{Z}[G]$-theory are equivalent.

\begin{lem}\label{lem:alpha}
	Let $D_T$ be a $2$-ended tangle diagram and let $\mathcal{F} = (R,A,\Delta,\epsilon)$ be a rank $2$ Frobenius system. We can see $A$ as a $\mathbb{Z}[G]$-module by letting $G$ act as $\mathcal{F}(\img{curtain.pdf})$. Then the functor $\gamma\colon \Kom(\mathcal{M}_{\mathbb{Z}[G]}) \rightarrow \Kom(\mathcal{M}_A)$ defined as
	\begin{equation*}
		\gamma (Y) \coloneqq Y \otimes_{\mathbb{Z}[G]} A\{1\},\quad Y \in \Kom(\mathcal{M}_{\mathbb{Z}[G]})
	\end{equation*}
	satisfies
	\begin{equation*}
		\alpha([D_T]) \cong \gamma(\Omega(D_T))
	\end{equation*}
	(cf.~\cref{fig:equivUnivAndZG} below).
\end{lem}

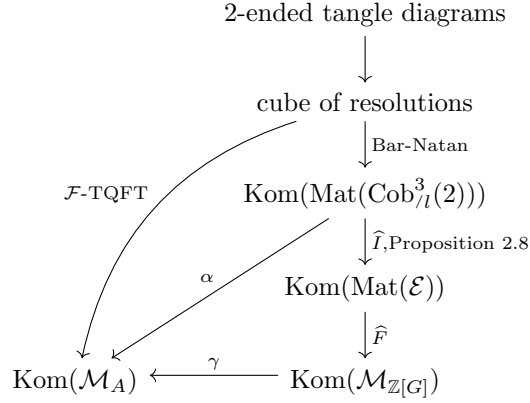
\begin{figure}[ht!]
\centering
\begin{tikzcd}
	& \textrm{$2$-ended tangle diagrams} \arrow{d} \\
	& \textrm{cubes of resolutions} \arrow[dddl, bend right, swap, "\mathcal{F}\textrm{-TQFT}"] \arrow{d}{\textrm{\cite{BN1}}} \\
	& \Kom(\Mat(\Cob^3_{\modl}(2))) \arrow[ddl, swap, "\alpha"] \arrow{d}{\widehat I, \textrm{ \cref{prop:equivcat}}} \\
	& \Kom(\Mat(\mathcal{E})) \arrow{d}{\widehat F} \\
	\Kom(\mathcal{M}_A) &  \Kom(\mathcal{M}_{\mathbb{Z}[G]}) \arrow[l, shift right, swap, "\gamma"]
\end{tikzcd}
\caption{The functors and constructions figuring in the statement of \cref{lem:alpha}.}
\label{fig:equivUnivAndZG}
\end{figure}

\begin{proof}
	We know that the functor $\widehat{I}$ is an equivalence of categories (with ``inverse''~$\widehat{B}$, cf.~\cref{prop:equivcat}) and that $\widehat{F}$ is an isomorphism of categories, thus if $\beta = \widehat{F} \circ \widehat{I}$ and $\zeta = \widehat{B} \circ \widehat{F}^{-1}$ we have that $\zeta(\beta(C)) \cong C$ for all $C \in \Kom(\Mat(\Cob^3_{\modl}(2)))$ (see \cref{fig:proofEquivUnivAndZG} below). In order to show that $\alpha([D_T]) \cong \gamma(\widehat{F}(\widehat{I}([D_T])))$, it is enough to prove that $\alpha(\zeta(Y)) \cong \gamma(Y)$ for all $Y \in \Kom(\mathcal{M}_{\mathbb{Z}[G]})$. This is done by introducing a new functor $J\colon \mathcal{E} \rightarrow \mathcal{M}_A$ sending the trivial $2$-ended tangle diagram $D_{T_0}$ to the $A$-module $A$ and the curtain with genus $k \geq 0$ to the linear map given by~$\mathcal{F}(\img{curtain.pdf})^k$. The functor $J$ induces a functor
	\begin{equation*}
		\widehat{J}\colon \Kom(\Mat(\mathcal{E})) \rightarrow \Kom(\mathcal{M}_A).
	\end{equation*}
	It is easy to see that $\widehat{J}=\alpha \circ \widehat{B}$, so $\alpha$ is naturally isomorphic to~$\widehat{J} \circ \widehat{I}$. Thus it only remains to check that $\widehat{J} \cong \gamma \circ \widehat{F}$. This follows immediately from the definitions.
\end{proof}

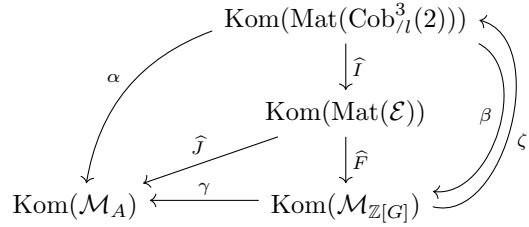
\begin{figure}[th]
    \centering
\begin{tikzcd}
	& \Kom(\Mat(\Cob^3_{\modl}(2))) \arrow[ddl, bend right, swap, "\alpha"] \arrow{d}{\widehat I} \arrow[dd, bend left=80, swap, "\beta"] \\
	& \Kom(\Mat(\mathcal{E})) \arrow[dl, swap, "\widehat{J}"] \arrow{d}{\widehat F} \\
	\Kom(\mathcal{M}_A) &  \Kom(\mathcal{M}_{\mathbb{Z}[G]}) \arrow[l, shift right, swap, "\gamma"] \arrow[uu, bend right=90, swap, "\zeta"]
\end{tikzcd}
\caption{Functors used in the proof of \cref{lem:alpha}.}
    \label{fig:proofEquivUnivAndZG}
\end{figure}

Before we move on, let us first show that the reduced $\mathcal{F}_{\mathbb{Z}[G]}$-complex $\llbracket\, \cdot\, \rrbracket$ is isomorphic to~$\Omega$.

\begin{prop}\label{prop:ZGcomplexesequiv}
	The reduced $\mathcal{F}_{\mathbb{Z}[G]}$-complex $\llbracket D_T \rrbracket$ is isomorphic to the $\mathbb{Z}[G]$-complex~$\Omega(D_T)$.
\end{prop}

\begin{proof}
	Using \cref{lem:alpha} with $\mathcal{F} = \mathcal{F}_{\mathbb{Z}[G]}$ and that $A_{\mathbb{Z}[G]}/(X) \cong \mathbb{Z}[G]$, we obtain
	\begin{equation*}
	\begin{split}
		\llbracket D_T \rrbracket & = C_{\mathbb{Z}[G]}(D_T) \otimes_{A_{\mathbb{Z}[G]}} A_{\mathbb{Z}[G]}/(X)\{-1\} \\ &
		= \alpha([D_T]) \otimes_{A_{\mathbb{Z}[G]}} A_{\mathbb{Z}[G]}/(X)\{-1\} \\ &
		\cong \gamma(\Omega(D_T)) \otimes_{A_{\mathbb{Z}[G]}} A_{\mathbb{Z}[G]}/(X)\{-1\} \\ &
		\cong \Omega(D_T) \otimes_{\mathbb{Z}[G]} A_{\mathbb{Z}[G]}/(X)\{-1\} \\ &
		\cong \Omega(D_T).\myqed
	\end{split}
	\end{equation*}
\end{proof}

\cref{prop:ZGcomplexesequiv} tells us that the reduced $\mathcal{F}_{\mathbb{Z}[G]}$-complex and the $\mathbb{Z}[G]$-complex are isomorphic. We will from now on denote both complexes with $\llbracket\, \cdot\, \rrbracket$ and no longer distinguish between them.

\begin{theorem}\label{thm:F5fromZG}
	Let $D_T$ be a $2$-ended tangle diagram and let $\mathcal{F} = (R,A,\Delta,\epsilon)$ be a rank $2$ Frobenius system. The $\mathcal{F}$-complex $C_\mathcal{F}(D_T)$ is determined by the $\mathbb{Z}[G]$-complex $\llbracket D_T \rrbracket$ in the following way:
	\begin{equation*}
		C_\mathcal{F}(D_T) \cong \llbracket D_T \rrbracket \otimes_{\mathbb{Z}[G]} A\{1\} \in \Kom(\mathcal{M}_A),
	\end{equation*}
	where $A$ is a $\mathbb{Z}[G]$-module via $G$ acting as $\mathcal{F}\big( \img{curtain.pdf} \big)$.
\end{theorem}

\begin{proof}
	The statement of the theorem follows immediately from \cref{lem:alpha} and \cref{prop:ZGcomplexesequiv}:
	\begin{equation*}
		C_\mathcal{F}(D_T) = \alpha([D_T]) \cong \gamma(\Omega(D_T)) \cong \llbracket D_T \rrbracket \otimes_{\mathbb{Z}[G]} A\{1\}.\myqed
	\end{equation*}
\end{proof}

Observe that \cref{thm:F5fromZG} specializes to \cref{thm:equiv} from the introduction:
\thmequiv*

\begin{proof}
	Apply \cref{thm:F5fromZG} with $\mathcal{F} = \mathcal{F}_{\univ}$.
\end{proof}

Let us make explicit how $\Z[G]$-homology determines the original Khovanov homology
(Naot mentions this statement in \cite[Section~6.6]{zbMATH05118580}).
\begin{corollary}\label{cor:zg2kh}
For all knots~$K$, the unreduced integral Khovanov chain complex may be obtained from $\llbracket K\rrbracket$
by tensoring with the $\Z[G]$-module $\Z\{-1\}\oplus \Z\{1\}$, where $G$ acts as
$\bigl(\begin{smallmatrix} 0 & 2 \\ 0 & 0 \end{smallmatrix}\bigr)$.
More sloppily said, replace every copy of $\Z[G]\{m\}$ by $\Z\{m-1\}\oplus \Z\{m+1\}$,
and every differential $n G^k$ with $n,k\in\Z, k\geq0$ by $\bigl(\begin{smallmatrix} n & 0 \\ 0 & n \end{smallmatrix}\bigr)$
for~$k = 0$, by $\bigl(\begin{smallmatrix} 0 & 2n \\ 0 & 0 \end{smallmatrix}\bigr)$ for~$k = 1$,
and by the zero matrix for~$k \geq 2$.
\end{corollary}
\begin{proof}
Apply \cref{thm:F5fromZG} to the Frobenius system $\Z[X]/(X^2)$ over~$\Z$, and forget the action of the algebra.
\end{proof}

\cref{thm:F5fromZG} shows us how to obtain the $\mathcal{F}$-complex $C_\mathcal{F}(D_T)$ from the $\mathbb{Z}[G]$-complex $\llbracket D_T \rrbracket$ for any rank $2$ Frobenius system~$\mathcal{F}$, which is in particular true for the universal system~$\mathcal{F}_\textrm{univ}$. In order to show that the $\mathcal{F}_\textrm{univ}$- and the $\mathbb{Z}[G]$-theory are in fact equivalent, it remains to prove that $\llbracket D_T \rrbracket$ is also determined by~$C_\textrm{univ}(D_T)$.

\begin{theorem}\label{thm:ZGfromF5} 
	Let $D_T$ be a $2$-ended tangle diagram. The $\mathbb{Z}[G]$-complex $\llbracket D_T \rrbracket$ is determined by the $\mathcal{F}_\textrm{univ}$-complex $C_\textrm{univ}(D_T)$ in the following way:
	\begin{equation*}
		\llbracket D_T \rrbracket \cong C_\textnormal{univ}(D_T) \otimes_{A_\textnormal{univ}} \mathbb{Z}[G]\{-1\} \in \Kom(\mathcal{M}_{\mathbb{Z}[G]}),
	\end{equation*}
	where $\mathbb{Z}[G]$ is an $A_{\textnormal{univ}}$-module by $X$ and $t$ acting as $0$ and $h$ as~$-G$.
\end{theorem}

\begin{proof}
	By \cref{thm:F5fromZG},
	\begin{equation*}
		C_{\mathcal{F}_\textrm{univ}}(D_T) \cong \llbracket D_T \rrbracket \otimes_{\mathbb{Z}[G]} A_\textrm{univ}\{1\}.\tag{$*$}
	\end{equation*}
	Consider $A_{\mathbb{Z}[G]}$ as an $A_{\textrm{univ}}$-module by letting $t$ act as $0$ and $h$ as~$-G$. Tensoring $(*)$ with $A_{\mathbb{Z}[G]}$ over $A_\textrm{univ}$ yields
	\begin{equation*}
	\begin{split}
		C_\textrm{univ}(D_T) \otimes_{A_\textrm{univ}} A_{\mathbb{Z}[G]} & \cong \big(\llbracket D_T \rrbracket \otimes_{\mathbb{Z}[G]} A_\textrm{univ}\{1\}\big) \otimes_{A_\textrm{univ}} A_{\mathbb{Z}[G]} \\ &
		\cong \llbracket D_T \rrbracket \otimes_{\mathbb{Z}[G]} A_{\mathbb{Z}[G]}\{1\} \\ &
		\cong C_{\mathbb{Z}[G]}(D_T).
	\end{split}
	\end{equation*}
	Therefore
	\begin{equation*}
	\begin{split}
		\llbracket D_T \rrbracket & = C_{\mathbb{Z}[G]}(D_T) \otimes_{A_{\mathbb{Z}[G]}} A_{\mathbb{Z}[G]}/(X)\{-1\} \\ &
		\cong (C_\textrm{univ}(D_T) \otimes_{A_\textrm{univ}} A_{\mathbb{Z}[G]}) \otimes_{A_{\mathbb{Z}[G]}} A_{\mathbb{Z}[G]}/(X)\{-1\} \\ &
		\cong C_\textrm{univ}(D_T) \otimes_{A_\textrm{univ}} A_{\mathbb{Z}[G]}/(X)\{-1\} \\ &
		\cong C_\textrm{univ}(D_T) \otimes_{A_\textrm{univ}} \mathbb{Z}[G]\{-1\}.\myqed
	\end{split}
	\end{equation*}
\end{proof}

The discussion in this section can be summarized by the commutative diagram in \cref{fig:univAndZG}, where $\xi \colon \Kom(\mathcal{M}_{A_{\textrm{univ}}}) \to \Kom(\mathcal{M}_{\mathbb{Z}[G]})$ is the functor given by
\begin{equation*}
	\xi(C) \coloneqq C \otimes_{A_{\textrm{univ}}} \mathbb{Z}[G]\{-1\}
\end{equation*}
for $C \in \Kom(\mathcal{M}_{A_\textrm{univ}})$.

\begin{figure}[ht!]
    \centering
\begin{tikzcd}
	& \textrm{$2$-ended tangle diagrams} \arrow{d} & \\
	& \textrm{cubes of resolutions} \arrow[dddl, bend right, swap, "\mathcal{F}_{\textrm{univ}}\textrm{-TQFT}"] \arrow[dddr, bend left, "\textrm{red. } \mathcal{F}_{\mathbb{Z}[G]}\textrm{-TQFT}"] \arrow{d}{\textrm{\cite{BN1}}} & \\
	& \Kom(\Mat(\Cob^3_{\modl}(2))) \arrow{d}{\widehat I, \textrm{ \cref{prop:equivcat}}} & \\
	& \Kom(\Mat(\mathcal{E})) \arrow[dl, swap, "\widehat J"] \arrow{dr}{\widehat F} & \\
	\Kom(\mathcal{M}_{A_\textrm{univ}}) \arrow[rr, shift right, swap, "\xi"] & & \Kom(\mathcal{M}_{\mathbb{Z}[G]}) \arrow[ll, shift right, swap, "\gamma"]
\end{tikzcd}
\caption{A summary of the relationships discussed in \cref{subsec:equivtheories}.}
    \label{fig:univAndZG}
\end{figure}
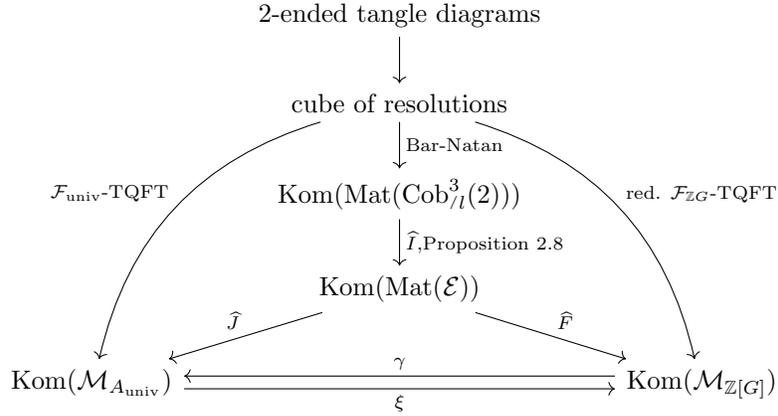

\subsection{Reduced $\mathbb{Z}[G]$-homology}\label{subsec:redzghomology}

Let $T$ be a $2$-ended tangle. We have seen in the previous subsection that the reduced $\mathbb{Z}[G]$-complex $\llbracket T \rrbracket$ is determined by the $\mathcal{F}_\textrm{univ}$-complex $C_\textrm{univ}(T)$ and vice-versa. One advantage of the $\mathbb{Z}[G]$-complex is that setting $G = 1$ yields a particularly simple homology theory.

\begin{prop}\label{prop:redzghomdim}
	Let $T$ be a $2$-ended tangle with a single component. Then
	\begin{equation*}
		H(\llbracket T \rrbracket_{G=1})\cong \mathbb{Z},
	\end{equation*}
	where $\llbracket T \rrbracket_{G=1}$ is the complex $\llbracket T \rrbracket_{G=1} \coloneqq \llbracket T \rrbracket \otimes_{\mathbb{Z}[G]} \mathbb{Z}[G]/(G-1)$.
\end{prop}

\begin{proof}
	Let us first look at the unreduced situation over the rationals, i.e.\ the complex
	\begin{equation*}
		C(T) \coloneqq (C_{\mathbb{Z}[G]}(T) \otimes_{A_{\mathbb{Z}[G]}} A_{\mathbb{Q}[G]})
		\otimes_{A_{\mathbb{Q}[G]}} \mathbb{Q}[G]/(G-1),
	\end{equation*}
	where $A_{\mathbb{Z}[G]} = \mathbb{Z}[X, G]/(X^2+GX)$ and $A_{\mathbb{Q}[G]} = \mathbb{Q}[X, G]/(X^2+GX)$. Note that $C(T)$ can be equivalently obtained from the Frobenius algebra $\mathbb{Q}[X]/(X^2+X)$ in the usual way. By \cite[Proposition 2.3]{mvt}, we know that
	\begin{equation*}
		H(C(T)) \cong \mathbb{Q} \oplus \mathbb{Q}.
	\end{equation*}
	In fact, using Wehrli's edge-coloring technique \cite[Section 2.1]{wehrliphd}, one obtains a decomposition
	\begin{equation*}
		C(T) = XC(T) \oplus (X+1)C(T),
	\end{equation*}
	where $XC(T)$ and $(X+1)C(T)$ are the subcomplexes generated by all elements having $X$ and $X+1$ as the first tensor factor (i.e.\ at the base point), respectively. Similar to \cite[Theorem 5]{wehrliphd}, one can show that both $XC(T)$ and $(X+1)C(T)$ have homology of dimension one.
	Now let's look at the reduced situation over the rationals, i.e.\ the complex
	\begin{equation*}
		\llbracket T \rrbracket_{\mathbb{Q}, G=1} \coloneqq \llbracket T \rrbracket_{G=1} \otimes_\mathbb{Z} \mathbb{Q},
	\end{equation*}
	By construction, $\llbracket T \rrbracket_{\mathbb{Q}, G=1}$ is equivalent to the complex $C(T)$ with $X$ set to $0$ in the first tensor factor of every summand in~$C(T)$, which means that the summand $XC(T)$ becomes trivial after reducing. Hence
	\begin{equation*}
		H(\llbracket T \rrbracket_{\mathbb{Q}, G=1}) \cong \mathbb{Q}.
	\end{equation*}
	Switching back to the integers, the above tells us that $\dim_\mathbb{Q}(\llbracket T \rrbracket_{G=1}) = 1$. Hence it remains to show that $\llbracket T \rrbracket_{G=1}$ has no torsion. This is done in the same way as in the proof of \cite[Proposition 2.4, (ii)]{mvt}.
\end{proof}

\begin{remark}
	Let $T$ be a $2$-ended tangle corresponding to a knot~$K$. It is interesting to note that one can extract the Rasmussen $s_\mathbb{F}$-invariant of $K$ over any field $\mathbb{F}$ from the $\mathbb{Z}[G]$-homology of~$K$. Indeed, consider the $\mathbb{Z}[G]$-complex with coefficients switched to some field~$\mathbb{F}$, i.e.
	\begin{equation*}
		\llbracket K \rrbracket_{\mathbb{F}[G]} = \llbracket K \rrbracket \otimes_{\mathbb{Z}[G]} \mathbb{F}[G].
	\end{equation*}
	This is a Khovanov-type complex over the PID~$\mathbb{F}[G]$, hence it decomposes into a single grading-shifted copy of the base ring $\mathbb{F}[G]\{n\}$ and some summands of the form $\mathbb{F}[G]\{m\} \overset{G^k}\rightarrow \mathbb{F}[G]\{2k+m\}$ for $k,m,n \in \mathbb{Z}$ and $k\geq 0$ (a so-called \emph{pawn} and several \emph{$G^k$-knights}, cf.~\cref{dfn:pawnknight}). Therefore, setting $G = 1$ in $\llbracket K \rrbracket_{\mathbb{F}[G]}$ yields
	\begin{equation*}
		\llbracket K \rrbracket_{\mathbb{F}[G]} \otimes_{\mathbb{F}[G]} \mathbb{F}[G]/(G-1) \simeq \mathbb{F}[G]\{n\} \otimes_{\mathbb{F}[G]} \mathbb{F}[G]/(G-1).
	\end{equation*}
	We claim that~$n$, i.e.\ the filtered degree of the generator of $\mathbb{F}[G]\{n\}$ in homology, is equal to~$s_\mathbb{F}(K)$. By \cite{mvt}, $s_\mathbb{F}(K)$ can be obtained from the homology of the unreduced complex $C_{\mathcal{F}_{\mathbb{F}[G]}}(K)$ corresponding to the Frobenius algebra $A_{\mathbb{F}[G]} = \mathbb{F}[G,X]/(X^2+GX)$ after setting~$G = 1$. On the other hand, using \cref{thm:F5fromZG} and the decomposition of $\llbracket K \rrbracket_{\mathbb{F}[G]}$ described above, we can write $C_{\mathcal{F}_{\mathbb{F}[G]}}(K)$ as
	\begin{equation*}
	\begin{split}
		C_{\mathcal{F}_{\mathbb{F}[G]}}(K) & \cong \llbracket K \rrbracket_{\mathbb{F}[G]} \otimes_{\mathbb{F}[G]} A_{\mathbb{F}[G]}\{1\} \\ &
                \cong (\mathbb{F}[G]\{n\} \oplus \mathcal{R}) \otimes_{\mathbb{F}[G]} A_{\mathbb{F}[G]}\{1\} \\ &
                \cong A_{\mathbb{F}[G]}\{n+1\} \oplus (\mathcal{R} \otimes_{\mathbb{F}[G]} A_{\mathbb{F}[G]}\{1\}),
	\end{split}
	\end{equation*}
	where $\mathcal{R}$ consists solely of summands $\mathbb{F}[G]\{m\} \overset{G^k}\rightarrow \mathbb{F}[G]\{2k+m\}$. If we now set $G = 1$ and take homology, we obtain
	\begin{equation*}
		H(C_{\mathcal{F}_{\mathbb{F}[G]}}(K) \otimes_{\mathbb{F}[G]} \mathbb{F}[G]/(G-1)) \cong \mathbb{F}[X]/(X^2+X)\{n+1\}.
	\end{equation*}
	Now $\mathbb{F}[X]/(X^2+X)\{n+1\}$ is generated by $1$ and $X$ in filtered degrees $n+1$ and $n-1$ respectively, hence $s_\mathbb{F}(K) = n$ by \cite{ras3} as claimed.
\end{remark}

\subsection{The Bar-Natan complex of tangles with base point}\label{subsect:bnzg}
Recall that the crossingless unoriented $2n$-ended tangle diagrams in $\Cob^3(2n)$ lie inside a fixed disk with fixed end points.
Let us fix one of those end points as base point.
Given a cobordism from the trivial 2-ended tangle diagram $D_{T_0}$ to itself,
and a cobordism in $\Cob^3(2n)$ between diagrams $D$ and~$D'$, one may glue these two cobordisms together such
that one of the end points of $D_{T_0}$ gets attached to the base point of $D$ and~$D'$.
This gives a bilinear map
\[
\Hom_{\Cob^3(2)}(D_{T_0}, D_{T_0}) \times \Hom_{\Cob^3(2n)}(D, D') \to \Hom_{\Cob^3(2n)}(D, D').
\]
Quotienting by the relations~$l$, and using that $\Hom_{\Cob^3(2)}(D_{T_0}, D_{T_0})$ is isomorphic to~$\Z[G]$, we obtain a $\Z[G]$-action on each of the morphism $\Z$-modules of~$\Cob^3_{\modl}(2n)$.
\begin{dfn}\label{dfn:bnzg}
Denote by $\Cob^{3,\bullet}_{\modl}(2n)$ the $\Z[G]$-enriched category
obtained from $\Cob^3_{\modl}(2n)$ by fixing one of the tangle end points as base point
and letting $\Z[G]$ act on the morphism groups as described above.
For a $2n$-ended tangle diagram $D$ with base point,
denote by $[D]^{\bullet}$ the Bar-Natan chain complex of $D$ over $\Cob^{3,\bullet}_{\modl}(2n)$.
Here, we identify equivalence classes of tangle diagrams in the disk in which $D$ lives
with equivalence classes of tangle diagrams
in the disk fixed for $\Cob^{3,\bullet}_{\modl}(2n)$, using a homeomorphism (which is unique up to isotopy) between
these disks that sends end points to end points and base point to base point.
\end{dfn}
Note that one may recover $\Cob^{3}_{\modl}(2n)$ from $\Cob^{3,\bullet}_{\modl}(2n)$ and $[D]$ from $[D]^{\bullet}$
by simply forgetting the $\Z[G]$-action and the base point. In other words, \cref{dfn:bnzg} only introduces the action of~$G$, but does not introduce any new objects and morphisms.
\begin{remark}
For~$n = 1$, the $\Z[G]$-action on $\Cob^{3}_{\modl}(2)$ is by construction the same as the one obtained
via the equivalence of $\Cob^{3}_{\modl}(2)$ and~$\mathcal{M}_{\Z[G]}$. In particular, for 2-ended tangle diagrams~$D$, both choices of base point result in the same $\Z[G]$-action on~$[D]^{\bullet}$.
\end{remark}

Gluing constructions as the one above have been formalized by Bar-Natan \cite{BN1} using the following tool.
\begin{dfn}
A \emph{$d$-input planar arc diagram} $\mathcal{D}$ is a disk (called \emph{output disk}), with $d$ enumerated open so-called \emph{input disks} removed from its interior; together with a proper smooth oriented 1-submanifold of~$\mathcal{D}$, with \emph{end points} on~$\partial \mathcal{D}$. Here, $\partial \mathcal{D}$ consists of the union of the~$\partial E$, with $E$ ranging over the input disks and the output disk. The number of end points on each such $\partial E$ is required to be even; if it is non-zero, then one of the end points is distinguished as base point of~$E$. An example can be seen below, and in \cref{fig:planararctwist}.
\end{dfn}

Let $\mathcal{D}$ be a $d$-input planar arc diagram with $2n_0$ end points on the output disk and $2n_i$ end points on the $i$-th input disk. By gluing tangle diagrams, $\mathcal{D}$ yields an operator
that takes as input $d$ base pointed tangle diagrams $D_1, \ldots, D_d$
that fit into the input disks, and that gives as output a base pointed tangle diagram $\mathcal{D}(D_1, \ldots, D_d)$.
For what follows, recall that for each~$n\geq 1$, we fixed a base point on the boundary of the disk containing the crossingless tangles of~$\Cob^3(2n)$.
By gluing crossingless tangle diagrams and cobordisms, $\mathcal{D}$ then gives a functor
\[
\prod_{i=1}^d \Cob^3(2n_i) \to \Cob^3(2n_0),
\]
which is compatible with modding out the relations~$l$.
By taking tensor products, this functor extends to a functor
\[
\prod_{i=1}^d \Kom(\Mat(\Cob^3_{\modl}(2n_i))) \to \Kom(\Mat(\Cob^3_{\modl}(2n_0))),
\]
which is compatible with homotopy equivalence.
Note that the orientations of the arcs and circles in $\mathcal{D}$ matter for the operator, but not for the functors.
We have the following compatibility result:
\begin{equation}\label{eq:planararcdiagram1}
\mathcal{D}([D_1], \ldots, [D_d]) \cong
[\mathcal{D}(D_1, \ldots, D_d)].
\end{equation}

Equipped with this tool set, one could give a more formal definition of the $\Z[G]$ action on $[D]^{\bullet}$ given in \cref{dfn:bnzg}, using a 2-input planar arc diagram whose two input disks have $2$ and $2n$ end points, respectively.

Moreover, for the following input diagram
\begin{center}$\mathcal{D} =$ \raisebox{-.4\height}{
\begingroup%
  \makeatletter%
  \providecommand\color[2][]{%
    \errmessage{(Inkscape) Color is used for the text in Inkscape, but the package 'color.sty' is not loaded}%
    \renewcommand\color[2][]{}%
  }%
  \providecommand\transparent[1]{%
    \errmessage{(Inkscape) Transparency is used (non-zero) for the text in Inkscape, but the package 'transparent.sty' is not loaded}%
    \renewcommand\transparent[1]{}%
  }%
  \providecommand\rotatebox[2]{#2}%
  \newcommand*\fsize{\dimexpr\f@size pt\relax}%
  \newcommand*\lineheight[1]{\fontsize{\fsize}{#1\fsize}\selectfont}%
  \ifx\svgwidth\undefined%
    \setlength{\unitlength}{89.90381556bp}%
    \ifx\svgscale\undefined%
      \relax%
    \else%
      \setlength{\unitlength}{\unitlength * \real{\svgscale}}%
    \fi%
  \else%
    \setlength{\unitlength}{\svgwidth}%
  \fi%
  \global\let\svgwidth\undefined%
  \global\let\svgscale\undefined%
  \makeatother%
  \begin{picture}(1,0.53390233)%
    \lineheight{1}%
    \setlength\tabcolsep{0pt}%
    \put(0,0){\includegraphics[width=\unitlength,page=1]{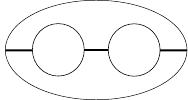}}%
    \put(-0.00300046,0.20855839){\makebox(0,0)[lt]{\lineheight{1.25}\smash{\begin{tabular}[t]{l}\textbf{*}\end{tabular}}}}%
    \put(0.14170149,0.20724699){\makebox(0,0)[lt]{\lineheight{1.25}\smash{\begin{tabular}[t]{l}\textbf{*}\end{tabular}}}}%
    \put(0.54719409,0.20724699){\makebox(0,0)[lt]{\lineheight{1.25}\smash{\begin{tabular}[t]{l}\textbf{*}\end{tabular}}}}%
  \end{picture}%
\endgroup%
},\end{center}
clearly $\mathcal{D}(D_1, D_2)$ is a diagram of the connected sum~$L_1 \# L_2$,
if $D_i$ is a 2-ended tangle diagram corresponding do the base pointed link~$L_i$,
in the sense of \cref{eq:2endlinks}. From $[\mathcal{D}(D_1, D_2)] \cong \mathcal{D}([D_1], [D_2])$,
it now follows that
\begin{equation}\label{eq:connectedsum}
\llbracket L_1 \# L_2 \rrbracket \cong \llbracket L_1 \rrbracket \otimes \llbracket L_2 \rrbracket.
\end{equation}

To adapt to $\Z[G]$-complexes, we consider planar arc diagrams
$\mathcal{D}$ satisfying the following condition:
$\mathcal{D}$ contains an arc connecting the base point of the output disk to the base point of the first input disk.
Then, $\mathcal{D}$ induces a functor
\[
\Kom(\Mat(\Cob^{3,\bullet}_{\modl}(2n_1))) \times \prod_{i=2}^d \Kom(\Mat(\Cob^3_{\modl}(2n_i))) \to \Kom(\Mat(\Cob^{3,\bullet}_{\modl}(2n_0))).
\]
For this functor, we have the following analog of \cref{eq:planararcdiagram1}:
\[
\mathcal{D}([D_1]^{\bullet}, [D_2], \ldots, [D_d]) \simeq
[\mathcal{D}(D_1, \ldots D_d)]^{\bullet}.
\]

\section{Properties of the $\lambda$-invariant}\label{sec:lambdaintro}
For the convenience of the reader, let us restate the definition of $\lambda$ from the introduction.
\dfnlambda*

Let us take this opportunity to clarify what we mean by `ungraded'.
\begin{dfn}\label{dfn:ungraded}
For chain complexes $(C,d), (C',d')$ in some additive category,
an \emph{ungraded} chain map $f\colon C\to C'$ 
is a morphism
\[
f\colon \bigoplus_{i=-\infty}^{\infty} C_i \to \bigoplus_{i=-\infty}^{\infty} C'_i
\]
that need \emph{not} respect homological degree, such that $d'\circ f = f \circ d$.
Whenever we want to highlight the difference, we call a chain map in the usual sense \emph{graded}.
If the underlying category is Abelian (so that one may take homology), then the ungraded chain map $f$ induces a morphism
\[
f_*\colon H(C) = \bigoplus_{i=-\infty}^{\infty} H_i(C) \to \bigoplus_{i=-\infty}^{\infty} H_i(C') = H(C').
\]
\end{dfn}
Some authors also call a chain complex without homological grading a \emph{differential module}.

\subsection{Basic properties and generalizations of $\lambda$}
First, let us extend the above definition of~$\lambda$.
\begin{dfn}\label{dfn:lambdageneral}
Abusing notation, we denote by $\lambda$ all of the following functions:
\begin{itemize}
\item $\lambda(A,B)$
for two chain complexes $A, B$ over~$\Z[G]$, or over $\Cob^{3,\bullet}_{\modl}(2n)$, is defined as
the minimal integer $k\geq 0$ such that there exist ungraded chain maps $f\colon A\to B$ and $g\colon B\to A$
and homotopies $g\circ f\simeq G^k\cdot \id_A$, $f\circ g\simeq G^k\cdot \id_B$, if such a $k$ exists, and $\infty$ otherwise.
\item $\lambda(A)$ for $A$ a chain complex over $\Z[G]$ is an abbreviation for $\lambda(A, \llbracket U\rrbracket)$.
\item $\lambda(D,D')$ for
$D$ and $D'$ two tangle diagrams in a fixed disk with the same end points and the same base point is
defined as $\lambda([D]^{\bullet}, [D']^{\bullet})$.
\item $\lambda(K, J)$ for $K$ and $J$ two knots is defined as $\lambda(\llbracket K\rrbracket, \llbracket J\rrbracket)$.
\end{itemize}
\end{dfn}
Note that for all knots~$K$, $\lambda(K)$ as in \cref{dfn:lambda} equals \[\lambda(K,U) = \lambda(\llbracket K\rrbracket, \llbracket U\rrbracket) = \lambda(\llbracket K\rrbracket)\] as in \cref{dfn:lambdageneral}.
\begin{remark}
One can naturally extend the definition of $\lambda$ from knots to links with base point, by setting
$\lambda(L, L')$ to be~$\lambda(T, T')$, where $T, T'$ are the 2-ended tangles corresponding
to the links $L, L'$ via \cref{eq:2endlinks}.
In this sense, most of this paper's results will generalize from knots to links. For simplicity's sake,
however, we are sticking with knots.
\end{remark}

Next, let us prove some useful basic properties of~$\lambda$.
\begin{lem}\label{lem:lambdadditive}
For some~$k\geq 1$, let $A_1,\ldots, A_k, B_1, \ldots, B_k$ be chain complexes over $\Z[G]$ or~$\Cob^{3,\bullet}_{\modl}(2n)$.
Then, for $A = A_1 \oplus \cdots \oplus A_k$ and $B = B_1 \oplus \cdots \oplus B_k$ one has
\[
\lambda(A, B) ~\leq~ \max(\lambda(A_1,B_1),\ldots, \lambda(A_k, B_k)).
\]
\end{lem}
\begin{proof}
Without loss of generality we can assume that~$k = 2$.
If either $\lambda(A_1,B_1)$ or $\lambda(A_2,B_2)$ are equal to $\infty$ the statement of the lemma is trivial, so let us assume that they are both finite. We pick chain maps $f_1,g_1$ such that $f_1 \circ g_1 \simeq G^{\lambda(A_1,B_1)} \cdot \id_{B_1}$ and $g_1 \circ f_1 \simeq G^{\lambda(A_1,B_1)} \cdot \id_{A_1}$, and choose maps $f_2,g_2$ similarly for~$\lambda(A_2,B_2)$.
Let $m = \max (\lambda(A_1,B_1), \lambda(A_2, B_2))$ and define $f\colon A_1\oplus A_2 \to B_1\oplus B_2$, $g\colon B_1\oplus B_2 \to A_1\oplus A_2$ as follows:
\[
f=
\begin{pmatrix}
G^{m-\lambda(A_1,B_1)} \cdot f_1 & 0  \\
0 & G^{m-\lambda(A_2,B_2)} \cdot f_2
\end{pmatrix}
, \quad g= 
\begin{pmatrix} 
g_1 & 0 \\
0 & g_2
\end{pmatrix}.
\]
We leave it to the reader to check that $f \circ g \simeq G^m \cdot \id_{B_1 \oplus B_2}$ and $g \circ f \simeq G^m \cdot \id_{A_1 \oplus A_2}$.
\end{proof}
Taking one of the $B_i$ as~$\llbracket U\rrbracket$, and all the others as~$0$,
we obtain the following special case of \cref{lem:lambdadditive},
which gives a useful upper bound for $\lambda$ of a direct sum.
\begin{corollary} \label{lem:lambdadirectsum}
Let $C^1,\ldots,C^n$ be chain complexes of $\Z[G]$-modules, fix a $j \in  \{1,\ldots,n\}$ and let~$l_k = \lambda(C^k,0)$, for all~$k \ne j$, and~$l_j = \lambda(C^j)$. Then:
\[
\lambda(\bigoplus_i C^i) \leq \max_{i} l_i.
\myqed
\]
\end{corollary}

\begin{lem} \label{lem:propertieslambda}
\begin{enumerate} [label=(\roman*)]
    \item $\lambda(A_1 \otimes A_2) \leq \lambda(A_1) + \lambda(A_2)$ for $A_1,A_2$ chain complexes of $\Z[G]$-modules.
    \item $\lambda(\overline{A})=\lambda(A)$, where $A$ is a $\Z[G]$-complex and $\overline{A}$ its dual.
\end{enumerate}
\end{lem}
\begin{proof}
For the first statement, let us assume that $\lambda(A_1),\lambda(A_2)$ are both finite (if either one is $\infty$ the statement is trivial). Let $f_i\colon A_i \to \llbracket U \rrbracket$, $g_i\colon \llbracket U \rrbracket \to A_i$ be chain maps such that $g_i \circ f_i \simeq G^{\lambda(A_i)} \cdot \id_{A_i}$ and $f_i \circ g_i \simeq G^{\lambda(A_i)} \cdot \id_{\llbracket U \rrbracket}$, for~$i=1,2$. Define $f\colon A_1 \otimes A_2 \to \llbracket U \rrbracket \otimes \llbracket U \rrbracket \cong \llbracket U \rrbracket$ and $g\colon \llbracket U \rrbracket \otimes \llbracket U \rrbracket \cong \llbracket U \rrbracket \to A_1 \otimes A_2$ as
\[
f=f_1 \otimes f_2, \qquad g=g_1 \otimes g_2.
\]
Then $g\circ f \simeq G^{\lambda(A_1)+\lambda(A_2)}\cdot \id_{A_1 \otimes A_2}$ and $f\circ g \simeq G^{\lambda(A_1)+\lambda(A_2)}\cdot \id_{\llbracket U \rrbracket}$, so $\lambda(A_1 \otimes A_2) \leq \lambda(A_1)+\lambda(A_2)$ as desired.

As for the second statement, it follows from the fact that if $f\colon A \to \llbracket U \rrbracket$, $g\colon \llbracket U \rrbracket \to A$ are chain maps such that $g \circ f \simeq G^{k} \cdot \id_{A}$ and $f \circ g \simeq G^{k} \cdot \id_{\llbracket U \rrbracket}$, then the induced dual chain maps $\overline{g}\colon \overline{A} \to \overline{\llbracket U \rrbracket} \cong \llbracket U \rrbracket$ and $\overline{f}\colon \overline{\llbracket U \rrbracket} \cong \llbracket U \rrbracket \to \overline{A}$ satisfy $\overline{f} \circ \overline{g} \simeq G^{k} \cdot \id_{\overline{A}}$ and $\overline{g} \circ \overline{f} \simeq G^{k} \cdot \id_{\llbracket U \rrbracket}$.
\end{proof}

\cref{prop:lambdabasicproperties} now follows directly from \cref{lem:propertieslambda}, since $\llbracket K \# J \rrbracket \cong \llbracket K \rrbracket \otimes \llbracket J \rrbracket$ (see \cref{eq:connectedsum}) and $\llbracket -K \rrbracket \cong \overline{\llbracket K \rrbracket}$.

\subsection{A closer look at $\lambda$ for tangles}\label{subsec:lambdatangles}
Here, we will again make use of planar arc diagrams, as introduced in \cref{subsect:bnzg}.
\begin{lem}\label{lem:lambdainputdiagram}
Let $\mathcal{D}$ be a 2-input planar arc diagram
containing an arc connecting the base points of the output disk and the first input disk.
Let $D_1$ and $D_1'$ be two tangle diagrams fitting into the first input disk,
and let $D_2$ be a tangle diagram fitting into the second input disk.
Then
\[
\lambda(\mathcal{D}(D_1, D_2), \mathcal{D}(D_1', D_2)) ~\leq~
\lambda(D_1, D_1').
\]
\end{lem}
\begin{proof}
If $\lambda(D_1,D_1')=\infty$ the statement is clear.
Suppose that $\lambda(D_1,D_1')=n\in \mathbb{N}$ and consider chain maps $f\colon [D_1]^{\bullet}\to[D_1']^{\bullet}$ and 
$g\colon [D_1']^{\bullet}\to[D_1]^{\bullet}$ satisfying
$f\circ g \simeq G^n \cdot\id_{[D_1']^{\bullet}}$
and
$g\circ f \simeq G^n \cdot\id_{[D_1]^{\bullet}}$.
Using the functor induced by~$\mathcal{D}$, we may define maps $\tilde f$ and $\tilde g$ as
\[
\begin{tikzcd}[column sep=5em]
\mathcal{D}([D_1]^{\bullet},[D_2]) \ar[r,"{\tilde f=\mathcal{D}(f,\id_{[D_2]})}",bend left=10]
&
\mathcal{D}([D_1']^{\bullet},[D_2]).
\ar[l,"{\tilde g =\mathcal{D}(g,\id_{[D_2]})}",bend left=10]
\end{tikzcd}
\]
These maps satisfy
\[
\tilde g \circ \tilde f = \mathcal{D}(g\circ f, \id_{D_2}) \simeq \mathcal{D}(G^n\cdot \id_{D_1}, \id_{D_2}) = G^n \cdot \id_{\mathcal{D}([D_1']^{\bullet}, [D_2])},
\]
and the analogous equality for~$\tilde g\circ \tilde f$. This shows the desired statement.
\end{proof}
\begin{figure}[t]
\centering
\def\svgwidth{0.8\textwidth}
{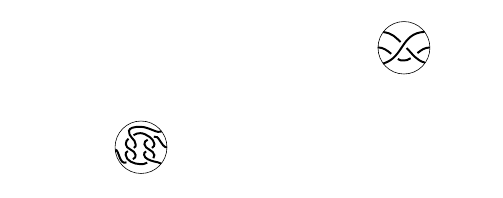}
\caption{Examples of the concepts introduced in \cref{dfn:addtwist}. An asterisk marks the base point.
(a)~$\mathcal{D}_4$, (b) a braid-like 6-ended tangle diagram~$Q$, (c) a 4-ended tangle diagram~$T$,
(d)~$\mathcal{D}_4(T,Q)$, called a braiding of $T$ using~$Q$.}
\label{fig:planararctwist}
\end{figure}
See \cref{fig:planararctwist} for examples of the following definitions.
\begin{dfn}\label{dfn:addtwist}
Let $\mathcal{D}_{2n}$ be the following 2-input planar arc diagram:
the two input disks are $2n$-ended and $(4n-2)$-ended, respectively;
$\mathcal{D}_{2n}$ consists of one arc connecting the base point of the output disk to the base point of the first input disk, $2n-1$ arcs connecting end points of the two input disks, and $2n-1$ arcs connecting end points of the second input disk to end points of the output disk. 

We say that a tangle diagram $Q$ with $2m$ end points is \emph{braid-like}, if it may be isotoped such that
$m$ end points are on the left, $m$ end points are on the right, and $Q$ consists of $m$ arcs that at no point have a vertical tangent.

For $\mathcal{D}_{2n}$ as above, $Q$ a $(4n-2)$-ended braid-like tangle diagram,
and $D$ a $2n$-ended tangle diagram, we say that $\mathcal{D}_{2n}(D, Q)$ is a \emph{braiding} of~$D$.
\end{dfn}
Recall from \cref{dfn:diagramoftangle} that to obtain a tangle diagram of a given tangle in a ball~$B$,
one must choose a homeomorphism between $B$ and the unit ball~$B_0$.
We will now show that two tangle diagrams of a fixed tangle are related by a finite sequence of Reidemeister moves and a braiding. In fact, the braiding only depends on the homeomorphisms between the balls, and not on the tangles. Let us make this precise.
\begin{lem}\label{lem:halftwists}
Let $B$ be a ball, and $P = \{p_1, \ldots, p_{2n}\}\subset \partial B$ for some~$n\geq 1$.
Let~$\varphi_1$, $\varphi_2$ be homeomorphisms from $B$ to the unit ball $B_0$ with $\varphi_1(P) = \varphi_2(P)$
and $\varphi_1(p_1) = \varphi_2(p_1)$.
Let $\mathcal{D}_{2n}$ be the 2-input planar arc diagram from \cref{dfn:addtwist}.
Then there is an unoriented braid-like $(4n-2)$-ended tangle diagram~$Q$,
such that for all tangles $T$ in $B$ with end points $P$ and base point $p_1$ the following holds:
if $D_1$ and $D_2$ are the tangle diagrams of $T$ coming from $\varphi_1$ and~$\varphi_2$, respectively,
then $\mathcal{D}_{2n}(D_1, Q)$ and $D_2$ are related by a finite sequence of Reidemeister moves and tangle diagram equivalences.
\end{lem}
\begin{proof}
Let $f\colon S^2\to S^2$ be the restriction of $\varphi_2 \circ \varphi_1^{-1}$ to~$S^2 = \partial B_0$.
Let us write $\tilde P = \varphi_1(P) = \varphi_2(P)$ and $\tilde p_i = \varphi_1(p_i)$.
Note $f(\tilde P) = \tilde P$.
In case that $f$ is isotopic to $\id_{S^2}$ along homeomorphisms fixing $\tilde P$ pointwise,
it follows that $\varphi_1(T)$ and $\varphi_2(T)$ are equivalent tangles, and thus $D_1$ and $D_2$
are related by a finite sequence of Reidemeister moves and tangle diagram equivalences.
To deal with general~$f$, let us consider the mapping class group of homeomorphisms $f\colon S^2\to S^2$ with
$f(\tilde P) = \tilde P$ and $f(\tilde p_1) = \tilde p_1$, up to isotopy along such maps.
Every such $f$ is isotopic to a homeomorphism fixing a neighborhood of $\tilde p_1$ pointwise,
and so this mapping class group is isomorphic to the mapping class group of the $(2n-1)$-punctured
disk, which is isomorphic to the braid group on $2n-1$ strands. More explicitly,
it is generated by $\sigma_1, \ldots, \sigma_{2n-2}$,
where $\sigma_i$ is a so-called \emph{half-twist}, switching the positions of the punctures $\tilde p_{i+1}$ and $\tilde p_{i+2}$~\cite[Section~9.1.3]{fm}.
So $f$ is isotopic to a product $\beta$ of the generators $\sigma_1, \ldots, \sigma_{2n-2}$.
Let $Q$ be a braid-like $(4n-2)$-ended tangle diagram corresponding to~$\beta$.
Then one sees that $\mathcal{D}_{2n}(D_1, Q)$ is a tangle diagram of $T$ coming from the homeomorphism
$(\varphi_2 \circ \varphi_1^{-1})\circ \varphi_1 = \varphi_2$. Therefore, $\mathcal{D}_{2n}(D_1, Q)$ and $D_2$
are related by a finite sequence of Reidemeister moves and tangle diagram equivalences, as desired.
\end{proof}
\begin{prop}\label{prop:lambdaequivariant}
Let $S$ and $T$ be tangles with the same end points and the same base point in a ball~$B$. Let $\varphi_1$ and $\varphi_2$ be homeomorphisms from $B$ to the unit ball~$B_0$, leading to tangle diagrams $D_{S1}, D_{S2}$ for $S$ and $D_{T1}, D_{T2}$ for~$T$, respectively.
Then
\[
\lambda(D_{S1}, D_{T1}) = \lambda(D_{S2}, D_{T2}).
\]
\end{prop}
\begin{proof}
By \cref{lem:halftwists}, there is a 2-input planar arc diagram $\mathcal{D}$ 
and a tangle $Q$ such that $\mathcal{D}(D_{S1}, Q)$ and $D_{S2}$ are related by a finite sequence of Reidemeister moves,
and so are $\mathcal{D}(D_{T1}, Q)$ and~$D_{T2}$.
By \cref{lem:lambdainputdiagram}, it follows that
$\lambda(D_{S2}, D_{T2}) \leq \lambda(D_{S1}, D_{T1})$.
Switching the roles of $\varphi_1$ and~$\varphi_2$, the opposite inequality also follows.
\end{proof}
As a consequence, the following is well-defined, since it does not depend on the choice of homeomorphism.
\begin{dfn}\label{dfn:lambdatangles}
Let $S$ and $T$ be tangles with the same end points and the same base point in a ball~$B$.
Then let $\lambda(S, T)$ be defined as $\lambda(D_S, D_T)$, where $D_S$ and $D_T$ are tangle diagrams of $S$ and~$T$, respectively, obtained via the same homeomorphism from $B$ to the unit ball~$B_0$.
\end{dfn}
\begin{prop}\label{prop:gluing}
Let $S$ and $T$ be two tangles in a ball $B$ with the same connectivity, base point and end points.
Let $R$ be a tangle in another ball~$B'$, and $\varphi\colon \partial B\to \partial B'$ an orientation-reversing homeomorphism sending end points to end points, such that $S\cup R$ and $T\cup R$ are knots in $B \cup_{\varphi} B' \cong S^3$. Then
\[
\lambda(S\cup R, T\cup R) \leq \lambda(S, T).
\]
\end{prop}
\begin{proof}
One may pick tangle diagrams~$D_S$, $D_T$ and $D_R$ for~$S$, $T$ and~$R$, respectively,
such that:
$D_S$ and $D_T$ come from the same homeomorphism from $B$ to~$B_0$;
gluing $D_S$ and $D_R$ (using a 2-input planar arc diagram~$\mathcal{D}$) results in a knot diagram of~$S\cup R$;
and similarly, $\mathcal{D}(D_T, D_R)$ is a diagram of~$T\cup R$.
Then, we have
\begin{align*}
\lambda(S\cup R, T\cup R) & = \lambda(\mathcal{D}(D_S,D_R), \mathcal{D}(D_T,D_R)) \\
 & \leq \lambda(D_S, D_T) \\
 & = \lambda(S,T)
\end{align*}
by the definition of $\lambda$ for knots, \cref{lem:lambdainputdiagram}, and \cref{dfn:lambdatangles}, respectively.
\end{proof}
\begin{prop}\label{prop:lambdametric}
Let us fix a ball and $2n$ end points on its boundary, and consider unoriented
tangles $T$ with a fixed number of components, a fixed connectivity, and a fixed base point in that ball.
On the set of equivalence classes of such tangles~$T$, $\lambda$~is a pseudometric.
\end{prop}
\begin{proof}
It is straight-forward to see that $\lambda$ is symmetric, satisfies the triangle inequality,
and~$\lambda(T,T) = 0$. Since any two tangles $S,T$ with same connectivity are related by crossing changes, \cref{thm:rationalreplacement} implies that~${\lambda(S,T) < \infty}$.
\end{proof}
Note that $\lambda(S,T) = 0$ if and only if $[S]^{\bullet}$ and $[T]^{\bullet}$ are ungradedly homotopy equivalent.
So the existence of non-equivalent tangles with homotopy equivalent Bar-Natan homology prevents $\lambda$ from being a metric. Still, this pseudometric allows for a nice formulation of the main step of the proof of \cref{thm:rationalreplacement}.
\begin{restatable}{prop}{propdiscrete}
\label{prop:discrete}
Fix a ball and four end points on its boundary, one of them distinguished as base point.
On the set of equivalence classes of unoriented rational tangles in that ball with fixed connectivity,
the pseudometric given by $\lambda$ is in fact equal to
the discrete metric. That is to say, $\lambda(S,T) = 1$ for inequivalent rational tangles $S$ and~$T$.
\end{restatable}
The proof of \cref{prop:discrete} will be given in \cref{sec:rational}.

\subsection{Decomposing $\Z[G]$-chain complexes into pieces} \label{subsec:pieces}

To analyze the $\Z[G]$-chain complex $\llbracket K\rrbracket$ of a knot $K$ and compute~$\lambda(K)$,
one may follow a divide-and-conquer strategy and decompose $\llbracket K\rrbracket$ as a direct sum.
This motivates the following definition.

\begin{dfn}
For a graded ring~$R$, a graded chain complex $C$ of free shifted $R$-modules of finite total rank,
i.e.~$C\in\Kom(\mathcal{M}_R)$, is called a \emph{piece} if it satisfies the following:
$C$ is not contractible (i.e.~not homotopy equivalent to the trivial complex),
and if $C$ is homotopy equivalent to $C'\oplus C''$ with $C', C'' \in \Kom(\mathcal{M}_R)$,
then either $C'$ or $C''$ is contractible.
In other words, a piece is an indecomposable object in the category $\Kom(\mathcal{M}_R)/h$ of chain complexes of finite total rank up to homotopy.
\end{dfn}
Let us now define the two most common kinds of pieces.
\begin{dfn}\label{dfn:pawnknight}
For a graded integral domain~$R$, let the \emph{pawn piece}, denoted by~$\sympawn$, be the chain complex consisting of just one copy $R$ in homological degree $0$ and quantum degree~$0$.
Given a non-trivial prime power~$z\in R$, we define the \emph{$z$-knight piece}, denoted by~$\symknight(z)$, to be the chain complex 
\[
_0R \xrightarrow{~z~} R \{ -\deg z \},
\]
where the left subscript denotes the homological degree.
\end{dfn}
The names of these pieces, coined by Bar-Natan \cite{BN1}, come from the fact that
a $\sympawn$ and $\symknight(G)$ piece in $\llbracket K\rrbracket$ result in the patterns
\[
\begin{array}{|c|} \hline \Q \\\hline \Q \\\hline \end{array} \qquad\text{and}\qquad
\begin{array}{|c|c|} \hline & \Q \\ \hline & \\ \hline \Q & \\\hline \end{array}\,,
\]
respectively, in unreduced rational Khovanov homology. This can be seen using \cref{cor:zg2kh}.
\begin{remark}
A complex $P$ is a piece if and only if the ring of endomorphisms of $P$ up to homotopy has precisely two distinct idempotents, namely the zero map and the identity map. Let us use this to check that pawns and knights actually are pieces.
For~$P = \sympawn$, the endomorphism ring of $P$ is isomorphic to~$R$, and there are no non-trivial homotopies.
Since $R$ is assumed to be an integral domain, the only idempotents are $0$ and~$1$, and~$0\neq 1$. So $\sympawn$ is indeed a piece.

Now consider~$P = \symknight(z)$. Ignoring the chain complex structure, $R$-module endomorphisms $P\to P$ are given
by $\left(\begin{smallmatrix}a & b\\c & d\end{smallmatrix}\right)$. Which among those maps are chain maps?
To respect homological degree, we must have~$b = c = 0$. To commute with the differential, we must have~$az = dz$.
Since $z\neq 0$ and $R$ is an integral domain, this implies~$a = d$. So the endomorphism ring of $P$ consists (as for~$\sympawn$)
just of multiples of~$\id_P$, i.e.~this ring is isomorphic to~$R$. All homotopies are  multiples of
$h\colon P_1\to P_0$,~$h(1)=1$. We have $h\circ d + d\circ h = z\cdot\id_P$, and so
the endomorphism ring of $P$ modulo homotopy is isomorphic to~$R/(z)$.
Since $z$ is by assumption a non-trivial prime power, $R/(z)$ has no non-trivial idempotents. Thus $\symknight(z)$ is indeed a piece.

In \cref{ex:t56} and \cref{sec:lambdacalc} below, we will claim that various chain complexes are pieces.
This may be checked by similar arguments as above; but since we don't actually make use of the fact that those complexes are pieces, we omit these arguments from the text.
\end{remark}

If $R$ is a graded PID, then pawns and knights are the only pieces.
This fact has been used previously to analyze homology theories coming from Frobenius algebras
over fields, e.g.~by Khovanov \cite{KH2},
or by Morrison \cite{morrison}%
\footnote{Morrison's `universal Khovanov homology' is equivalent to $\llbracket \,\cdot\,\rrbracket\otimes\Q$,
i.e.~the reduced theory coming from the Frobenius algebra $\Q[t, X] / (X^2 - tX)$ over~$\Q[t]$.
Since $\Q[t]$ is a PID, the chain complexes coming from that theory are homotopy equivalent to
a sum of $\sympawn$ and $\symknight(t^n)$ pieces, which Morrison calls $E$ and $C_n$ (or~$KhC[n]$), respectively.
This homology theory can be calculated with \texttt{JavaKh}~\cite{javakh}.}.
In the introduction, we have seen $\llbracket K\rrbracket$ for $K = U, T_{2,3}, T_{3,4}$,
and for those examples, $\llbracket K\rrbracket$ also decomposes as sum of pawns and knights.
Let us consider a further example, which demonstrates that the pieces of $\Z[G]$-chain complexes
can be significantly more complicated (in fact, we do not know a classification of those pieces,
cf.~\cite[Question~1]{zbMATH05118580}).

\begin{example}\label{ex:t56}
As one may compute with \texttt{khoca} and \texttt{homca},
the chain complex $\llbracket T_{5,6}\rrbracket$ is homotopy equivalent to the sum of
\[
_0\sympawn\{20\} ~\oplus~ {_2\symknight(G)\{24\}} ~\oplus~ {_4\symknight(G^2)\{26\}}
\]
and the following
four more complicated pieces (where we write~$R = \Z[G]$):
\medskip

\noindent\begin{tabular}{cc}
$P_1=$\begin{tikzcd}[row sep = 0pt, column sep=1.6em]
 _6R\{28\} \ar[r,"G"] & R\{30\} \\
 \oplus & \oplus \\
 R\{30\} \ar[r,"G",swap] \ar[ruu,"2",pos=.6]& R\{32\} \\
\end{tikzcd}, &
$P_2=$\begin{tikzcd}[row sep = 0pt, column sep=1.6em]
 & R\{34\} \ar[rd,"G",pos=.4] \\
_{8}R\{30\} \ar[ru,"2G^2",pos=.7]\ar[rd,"G^3",swap,pos=.6] & \oplus & R\{36\}, \\
 & R\{36\} \ar[ru,"-2",swap,pos=.3]
\end{tikzcd}
\\[9ex]
$P_3=$\begin{tikzcd}[row sep = 0pt, column sep=1.6em]
 & R\{36\} \ar[rd,"G^2",pos=.3] \\
_{10}R\{34\} \ar[ru,"5G",pos=.7]\ar[rd,"G^2",swap,pos=.6] & \oplus & R\{40\}, \\
 & R\{38\} \ar[ru,"-5G",swap,pos=.2]
\end{tikzcd} &
$P_4=$\begin{tikzcd}[row sep = 0pt, column sep=1.6em]
 & R\{40\} \ar[rd,"G",pos=.4] \\
_{12}R\{36\} \ar[ru,"3G^2",pos=.7]\ar[rd,"G^3",swap,pos=.6] & \oplus & R\{42\}. \\
 & R\{42\} \ar[ru,"-3",swap,pos=.3]
\end{tikzcd}
\end{tabular}
\bigskip

Note that $P_3$ is isomorphic to $_{10}\symknight(G^2) \otimes \symknight(5G)\{34\}$.
Let us now compute $\lambda$ of~$T_{5,6}$.
We have $\lambda(\symknight(G^k)) = k$ for $k\in\{1,2\}$ (in fact, for all~$k\geq 1$) and 
leave it to the reader to check that $\lambda(P_i, 0) \leq 3$ for $i\in\{1,2,3,4\}$.
Using \cref{lem:lambdadirectsum}, this implies $\lambda(T_{5,6}) \leq 3$.
To show $\lambda(T_{5,6}) \geq 3$, we rely on the maximal $G$-torsion order of homology, denoted by~$\mathfrak{u}_G$. This invariant is discussed in detail in the next \cref{subsec:torsionorders}. It gives a lower bound $\mathfrak{u}_G \leq \lambda$ (see \cref{lem:maxGtorsionVSlambda}).
Consider the homology of~$\overline{P_4}$, the dual of~$P_4$.
The annihilator of the class of a generator of ${_{-12}R\{-36\}}$ is the ideal $(3G^2,G^3) \subset \Z[G]$,
and so the $G$-torsion order of that homology class is equal to~$3$.
Hence $\lambda(T_{5,6}) = \lambda(-T_{5,6}) \geq \mathfrak{u}_G(-T_{5,6}) \geq 3$, and thus $\lambda(T_{5,6}) = 3$.
\end{example}

\begin{remark}\label{rmk:nonuniq}
If $R$ is Noetherian, then every chain complex in $\Kom(\mathcal{M}_R)/h$ can be written as a sum of finitely many pieces.
If $R$ is a graded PID, then this decomposition is essentially unique, i.e.~unique up to the order of the summands.
This is not true for~$R = \Z[G]$, as the following example demonstrates.
So, in this text we will often 
decompose chain complexes $\llbracket K\rrbracket$ as sums of pieces, but we will never rely on this decomposition being unique.

Let us now give an example of a chain complex that admits two essentially different decompositions as sums of pieces.
For any non-zero integer~$n$, let $Q_n$ be the complex
\[
\begin{tikzcd}[row sep = 0pt, column sep=1.6em]
 & R\{2\} \ar[rd,"-G",pos=.3] \\
_{0}R\{0\} \ar[ru,"nG",pos=.7]\ar[rd,"G^2",swap,pos=.6] & \oplus & R\{4\}. \\
 & R\{4\} \ar[ru,"n",swap,pos=.2]
\end{tikzcd}
\]
One computes that the endomorphism ring of $Q_n$ modulo homotopy is isomorphic to~$R/(G^2,nG)$. This ring does not admit non-trivial idempotents, and so $Q_n$ is a piece.
Now, the Smith normal form gives us invertible $2\times 2$ integer matrices $S,T$ such that $S\left(\begin{smallmatrix}2 & 0\\0 & 3\end{smallmatrix}\right)T = \left(\begin{smallmatrix}1 & 0\\0 & 6\end{smallmatrix}\right)$.
This leads to the following change of basis, which demonstrates $Q_2 \oplus Q_3 \cong Q_1 \oplus Q_6$, giving us the desired example. Note that~$Q_1 \simeq \symknight(G)$.
\newcommand{\sbm}[1]{{\let\amp=&\left(\begin{smallmatrix}#1\end{smallmatrix}\right)}}
\[
\begin{tikzcd}[row sep = 0pt, column sep=2em, ampersand replacement=\&]
 \& R\{2\}^{\oplus 2} \ar[rd,"-G",pos=.3] \\
_{0}R\{0\}^{\oplus 2} \ar[ru,"G\sbm{2 & 0\\0 & 3}",pos=.7]\ar[rd,"G^2",swap,pos=.6] \& \oplus \& R\{4\}^{\oplus 2} \\
 \& R\{4\}^{\oplus 2} \ar[ru,"\sbm{2 & 0\\0 & 3}",swap,pos=.2]
\end{tikzcd}\ \cong\ 
\begin{tikzcd}[row sep = 0pt, column sep=2em, ampersand replacement=\&]
 \& R\{2\}^{\oplus 2} \ar[rd,"-G",pos=.3] \\
_{0}R\{0\}^{\oplus 2} \ar[ru,"GS\sbm{2 & 0\\0 & 3}T",pos=.7]\ar[rd,"G^2",swap,pos=.6] \& \oplus \& R\{4\}^{\oplus 2} \\
 \& R\{4\}^{\oplus 2} \ar[ru,"S\sbm{2 & 0\\0 & 3}T",swap,pos=.2]
\end{tikzcd}
\]
\end{remark}
\subsection{Torsion orders}\label{subsec:torsionorders}
When computing $\lambda$ of a knot $K$ it is fairly simple to find an upper bound $k \geq \lambda(K)$ by defining ungraded chain maps $f \colon \llbracket K \rrbracket \to \llbracket U \rrbracket$, $g \colon \llbracket U \rrbracket \to \llbracket K \rrbracket$ such that $g \circ f$ and $f \circ g$ are homotopic to~$G^k$. In order to compute the exact value of $\lambda$ however, one has to find the minimal such~$k$, which can be a hard task. The invariants described in this subsection give lower bounds for $\lambda$ in terms of the maximal torsion order in homology.

In 2017, Alishahi and Dowlin \cite{zbMATH07178864,zbMATH07005602} introduced the following knot invariants which are lower bounds for the unknotting number: $\mathfrak{u}_h$ is defined as the maximal order of $h$-torsion in the unreduced homology with Frobenius algebra $\mathbb{F}_2[h,X]/(X^2+hX)$, while $\mathfrak{u}_X$ is the maximal $X$-torsion order in the unreduced homology $H_{\Q[t]}$ with Frobenius algebra $\Q[t,X]/(X^2-t)$ (a lift of Lee homology). Note that $H_{\Q[t]}$ is a module over $\Q[t,X]/(X^2-t)$, but that ring is just equal to its subring $\Q[X]$, and so we consider $H_{\Q[t]}$ as a module over $\Q[X]$.
It was then remarked in \cite{CGLLSZ} that for the latter invariant one can replace $\Q$ with $\mathbb{F}_p$ for any odd prime $p$ in order to obtain new bounds~$\mathfrak{u}_{(X,p)}$.
Finally, Gujral \cite{gujral} introduced a lower bound $\nu$ for the ribbon distance:
$\nu$
is the maximal order of $(2X-(\alpha_1+\alpha_2))$-torsion in the $\alpha$-homology of a knot,
which is the unreduced homology with Frobenius algebra $\mathbb{Z}[X,\alpha_1,\alpha_2]/((X-\alpha_1)(X-\alpha_2))$ over the ground ring $\mathbb{Z}[\alpha_1,\alpha_2]$, introduced in~\cite{frob2}.

The following invariant is the analog of those bounds in the $\Z[G]$-setting.
\begin{dfn} 
Let $K$ be a knot and consider its $\Z[G]$-complex $\llbracket K \rrbracket \in \Kom(\mathcal{M}_{\Z[G]})$. Since $\Kom(\mathcal{M}_{\Z[G]}) \subset \Kom(\Z[G]\textrm{-Mod})$ we can take the homology $H(\llbracket K \rrbracket)$ of $\llbracket K \rrbracket$ in the latter category (see \cref{rem:tqft}). Let now $a \in H(\llbracket K \rrbracket)$. We say that $a$ is $G$\emph{-torsion} if there is an $n \in \Z_{\geq 0}$ such that~$G^n \cdot a = 0$. 
Let the \emph{order} of a $G$-torsion element~$a$, $\text{ord}_G(a)$, be the minimal such $n$ and $T\left( H(\llbracket K \rrbracket)\right)$ the $\Z[G]$-module of $G$-torsion elements.
\end{dfn}

\begin{dfn} \label{def:maxGtorsion}
We define $\mathfrak{u}_G(K)$ to be the maximal order of a $G$-torsion element:
\[
\mathfrak{u}_G(K)\coloneqq\max_{a \in T \left( H(\llbracket K \rrbracket) \right) }\text{ord}_G(a).
\]
\end{dfn}

\begin{prop}
For all knots $K$, we have\medskip

\mbox{}
\hfill(i) \ ${\mathfrak{u}_G = \nu}$,
\hfill(ii) \ ${\mathfrak{u}_G \geq \mathfrak{u}_X}$,
\hfill(iii) \ ${\lambda \geq \mathfrak{u}_{(X,p)}}$,
\hfill(iv) \ ${\lambda\geq \mathfrak{u}_h}$.
\end{prop}
\begin{proof}
(i) By \cref{thm:equiv}, the $\alpha$-chain complex of $K$ is homotopy equivalent to
$\llbracket K \rrbracket \otimes_{\Z[G]} \Z[X,\alpha_1,\alpha_2]/((X-\alpha_1)(X-\alpha_2))$,
where $G$ acts as $2X - \alpha_1 - \alpha_2$ on the second tensor factor.
Less formally, $\alpha$-homology is just $\Z[G]$-homology with $G$ disguised as $2X - \alpha_1 - \alpha_2$.
The statement follows.

(ii) Again by \cref{thm:equiv},
$C_{\Q[t]}(K) \simeq \llbracket K\rrbracket \otimes_{\Z[G]} \Q[X]$,
where $G$ acts as $2X$ on~$\Q[X]$. So $X$-torsion in $H_{\Q[t]}(K)$ corresponds to
$G$-torsion in $H(\llbracket K\rrbracket \otimes \Q)$.
Thus it suffices to show that the maximal $G$-torsion order in $\Z[G]$-homology
is greater than or equal to the maximal $G$-torsion order in $\Q[G]$-homology.
So, let a cycle $x$ in the chain complex $\llbracket K\rrbracket \otimes \Q$
be given that represents a homology class of maximal $G$-torsion order in $\Q[G]$-homology,
i.e.~$G^{\mathfrak{u}_X(K)} [x] = 0$ and $G^{\mathfrak{u}_X(K) - 1} [x] \neq 0$.
Choose $n\in\mathbb{Z}\setminus\{0\}$ such that $nx \in \llbracket K\rrbracket$.
Since $G^{\mathfrak{u}_X(K)} x$ is a boundary over $\mathbb{Q}$, one may choose
$m\in\mathbb{Z}\setminus\{0\}$ such that $G^{\mathfrak{u}_X(K)} nmx$ is a boundary over $\mathbb{Z}$.
So $nm[x]$ is $G$-torsion in $H(\llbracket K\rrbracket)$.
Moreover, if $G^{\mathfrak{u}_X(K) - 1} nmx$ were a boundary $dy$ in $\llbracket K\rrbracket$,
then $G^{\mathfrak{u}_X(K) - 1} x = d(y/(nm))$ and thus $G^{\mathfrak{u}_X(K) - 1} [x] = 0 \in H(\llbracket K\rrbracket\otimes\Q)$ would follow, which is a contradiction. So the $G$-torsion order of $nm[x] \in H(\llbracket K\rrbracket)$ equals $\mathfrak{u}_X(K)$, which implies the claim $\mathfrak{u}_G(K) \geq \mathfrak{u}_X(K)$.

(iii)
Let $f\colon \llbracket K\rrbracket \to \Z[G]$ and $g\colon \Z[G] \to \llbracket K\rrbracket$
be chain maps such that $f\circ g$ and $g\circ f$ are homotopic to $G^{\lambda(K)}$ times
the identity of $\Z[G]$ and $\llbracket K\rrbracket$, respectively.
Once again by \cref{thm:equiv},
$C_{\mathbb{F}_p[t]}(K) \simeq \llbracket K\rrbracket \otimes_{\Z[G]} \mathbb{F}_p[X]$,
where $G$ acts as $2X$ on~$\mathbb{F}_p[X]$.
So $f$ and $g$ induce maps $f_p\colon H_{\mathbb{F}_p[t]}(K) \to \mathbb{F}_p[X]$
and $g_p\colon \mathbb{F}_p[X] \to H_{\mathbb{F}_p[t]}(K)$
such that $f_p\circ g_p$ and $g_p\circ f_p$ are multiplication with $(2X)^{\lambda(K)}$
on $\mathbb{F}_p[X]$ and $H_{\mathbb{F}_p[t]}(K)$, respectively.
The existence of such maps implies, in the usual way (see e.g.~the proof of \cref{lem:maxGtorsionVSlambda} below) that $X$-torsion orders in $H_{\mathbb{F}_p[t]}(K)$ are at most $\lambda(K)$.

(iv) This inequality may be proven similarly as the previous one,
using that by \cref{thm:equiv},
$C_{\mathbb{F}_2[h]}(K) \simeq \llbracket K\rrbracket \otimes_{\Z[G]} \mathbb{F}_2[h]$,
where $G$ acts as $h$ on~$\mathbb{F}_2[h]$.
\end{proof}
Note that the inequalities $\mathfrak{u}_G \geq \mathfrak{u}_{(X,p)}$ and $\mathfrak{u}_G \geq \mathfrak{u}_h$ do not hold in general; counterexamples are given by the complexes $P_4$ (for $\mathfrak{u}_{(X,3)}$) and $P_2$ (for $\mathfrak{u}_h$) from \cref{ex:t56}.

As a side note it is worthwhile to observe, although we  will not make use of it, that $\mathfrak{u}_h,\mathfrak{u}_X$ are also linked to the convergence of the Bar-Natan \cite{tur} and the Lee \cite{Lee} spectral sequences respectively. For all knots~$K$, these sequences start at Khovanov homology of $K$ (with coefficients in $\mathbb{F}_2$ and $\Q$ respectively) and, letting $n_{\textrm{BN}}$ and $n_{\textrm{Lee}}$ be the pages at which they collapse, we have $\mathfrak{u}_h(K)=n_{\textrm{BN}}-1$ and $\lceil \mathfrak{u}_X(K)/2\rceil=n_{\textrm{Lee}}-1$. As a consequence, the following interesting result holds:
\begin{corollary}[\cite{zbMATH07005602}]
The Knight Move Conjecture is true for all knots $K$ with~$u(K) \leq 2$.
\end{corollary}
In light of \cref{thm:rationalreplacement}, we even have:
\begin{corollary}
The Knight Move Conjecture is true for all knots $K$ with\break ${u_q(K) \leq 2}$.\qed
\end{corollary}

The connection discussed above between invariants related to $\lambda$ and spectral sequences
brings us to the following natural question:
\begin{question}
Is there a spectral sequence $E_G$ such that for any knot~$K$, $E_G(K)$ starts at Khovanov homology (with coefficients in~$\Z$) and collapses at a page whose number is determined by~$\mathfrak{u}_G(K)$?
\end{question}
We suspect this question has a positive answer. Namely, consider the chain complex obtained from $\llbracket K \rrbracket$ by setting~$G=1$. The resulting complex is filtered, and gives rise to a spectral sequence $E_G(K)$ starting at (reduced) Khovanov homology with integer coefficients. It seems likely that $E_G(K)$ collapses at page~$\mathfrak{u}_G(K)-1$.

\begin{lem} \label{lem:lambdadetectsU}
The invariant $\mathfrak{u}_G$ detects the unknot, i.e.~$\mathfrak{u}_G(K) = 0$ holds if and only if $K$ is trivial.
\end{lem}

\begin{proof}
We start by noticing that~$\mathfrak{u}_G(U)=0$: this is clear since $H(\llbracket U \rrbracket)=\Z[G]$ is torsion free. 
Note that $H_{\Q[t]}(K) \cong \Q[X] \oplus T \left( H_{\Q[t]}(K) \right)$, where $T \left( H_{\Q[t]}(K) \right)$ is the $X$-torsion part. Since Khovanov homology detects the unknot we have $T \left( H_{\Q[t]}(K) \right) = \{ 0 \}$ only if~$K=U$. This implies that if $K$ is not the unknot, then $\mathfrak{u}_G(K) \geq \mathfrak{u}_X(K) > 0$. 
\end{proof}

\begin{lem} \label{lem:maxGtorsionVSlambda}
Let $K$ be a knot. Then $\mathfrak{u}_G(K) \leq \lambda(K).$
\end{lem}

\begin{proof}
Let $n=\lambda(K)$ and let $f \colon \llbracket K \rrbracket \to \llbracket U \rrbracket$, $g \colon \llbracket U \rrbracket \to \llbracket K \rrbracket$ be ungraded chain maps such that $g \circ f \simeq G^n \cdot \id_{\llbracket K \rrbracket}$ and $f \circ g \simeq G^n \cdot \id_{\llbracket U \rrbracket}$. Then, for every~$a \in H(\llbracket K \rrbracket)$: 
\[
\text{ord}_G \left( f_{*}(a) \right) \geq \text{ord}_G \left( g_{*} \circ f_{*}(a) \right) =\text{ord}_G \left( G^n \cdot a \right) \geq \text{ord}_G \left( a \right)  - n.
\]
Taking the maximum over $T\left( H(\llbracket K \rrbracket)\right)$ we get:
\[
0 = \mathfrak{u}_G(U) = \max_{a \in T \left( H(\llbracket K \rrbracket)  \right)} \text{ord}_G \left( f_{*}(a) \right) \geq \max_{a \in T \left( H(\llbracket K \rrbracket)  \right)} \text{ord}_G \left( a \right) - n = \mathfrak{u}_G (K) - n.
\]
This shows that $\mathfrak{u}_G (K) \leq n = \lambda (K)$.
\end{proof}
The two previous lemmas, combined with the fact that clearly $\lambda(U)=0$ (using the maps $f=g=\id_{\llbracket U \rrbracket}$), show that $\lambda$ detects the unknot, as claimed in \cref{prop:lambdadetectsunknot}.
A more direct proof may also be given as follows:
\begin{proof}[Proof of \cref{prop:lambdadetectsunknot}]
Let $K$ be a knot. It follows from the definition of $\lambda$ that $\lambda(K)=0$ if and only if $\llbracket K \rrbracket$ is ungraded chain homotopy equivalent to~$\llbracket U \rrbracket$. Since, even as an ungraded module, Khovanov homology detects the unknot, the latter condition holds if and only if~$K=U$.
\end{proof}

In \cref{dfn:lambdageneral} we saw how to define $\lambda$ for any $\Z[G]$-complex. Similarly, one can define the invariants $\mathfrak{u}_G, \mathfrak{u}_h, \mathfrak{u}_X, \mathfrak{u}_{(X,p)}$ on chain complexes over~$\mathcal{M}_R$, where $R$ is equal to $\Z[G],\mathbb{F}_2[h],\Q[X],\mathbb{F}_p[X]$ respectively. Namely, if $C \in \Kom(\mathcal{M}_R)$,  and $\eta = G, h, X$ and~$(X,p)$, then $\mathfrak{u}_{\eta}(C)$ is the maximal order of $\eta$-torsion in the homology $H(C) \in R \textrm{-Mod}$ of~$C$. 
\cref{lem:maxGtorsionVSlambda} also holds for complexes over~$\mathcal{M}_{\Z[G]}$.

We now state a few properties of $\mathfrak{u}_G,\mathfrak{u}_h,\mathfrak{u}_X,\mathfrak{u}_{(X,p)}$.
\begin{lem} \label{lem:propertiesuG}
Let $R_{G}=\Z[G], R_h=\mathbb{F}_2[h], R_X=\Q[X], R_{(X,p)}=\mathbb{F}_p[X]$, and let $A,B$ be chain complexes over~$\mathcal{M}_{R_{\eta}}$, for $\eta=G,h,X$ and~$(X,p)$. We have:
\begin{enumerate} [label=(\roman*)]
    \item $\mathfrak{u}_{\eta}(A \oplus B) = \max ( \mathfrak{u}_{\eta}(A), \mathfrak{u}_{\eta}(B) )$.
    \item For $\eta=h,X$ and $(X,p)$,
    \[
    \mathfrak{u}_{\eta}(A \otimes B) =
    \begin{cases*}
        \max (\mathfrak{u}_{\eta}(A), \mathfrak{u}_{\eta}(B)) & if $p_A > 0$ and $p_B > 0$, \\
        \mathfrak{u}_{\eta}(A) & if $p_A = 0$ and $p_B > 0$, \\ 
        \mathfrak{u}_{\eta}(B) & if $p_A > 0$ and $p_B = 0$, \\ 
        \min (\mathfrak{u}_{\eta}(A), \mathfrak{u}_{\eta}(B))  & if $p_A = p_B = 0$,
    \end{cases*}
    \]
    where, given a complex $C$ over~$\mathcal{M}_{R_{\eta}}$, $p_C$ is the number of pawn summands in the decomposition of $C$ into pawns and knights.
\end{enumerate}
\end{lem}

\begin{proof}
The first statement is clear. For the second one, we use the fact that $R_{\eta}$ is a PID. As noted in \cref{subsec:pieces}, this implies that
\[
A \simeq \sympawn^{\oplus p_A} \oplus \symknight(a_1)\oplus \dots \oplus \symknight(a_n),
\]
with $a_i=X^{k_i}$ (if $\eta = X$ or~$(X,p)$) or $a_i=h^{k_i}$ (if~$\eta = h$) and~$k_i \leq k_{i+1}$. Similarly,
\[
B \simeq \sympawn^{\oplus p_B} \oplus \symknight(b_1)\oplus \dots \oplus \symknight(b_m),
\]
with $b_j=X^{l_j}$ or $b_j=h^{l_j}$ and~$l_j \leq l_{j+1}$. One checks that $\mathfrak{u}_{\eta}(\sympawn)=0$ and $\mathfrak{u}_{\eta} (\symknight(a_i))=k_i$, therefore, by point (i), $\mathfrak{u}_{\eta} (A) = \max (\mathfrak{u}_{\eta}(\sympawn),\mathfrak{u}_{\eta} (\symknight(a_1)),\ldots,\mathfrak{u}_{\eta} (\symknight(a_n))) = k_n$. Similarly, $\mathfrak{u}_{\eta} (B)=l_m$. Now
\[
A \otimes B \simeq \bigoplus_{\substack{i \in \{0,\ldots,n\} \\ j \in \{0,\ldots,m\}}} A_i \otimes B_j,
\]
with $A_0=\sympawn^{\oplus p_A}$, $B_0=\sympawn^{\oplus p_B}$ and $A_i=\symknight(a_i)$, $B_j=\symknight(b_j)$ for $i,j>0$. It is a simple exercise to check that $\mathfrak{u}_{\eta}(\symknight(a_i) \otimes \symknight(b_j))=\min (k_i,l_j)$. It follows that 
\[
\begin{split}
    \mathfrak{u}_{\eta} (A \otimes B) &= \max_{i,j} (\mathfrak{u}_{\eta}(A_i \otimes B_j)) \\
    &= \max (\mathfrak{u}_{\eta} (A_0 \otimes B_m), \mathfrak{u}_{\eta} (A_n \otimes B_0), \mathfrak{u}_{\eta} (A_n \otimes B_m) ) \\
    &= \max (\mathfrak{u}_{\eta} (\sympawn^{\oplus p_A} \otimes \symknight(b_m)), \mathfrak{u}_{\eta} (\symknight(a_n) \otimes \sympawn^{\oplus p_B}), \mathfrak{u}_{\eta} (\symknight(a_n) \otimes \symknight(b_m)) ).
\end{split}
\]
Then statement (ii) can be deduced from the following:
\begin{gather*}
    \mathfrak{u}_{\eta} (\sympawn^{\oplus p_A} \otimes \symknight(b_m)) = \begin{cases*}
        \mathfrak{u}_{\eta} (\symknight(b_m))=l_m & if $p_A > 0$, \\
        0 & if $p_A=0$,
    \end{cases*} \\
    \mathfrak{u}_{\eta} (\symknight(a_n) \otimes \sympawn^{\oplus p_B}) = \begin{cases*}
        \mathfrak{u}_{\eta} (\symknight(a_n))=k_n & if $p_B > 0$, \\
        0 & if $p_B=0$,
    \end{cases*} \\
    \mathfrak{u}_{\eta} (\symknight(a_n) \otimes \symknight(b_m)) = \min (k_n,l_m).
\myqed
\end{gather*}
\end{proof}

\begin{remark}
As a consequence of (ii) of \cref{lem:propertiesuG}, if $\eta=h,X$ or $(X,p)$ and $K$ and $J$ are two knots, 
\begin{equation}\label{eq:eta-torsion_conn_sum}
    \mathfrak{u}_{\eta}(K \# J) = \max (\mathfrak{u}_{\eta}(K), \mathfrak{u}_{\eta}(J)),
\end{equation}
as the unreduced chain complexes over $\mathcal{M}_{R_\eta}$ associated to knots always contain a pawn summand (more precisely, exactly one pawn summand if $\eta = X$ or $(X,p)$, and exactly two if~$\eta = h$). 

Note that neither statement (ii) of \cref{lem:propertiesuG}, nor equation \eqref{eq:eta-torsion_conn_sum} hold in general for~$\mathfrak{u}_G$. This is due to the fact that, for~$\eta \ne G$, the $\mathfrak{u}_{\eta}$ are defined over PIDs, while $\mathfrak{u}_G$ is not (cf.~\cref{subsec:pieces}).
Later on in this article, \cref{rem:examplesKand-K,rem:u_GofKand-K} will provide us with examples of knots $K,J$ such that $\mathfrak{u}_{G}(K \# J) < \max (\mathfrak{u}_{G}(K), \mathfrak{u}_{G}(J))$, and others where $\mathfrak{u}_{G}(K \# J) > \max (\mathfrak{u}_{G}(K), \mathfrak{u}_{G}(J))$: if $K_1,K_2$ satisfy (1) of \cref{prop:lambdaKand-K} and $J$ satisfies (2), then $\mathfrak{u}_G (K_1) = \mathfrak{u}_G (K_2) = \mathfrak{u}_G (J) =1$, but $\mathfrak{u}_G(K_1 \# K_2) = 2 = \mathfrak{u}_G (K_1) + \mathfrak{u}_G (K_2)$ and $\mathfrak{u}_G((K_1 \# K_2) \# J) = 1 = \mathfrak{u}_G (K_1 \# K_2) - \mathfrak{u}_G (J)$. Therefore, the best that we can hope for is that $\mathfrak{u}_G (K \# J) \leq \mathfrak{u}_{G}(K) + \mathfrak{u}_{G}(J)$.
\end{remark}

\subsection{$\lambda$ of thin knots}
\label{subsec:thin}
For a homogeneous element $x$ of a doubly graded chain complex, with quantum degree $q$ and homological degree~$t$,
let the $\delta$-degree of $x$ be~$q - 2t$.
It was noted early on in the development of Khovanov homology that the rational Khovanov homology of alternating links is supported in a single $\delta$-degree~\cite{Lee}.
This led to various different notions of thinness and homological width of links, see e.g.~\cite{zbMATH06292128,1806.05168}.
In this article, we call a knot \emph{thin} if
their reduced integral Khovanov homology consists of free modules supported in a single $\delta$-degree (as already defined in the introduction). Let us prove in this subsection that $\lambda$ of non-trivial thin knots is~$1$.

\begin{lem}\label{lem:lambdaofknights}
If a chain complex $C \in \Kom(\mathcal{M}_{\Z[G]})$
decomposes (ignoring gradings) as a sum of one $\sympawn$ and finitely many $\symknight(G)$ pieces,
then~$\lambda(C) \leq 1$.
\end{lem}
\begin{proof}
Since $\lambda(\sympawn) = 0$ and $\lambda(\symknight(G),0) = 1$, this follows from \cref{lem:lambdadirectsum}.
\end{proof}

\begin{lem}\label{lem:minimalchaincomplex}
Let $K$ be a knot whose reduced integral Khovanov homology is torsion free.
Then $\llbracket K\rrbracket$ is homotopy equivalent to a chain complex $C \in \Kom(\mathcal{M}_{\Z[G]})$
of free shifted $\Z[G]$-modules, such that the Poincaré polynomial of $C$
is equal to the Poincaré polynomial of reduced integral Khovanov homology of~$K$.
\end{lem}
\begin{proof}
Start by picking an arbitrary chain complex $C' \in \Kom(\mathcal{M}_{\Z[G]})$ that is
homotopy equivalent to~$\llbracket K\rrbracket$.
Consider the chain complex $C' \otimes_{\Z[G]} \Z[G]/(G)$. This is a chain complex over the integers,
whose homology is isomorphic to reduced integral Khovanov homology of~$K$. In particular, it has torsion free homology by assumption.
One may select bases for the chain groups of the complex $C' \otimes_{\Z[G]} \Z[G]/(G)$,
with respect to which the matrices of the differentials are in Smith normal form.
Because homology is torsion free, all the entries of these matrices are $0$ or~$1$.
Gaussian elimination (see e.g.~\cref{lem:gaussian}) of all the entries equal to $1$ yields a homotopy equivalence between $C' \otimes_{\Z[G]} \Z[G]/(G)$
and a complex $Z$ with trivial differentials. So, $Z$ is isomorphic to the
reduced integral Khovanov homology of~$K$.

Now, one may lift the bases of $C' \otimes_{\Z[G]} \Z[G]/(G)$ to homogeneous bases of~$C'$.
Since the matrices of the differentials of $C'$ have homogeneous entries,
it follows that if a matrix entry of a differential of $C' \otimes_{\Z[G]} \Z[G]/(G)$ equals 1,
then the corresponding matrix entry of the corresponding differential of $C'$ also equals 1.
Therefore, one may lift the homotopy equivalence constructed above, obtaining
a homotopy equivalence between $C'$ and a complex $C \in \Kom(\mathcal{M}_{\Z[G]})$,
such that $C \otimes_{\Z[G]} \Z[G]/(G)$ is isomorphic to~$Z$.
It follows that $C$ and the reduced integral Khovanov homology of $K$ have the same Poincaré polynomial,
as desired.
\end{proof}

\begin{lem}\label{lem:thinknights}
For all thin knots~$K$, $\llbracket K\rrbracket$ is up to degree shifts
homotopy equivalent to a sum of one $\sympawn$ piece
and finitely many $\symknight(G)$ pieces.
\end{lem}
\begin{proof}
By \cref{lem:minimalchaincomplex}, we may pick a chain complex $C \in \Kom(\mathcal{M}_{\Z[G]})$
that is homotopy equivalent to~$\llbracket K\rrbracket$,
and has the same Poincaré polynomial as reduced integral Khovanov homology of~$K$.
Since the latter is supported on a single $\delta$-degree, so is~$C$.
Choosing arbitrary bases for the chain modules of~$C$, it follows that every entry
of the matrices of the differentials is an integer multiple of~$G$.
Similarly as in the proof of \cref{lem:minimalchaincomplex}, one may choose new bases
for the modules of~$C$, such that the matrices of the differentials equal
$G$ times a matrix in Smith normal form. Consequently, ignoring gradings $C$ decomposes as
a direct sum of $\sympawn$ and $\symknight(aG)$ pieces, with a priori varying~$a\in \Z_{>0}$.
By \cref{prop:redzghomdim}, there is exactly one $\sympawn$ piece, and all other pieces
are $\symknight(G)$ pieces.
\end{proof}

\propthin*
\begin{proof}
\cref{lem:thinknights} and \cref{lem:lambdaofknights} imply~$\lambda(K) \leq 1$,
whereas \cref{prop:lambdadetectsunknot} implies~$\lambda(K) \geq 1$.
\end{proof}
\begin{remark}
Note that \cref{lem:thinknights} also provides a proof (at least for knots)
for Bar-Natan's `structural conjecture' that all alternating links are `Khovanov basic' \cite[Conjecture~1]{BN1}.
\end{remark}

In \cite{CGLLSZ}, upper bounds for $\mathfrak{u}_X, \mathfrak{u}_h$ and $\mathfrak{u}_{(X,p)}$
are given in terms of the homological width of Khovanov homologies. This motivates the following question.
\begin{question}
Let $K$ be a knot such that $\llbracket K \rrbracket$ is homotopy equivalent to a complex
supported in $n$ adjacent $\delta$-degrees. Does then $\lambda(K) \leq n$ follow?
\end{question}

\section{Calculations of $\Z[G]$-homology and the $\lambda$-invariant}\label{sec:lambdacalc}
\subsection{Proof of \refinsec{Theorem}{thm:lambdahigh}: $\lambda$ can be arbitrarily big}

The purpose of this subsection is to show that our invariant $\lambda$ can grow arbitrarily. More precisely, as claimed in \cref{thm:lambdahigh}, for all $n \in \mathbb{N}$ we will define a knot $K$ such that~$\lambda(K) = n$.

We saw that $\Z[G]$-homology is bigraded. However, our invariant $\lambda$ does not depend on the quantum grading, therefore we will omit quantum shifts in this section.

\begin{dfn} \label{def:staircase}
For every $n \in \Z_{> 0}$ the \emph{staircase of rank}~$2n+1$, denoted by~$S_n$, is defined as the chain complex
\[
\begin{tikzcd}
0 \ar[r]& C_0 \ar["d_{S_n}",r]& C_1 \ar[r]& 0,
\end{tikzcd}
\]
where $C_0=(\Z[G])^{n+1}, C_1=(\Z[G])^{n}$ and
\[
d_{S_n}=
\begin{pNiceMatrix}
     2 & G && \Block{2-2}<\Large>{0} & \\
     &\Ddots&\Ddots&&  \\
     \Block{2-2}<\Large>{0} &&&& \\
     &&& 2 & G \\
\end{pNiceMatrix}.
\]
We can represent a staircase $S_n$ as shown in \cref{fig:staircase}.
\begin{figure}[bht]
\centering
\begin{tikzcd}[baseline=-27mm,row sep = 2.9pt]
	R_1 \arrow{rr}{2} && R_{n+2} \\
	\oplus && \oplus \\
	R_2 \arrow{rr}{2} \arrow{rruu}{G} && R_{n+3} \\
	\oplus && \oplus \\
	\vdots & \vdots & \vdots \\
	\oplus && \oplus \\
	R_n \arrow{rr}{2} && R_{2n+1} \\
	\oplus && \\
	R_{n+1} \arrow{rruu}{G} && \\
\end{tikzcd}\hspace{2cm}%
\parbox[b]{4cm}{\includegraphics[scale=.7]{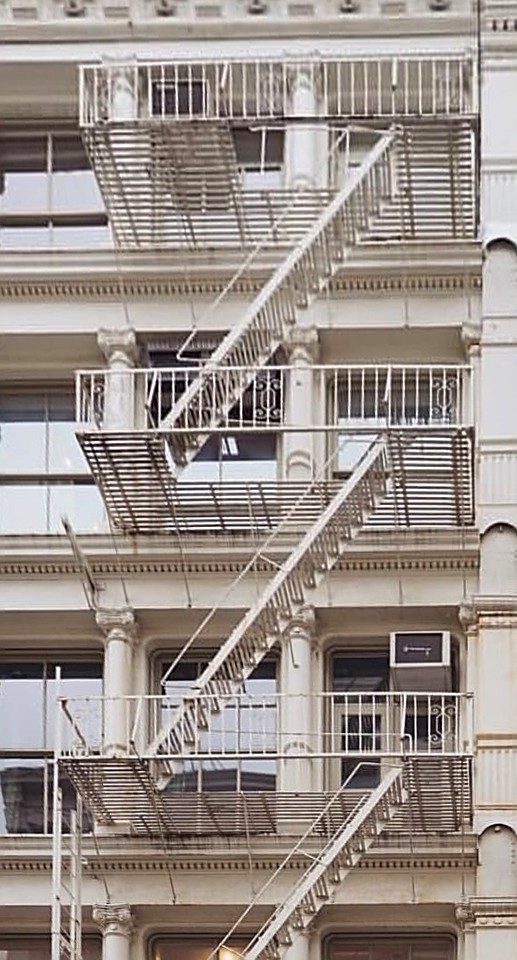}}
\caption{The staircase $S_n$ of rank~$2n + 1$. Here, $R_i=\Z[G]$ for all~$i$.}
    \label{fig:staircase}
\end{figure}
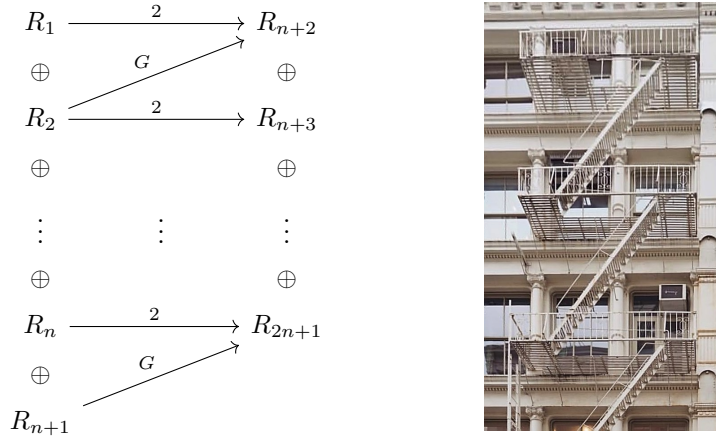
\end{dfn}

Let's prove a few lemmas about staircases.

\begin{lem} \label{lem:lambdaStaircase}
Let $\text{S}_n$ be a staircase of rank~$2n+1$. Then $\lambda(\text{S}_n)=\mathfrak{u}_G(S_n)=n$.
\end{lem}

\begin{proof}
We will first show that $\lambda(S_n) \leq n$ by finding ungraded chain maps $f \colon S_n \to \llbracket U \rrbracket$, $g \colon \llbracket U \rrbracket \to S_n$ such that $f \circ g \simeq G^n \cdot \id_{\llbracket U \rrbracket}$ and $g \circ f \simeq G^n \cdot \id_{S_n}$.

Since $f,g$ need not respect the homological degree we will simply consider $S_n$ as a pair $\left( S_n = R_1 \oplus \cdots \oplus R_{2n+1}, d\colon S_n \to S_n \right)$, with $R_i=\Z[G]$ as in \cref{fig:staircase} and
\[
d=
\NiceMatrixOptions{code-for-last-row = \color{darkred},
                   code-for-first-col = \color{darkred}}
\begin{pNiceArray}{ccc|cc}[last-row,first-col,margin]
\arrayrulecolor{darkred}
    \Block{2-1}{\color{darkred}\scriptstyle{n+1}}& \Block{2-3}<\LARGE>{0} &&&\Block{2-2}<\LARGE>{0}& \\
    \phantom{000}& &&&& \\
      \hline
    \Block{2-1}{\color{darkred}\scriptstyle{n}} & \Block{2-3}{d_{S_{n}}} &&&\Block{2-2}<\LARGE>{0}& \\
    & &&&& \\
    & \Block{1-3}{\color{darkred}\scriptstyle{n+1}}&&&\Block{1-2}{\color{darkred}\scriptstyle{n}}& \\
\end{pNiceArray}.
\]
Define
\begin{equation*}
\begin{split}
   & f = \begin{pmatrix} 1 & 0 & \cdots & 0
        \end{pmatrix} , \\
   & g = \begin{pmatrix} G^n & -2G^{n-1} & 4G^{n-2} & \cdots & (-2)^n & 0 & \cdots & 0
        \end{pmatrix}^{\mathsf{T}}.
\end{split}
\end{equation*}
It's easy to check that $f$ and $g$ are ungraded chain maps (interestingly, they also respect the homological degree, so they are actually graded chain maps) and $f \circ g = G^n \cdot \id_{\llbracket U \rrbracket}$. We need to verify that $g \circ f \simeq G^n \cdot \id_{S_n}$. We have
\[
g \circ f = 
\begin{pNiceArray}{c|ccccc}[first-row, last-col,margin]
\arrayrulecolor{darkred}
     &&& \color{darkred} \scriptstyle 2n && & \\
     G^n & \Block{7-5}<\LARGE>{0} &&&& & \\
     -2G^{n-1} &&&&& & \\
     \Vdots &&&&& & \\
     (-2)^n &&&&& & \color{darkred} \scriptstyle 2n+1 \\
     0 &&&&& & \\
     \Vdots &&&&& & \\
     0 &&&&& & \\ 
\end{pNiceArray}
\]
so let's show that 
\[
G^n \cdot \id_{S_n} - g \circ f = 
\NiceMatrixOptions{code-for-first-row = \color{darkred},
                   code-for-last-col = \color{darkred}}
\begin{pNiceArray}{c|cccccc}[first-row, last-col, margin]
\arrayrulecolor{darkred}
     &\Block{1-6}{\color{darkred} \scriptstyle 2n}&&&&& & \\
     0 & 0 &\Cdots&&&& 0 & \\
     2G^{n-1} & G^n &&&& \Block{2-2}<\LARGE>{0} & & \\
     \Vdots &&\Ddots&&&& & \\
     -(-2)^n &&&&&& & \scriptstyle 2n+1 \\
     0 &&&&&& & \\
     \Vdots & \Block{2-2}<\LARGE>{0} &&&&& & \\
     0 &&&&&& G^n & \\ 
\end{pNiceArray}
\]
is nullhomotopic. We define $h\colon S_n \to S_n$ as
\[
h=
\NiceMatrixOptions{code-for-first-row = \color{darkred},
                   code-for-last-col = \color{darkred}
                   \scriptstyle }
\begin{pNiceArray}{ccc|cccc}[first-row,last-col,margin]
\arrayrulecolor{darkred}
     & \color{darkred} \scriptstyle n+1 && \Block{1-4}{\color{darkred} \scriptstyle n} &&&& \\
     \Block{5-3}<\LARGE>{0} &&& 0 & \Block{3-3}<\LARGE>{0} &&& \\
     &&& G^{n-1} &&&& \\
     &&& -2G^{n-2} &\Ddots&&& n+1 \\
     &&& \Vdots &\Ddots&&& \\
     &&& (-2)^{n-1} & \Cdots & -2G^{n-2} & G^{n-1} & \\
      \hline
     \Block{3-3}<\LARGE>{0} &&& \Block{3-4}<\LARGE>{0} &&&& \\
     &&&&&&& n \\
     &&&&&&& \\
\end{pNiceArray}.
\]
It's easy to see that $d \circ h + h \circ d = G^n \cdot \id_{S_n} - g \circ f$. Therefore $G^n \cdot \id_{S_n} - g \circ f$ is nullhomotopic and $g \circ f \simeq G^n \cdot \id_{S_n}$.

We now show that~$\lambda(S_n) \geq n$. Since $\lambda \geq \mathfrak{u}_G$, it is enough to find an element $x \in H(\llbracket S_n \rrbracket)$ that has $G$-torsion order~$n$. Let $x = \begin{pmatrix}
     1 & 0 & \cdots & 0
\end{pmatrix}^{\mathsf{T}} \in C_1(S_n)=\Z[G]^n$. Since
\[
G^n \cdot x = d_{S_n} \left( \begin{pmatrix}
     0 & G^{n-1} & -2G^{n-2} & \cdots & (-2)^{n-1}
\end{pmatrix}^{\mathsf{T}} \right)
\]
clearly $\text{ord}_G(x) \leq n$. We will now show that the inequality $\text{ord}_G(x) \geq n$ also holds, i.e.\ that $G^k \cdot x \ne 0$ for~$k < n$. Let $k \in \Z_{\geq 0}$ and $a = \begin{pmatrix}
     a_1 & \cdots & a_{n+1}
\end{pmatrix}^{\mathsf{T}} \in C_0(S_n)=\Z[G]^{n+1}$ such that $d_{S_n}(a)=G^{k} \cdot x$. Let's prove that this implies~$k \geq n$. The equation
\[
\begin{pmatrix}
2a_1+Ga_2 \\
2a_2+Ga_3 \\
\vdots \\
2a_n+Ga_{n+1}
\end{pmatrix} 
= d_{S_n}(a) = G^k \cdot x =
\begin{pmatrix}
G^k \\
0 \\
\vdots \\
0
\end{pmatrix} 
\]
yields the following:
\begin{align*}
    &2a_n+Ga_{n+1}=0 &\Rightarrow& &a_{n+1}=-2b_n, \quad & a_{n}=Gb_n &\text{ for } b_n \in \Z[G] \\
    &2a_{n-1}+Ga_n=0 &\Rightarrow& &a_{n}=-2Gb_{n-1}, \quad & a_{n-1}=G^2b_{n-1} &\text{ for } b_{n-1} \in \Z[G] \\
    &\qquad \vdots &&&\vdots&&\vdots \qquad \\
    &2a_{2}+Ga_3=0 &\Rightarrow& &a_{3}=-2G^{n-2}b_{2}, \quad & a_{2}=G^{n-1}b_{2} &\text{ for } b_{2} \in \Z[G] \\
    &2a_{1}+Ga_2=G^k &\Rightarrow& &2a_{1}+G^nb_2=G^k \quad & \Rightarrow k \geq n. & \\
\end{align*}
This proves that~$\text{ord}_G(x)=n$, so $\lambda(S_n) \geq \mathfrak{u}_G(S_n) \geq n$. It follows that $\lambda(S_n)=\mathfrak{u}_G(S_n)=n$.
\end{proof}

\begin{lem} \label{lem:tensorStaircases}
Forgetting about quantum shifts we have
\[
\text{S}_1 \otimes \text{S}_n \cong \text{S}_{n+1} \oplus (\symknight (G) \otimes \symknight (2))^{\oplus n}.
\]
\end{lem}

\begin{proof}
Let $C= S_1 \otimes S_n$. The complex $C$ is isomorphic to 
\begin{equation*}
\centering
\begin{tikzcd}[row sep = 0pt]
	(\Z[G])^{2n+2} \arrow{rr}{A} && (\Z[G])^{3n+1} \arrow{rr}{B} && (\Z[G])^n. \\
	\rotatebox[origin=c]{90}{=} && \rotatebox[origin=c]{90}{=} && \rotatebox[origin=c]{90}{=}  \\
    C_0 && C_1 && C_2 \\
\end{tikzcd}
\end{equation*}
In order to describe the maps $A,B$ we have to choose a basis for $C_0,C_1,C_2$. Let's follow the notation of \cref{fig:staircase} for $S_1$ and~$S_n$, and denote by $a_i$ the generators of the $R_i$ belonging to $S_1$ and by $b_i$ the generators of the $R_i$ in~$S_n$. Then $C_0,C_1,C_2$ are respectively generated by
\[
\begin{pNiceMatrix}[margin]
\arrayrulecolor{darkred}
    a_1 \otimes b_1 \\
    \Vdots \\
    a_1 \otimes b_{n+1} \\
    \hline
    a_2 \otimes b_1 \\
    \Vdots \\
    a_2 \otimes b_{n+1} \\    
\end{pNiceMatrix}
, \qquad
\begin{pNiceMatrix}[margin]
\arrayrulecolor{darkred}
    a_1 \otimes b_{n+2} \\
    \Vdots \\
    a_1 \otimes b_{2n+1} \\
    \hline
    a_2 \otimes b_{n+2} \\
    \Vdots \\
    a_2 \otimes b_{2n+1} \\ 
    \hline
    a_3 \otimes b_{1} \\
    \Vdots \\
    a_3 \otimes b_{n+1} \\   
\end{pNiceMatrix}
, \qquad
\begin{pNiceMatrix}[margin]
\arrayrulecolor{darkred}
    a_3 \otimes b_{n+2} \\
    \Vdots \\
    a_3 \otimes b_{2n+1} \\
\end{pNiceMatrix}.
\]
We can now write out the differentials of~$C$:
\[
A=
\NiceMatrixOptions{code-for-first-row = \color{darkred}
                   \scriptstyle ,
                   code-for-last-col = \color{darkred} \scriptstyle }
\begin{pNiceArray}{c|c}[first-row,last-col,margin]
\arrayrulecolor{darkred}
     n+1 & n+1 & \\
     \Block{2-1}{d_{S_n}} & \Block{2-1}<\LARGE>{0} & n \\
     & & \\
     \hline
     \Block{3-1}<\LARGE>{0} & & \\
     & d_{S_n} & n \\
     & & \\
     \hline
     & & \\
     2 \cdot \mathbbm{1}_{n+1} & G \cdot \mathbbm{1}_{n+1} & n+1 \\
     & & \\
\end{pNiceArray}
\hspace{-0.5cm}, \qquad B=
\NiceMatrixOptions{code-for-first-row = \color{darkred}                             \scriptstyle ,
                   code-for-last-col = \color{darkred} \scriptstyle }
\begin{pNiceArray}{c|c|c}[first-row,last-col,margin]
\arrayrulecolor{darkred}
     n & n & n+1  & \\
     && & \\
     2 \cdot \mathbbm{1}_{n} & G \cdot \mathbbm{1}_{n} & -d_{S_n} & n \\
     && & \\
\end{pNiceArray}
\]
where $d_{S_n}$ is the $n \times (n+1)$-matrix introduced in \cref{def:staircase}.

Let now $C'=\text{S}_{n+1} \oplus (\symknight (G) \otimes \symknight (2))^{\oplus n}$. This complex is given by 
\begin{equation*}
\centering
\begin{tikzcd}[row sep = 0pt]
	(\Z[G])^{n+2} \arrow{rr}{d_{S_{n+1}}} && (\Z[G])^{n+1} && \\
	\oplus && \oplus && \\
	(\Z[G])^{n} \arrow{rr}{A''} && (\Z[G])^{2n} \arrow{rr}{B''} && (\Z[G])^n \\
\color{darkred}	\rotatebox[origin=c]{90}{=} && \color{darkred}	\rotatebox[origin=c]{90}{=} && \color{darkred}	\rotatebox[origin=c]{90}{=}  \\
\color{darkred} C'_0 \arrow[darkred]{rr}{A'} && \color{darkred}	C'_1	\arrow[darkred]{rr}{B'} &&\color{darkred} C'_2 \\
\end{tikzcd}
\end{equation*}
where
\[
A''=
\begin{pNiceMatrix}[name=m,margin]
    G &&\Block{2-2}<\LARGE>{0}& \\
    2 &&& \\
    & G && \\
    & 2 && \\
    &&& \\
    &&& \\
    \Block{2-2}<\LARGE>{0}&&& G \\
    &&& 2
\CodeAfter \tikz[remember picture] {
    \draw[line width=0.3mm, shorten >=2.5mm, shorten <=2.5mm, loosely dotted] (m-3-2)-- (m-7-4);
    \draw[line width=0.3mm, shorten >=2.5mm, shorten <=2.5mm, loosely dotted] (m-4-2)-- (m-8-4);
    }
\end{pNiceMatrix}
, \qquad B''=
\begin{pNiceMatrix}[name=n,margin]
    \Block{2-2}<\LARGE>{0}&&&&&& -2 & G \\
    &&&&&&& \\ 
    && -2 & G &&&\Block{2-2}<\LARGE>{0}& \\
    -2 & G &&&&&&
\CodeAfter \tikz[remember picture] {
    \draw[line width=0.3mm, shorten >=2.5mm, shorten <=2.5mm, loosely dotted] (n-1-7)-- (n-3-3);
    \draw[line width=0.3mm, shorten >=2.5mm, shorten <=2.5mm, loosely dotted] (n-1-8)-- (n-3-4);
    }
\end{pNiceMatrix}
\]
and
\[
A'=
\NiceMatrixOptions{code-for-first-row = \color{darkred},
code-for-first-col = \color{darkred} }
\begin{pNiceArray}{ccc|cc}[first-row,first-col,margin]
\arrayrulecolor{darkred}
    & \Block{1-3}{\color{darkred}\scriptstyle n+2}&&&\Block{1-2}{\color{darkred}\scriptstyle n} \\
   \Block{2-1}{\color{darkred}\scriptstyle n+1} & \Block{2-3}{d_{S_{n+1}}} &&&\Block{2-2}<\LARGE>{0}& \\
    \phantom{000}& &&&& \\
   \hline
   \Block{2-1}{\color{darkred}\scriptstyle 2n} & \Block{2-3}<\LARGE>{0}&&&\Block{2-2}{A''}& \\
    & &&&&\\
\end{pNiceArray}
, \qquad B'=
\NiceMatrixOptions{code-for-first-row = \color{darkred},
                   code-for-last-col = \color{darkred} \scriptstyle }
\begin{pNiceArray}{cc|c}[first-row,last-col,margin]
\arrayrulecolor{darkred}
    \Block{1-2}{\color{darkred}\scriptstyle n+1}&& \scriptstyle 2n & \\
    \Block{3-2}<\LARGE>{0}&& & \\
    && B'' & n \\
    && & \\
\end{pNiceArray}.
\]
Our goal is therefore to find a change of basis to obtain $C'$ from~$C$. We will do this in two steps: we first define a change of basis from $C$ to $S_{n+1} \oplus \widetilde{C}$, for some complex $\widetilde{C}=\left( \widetilde{C}_0 \overset{\widetilde{A}}{\longrightarrow} \widetilde{C}_1 \overset{\widetilde{B}}{\longrightarrow} \widetilde{C}_2 \right)$; then we do a second change of basis that yields $(\symknight (G) \otimes \symknight (2))^{\oplus n}$ from~$\widetilde{C}$.

For the first step we have to find two invertible matrices $M,N$ (over~$\Z[G]$) of dimension $2n+2$ and $3n+1$ respectively, such that
\begin{equation}
NAM=
\begin{pNiceArray}{ccc|cc}[margin]
\arrayrulecolor{darkred}
     \Block{2-3}{d_{S_{n+1}}} &&&\Block{2-2}<\LARGE>{0}& \\
     &&&& \\
     \hline
     \Block{2-3}<\LARGE>{0}&&&\Block{2-2}{\widetilde{A}}& \\
     &&&& \\
\end{pNiceArray}
=
\NiceMatrixOptions{code-for-first-row = \color{darkred},
                   code-for-last-col = \color{darkred} \scriptstyle }
\begin{pNiceArray}{cccc|cccc}[first-row, last-col,margin]
\arrayrulecolor{darkred}
    \Block{1-4}{\color{darkred}\scriptstyle n+2} &&& & \Block{1-4}{\color{darkred}\scriptstyle n} &&& & \\
    2 & G && &\Block{3-4}<\LARGE>{0}&&& & \\
    & \Ddots & \Ddots & &&&& & \color{darkred} n+1 \\
    && 2 & G & &&& & \\
    \hline
    \Block{3-4}<\LARGE>{0}&&& &\Block{3-4}{\widetilde{A}}&&& & \\
    &&& &&&& & \color{darkred} 2n \\
    &&& &&&& & \\
\end{pNiceArray}
\label{eq:changeOfBasisA}
\end{equation}
and
\begin{equation}
BN^{-1}=
\NiceMatrixOptions{code-for-first-row = \color{darkred} \scriptstyle,
code-for-last-col = \color{darkred} \scriptstyle }
\begin{pNiceArray}{ccc|c}[first-row,last-col,margin]
\arrayrulecolor{darkred}
    & n+1 && 2n & \\ 
    \Block{3-3}<\LARGE>{0}&&& & \\
    &&& \widetilde{B} & n \\
    &&& & \\
\end{pNiceArray}.
\label{eq:changeOfBasisB}
\end{equation}
We define $M$ as follows:
\[
M=
\NiceMatrixOptions{code-for-first-row = \color{darkred},
code-for-last-col = \color{darkred} \scriptstyle }
\begin{pNiceArray}{cccc|cccc}[first-row,last-col,name=p,margin]
\arrayrulecolor{darkred}
     \Block{1-4}{\color{darkred}\scriptstyle n+1}&&& &\Block{1-4}{\color{darkred}\scriptstyle n+1}&&& & \\
     \Block{3-4}{\mathbbm{1}_{n+1}} &&&     &\Block{3-4}<\LARGE>{0}&&& & \\
     &&& &&&& & n+1 \\
     &&& &&&& & \\
     \hline
     0 & 1 &&      &&&& 1  & \\
     &&&                    &&&& & n+1 \\
     &&& 1                  &&&& & \\  
     &&& 0     & 1 &&& &
\CodeAfter \tikz[remember picture] {
    \draw[line width=0.3mm, shorten >=2.5mm, shorten <=2.5mm, loosely dotted] (p-4-2)-- (p-6-4);    
    \draw[line width=0.3mm, shorten >=2.5mm, shorten <=2.5mm, loosely dotted] (p-4-8)-- (p-7-5);
    }
\end{pNiceArray}.
\]
We get
\[
AM=
\NiceMatrixOptions{code-for-first-row = \color{darkred},
                   code-for-last-col = \color{darkred} \scriptstyle }
\begin{pNiceArray}{cccccc|cccc}[first-row,last-col,name=q,margin]
\CodeBefore
        \rectanglecolor{myblue}{1-1}{3-6}
        \rectanglecolor{myred}{4-1}{6-6}
        \rectanglecolor{myred}{7-1}{9-6}
        \rectanglecolor{myblue}{10-1}{10-6}
    \Body
\arrayrulecolor{darkred}
     \Block{1-6}{\color{darkred}\scriptstyle n+2}&&&&&&\Block{1-4}{\color{darkred}\scriptstyle n}&&& & \\
     2 & G &&&& 0 &\Block{3-4}<\LARGE>{0}&&&  & \\
     &&&& & \Vdots &&&& & n \\
     &&& 2 & G & 0  & &&& & \\
     \hline
     0 & 2 & G &&  & &&& G & 2 & \\
     \Vdots &&&& & & G &&& & n  \\
     0 &&&& 2 & G  &  2 &&& & \\
     \hline
     2 & G &&&& 0  & &&& G & \\
     &&&& & \Vdots  & &&& & n  \\
     &&& 2 & G & 0  &  G &&& & \\
     \hline
     0 & \Cdots && 0 & 2 & G  & 0 & \Cdots && 0 &
\CodeAfter \tikz 
   {\draw[line width=0.3mm, shorten >=2.5mm, shorten <=2.5mm, loosely dotted] (q-1-1)-- (q-3-4);
    \draw[line width=0.3mm, shorten >=2.5mm, shorten <=2.5mm, loosely dotted] (q-1-2)-- (q-3-5);
    \draw[line width=0.3mm, shorten >=2.5mm, shorten <=2.5mm, loosely dotted] (q-4-2)-- (q-6-5);
    \draw[line width=0.3mm, shorten >=2.5mm, shorten <=2.5mm, loosely dotted] (q-4-3)-- (q-6-6);
    \draw[line width=0.3mm, shorten >=2.5mm, shorten <=2.5mm, loosely dotted] (q-7-1)-- (q-9-4);
    \draw[line width=0.3mm, shorten >=2.5mm, shorten <=2.5mm, loosely dotted] (q-7-2)-- (q-9-5);
    \draw[line width=0.3mm, shorten >=2.5mm, shorten <=2.5mm, loosely dotted] (q-5-7)-- (q-4-9);
    \draw[line width=0.3mm, shorten >=2.5mm, shorten <=2.5mm, loosely dotted] (q-6-7)-- (q-4-10);
    \draw[line width=0.3mm, shorten >=2.5mm, shorten <=2.5mm, loosely dotted] (q-9-7)-- (q-7-10);
}
\end{pNiceArray}.
\]
In order to obtain the right-hand side of equation \eqref{eq:changeOfBasisA} from $AM$ we have to cancel all the pairs \quad $2 \quad G$ \quad in the blocks highlighted in light red (this can be easily done by subtracting rows of the dark blue blocks), move $r_{3n+1}$ to row ${n+1}$ (where $r_i$ indicates the $i$-th row) and slide down $r_{n+1}, \ldots, r_{3n}$ consequently. Let us call $N$ the matrix expressing these operations, i.e.
\[
N = Q \cdot P^{3n}_{n} \cdots P^{2n+1}_{1} \cdot P^{2n}_{3n+1} \cdot P^{2n-1}_{n} \cdots P^{n+1}_{2}
\]
where $P^a_b$ is obtained from $\mathbbm{1}_{3n+1}$ by replacing $r_a$ with $r_a-r_b$, and $Q$ is the matrix expressing the appropriate reordering of rows $r_{n+1},\ldots,r_{3n+1}$. It is thus clear that $N$ satisfies equation \eqref{eq:changeOfBasisA}.
It is straightforward to check that $N^{-1}$ also satisfies \eqref{eq:changeOfBasisB}. Thus we showed that $S_1 \otimes S_n \cong S_{n+1} \oplus \widetilde{C}$.

The next step is to define a change of basis from $\widetilde{C}$ to $(\symknight (G) \otimes \symknight (2))^{\oplus n}$. For that purpose we have to explicitly describe $\widetilde{A}$ and~$\widetilde{B}$:
\[
\widetilde{A}=
\NiceMatrixOptions{code-for-last-col = \color{darkred} \scriptstyle }
\begin{pNiceArray}{cccc}[last-col,name=r,margin]
\CodeBefore 
        \rectanglecolor{myred}{1-1}{4-4}
        \rectanglecolor{myblue}{5-1}{8-4}
    \Body
\arrayrulecolor{darkred}
     \phantom{0}&& G & 2 & \\
     &&& & \\
     G &&& & n \\
     2 &&& \phantom{0} & \\
     \hline
     \phantom{0} &&& G & \\
     &&& & n \\
     &&& & \\
     G &&& \phantom{0} & 
\CodeAfter \tikz{
    \draw[line width=0.3mm, shorten >=2.5mm, shorten <=2.5mm, loosely dotted] (r-1-3)-- (r-3-1);
    \draw[line width=0.3mm, shorten >=2.5mm, shorten <=2.5mm, loosely dotted] (r-1-4)-- (r-4-1);
    \draw[line width=0.3mm, shorten >=2.5mm, shorten <=2.5mm, loosely dotted] (r-5-4)-- (r-8-1);
}
\end{pNiceArray}
, \qquad
\widetilde{B}=
\NiceMatrixOptions{code-for-first-row = \color{darkred} }
\begin{pNiceArray}{cccc|cccc}[first-row,name=t,margin]
\arrayrulecolor{darkred}
     \Block{1-4}{\color{darkred}\scriptstyle n}&&& & \Block{1-4}{\color{darkred}\scriptstyle n} &&& \\
     G &&& &  -2 & -G && \\
     &\Ddots&& & &&& \\
     &&& & &&& -G \\
     &&& G  & &&& -2
\CodeAfter \tikz{
    \draw[line width=0.3mm, shorten >=2.5mm, shorten <=2.5mm, loosely dotted] (t-1-5)-- (t-4-8);
    \draw[line width=0.3mm, shorten >=2.5mm, shorten <=2.5mm, loosely dotted] (t-1-6)-- (t-3-8);
}
\end{pNiceArray}.
\]
We can easily obtain $A''$ from $\widetilde{A}$ by removing the $G$ entries in the light red block (by subtracting rows of the dark blue block) and then reordering the rows appropriately. These operations are expressed by the matrix 
\[
L=\widetilde{Q} \cdot P^{n-1}_{2n} \cdots P^{1}_{n+2}
\]
where $P^a_b$ is obtained from $\mathbbm{1}_{2n}$ by replacing $r_a$ with $r_a-r_b$, and $Q$ is the matrix expressing the appropriate reordering of the rows.
One can check that $L\widetilde{A}=A''$ and $\widetilde{B}L^{-1}=B''$.
This concludes the proof of the lemma.
\end{proof}

We are now ready to prove \cref{thm:lambdahigh}.
\begin{proof}[Proof of \cref{thm:lambdahigh}]
\cref{lem:lambdaStaircase} and \cref{lem:tensorStaircases}, together with the fact that  $\llbracket K_1 \# K_2 \rrbracket \cong \llbracket K_1 \rrbracket \otimes \llbracket K_2 \rrbracket$ for any two knots $K_1,K_2$ (see \cref{eq:connectedsum}), are enough to construct knots with arbitrarily big~$\lambda$. Indeed, let us consider the knot $K=14n19265$. This knot was used by Seed to show that $s(K) \ne s_{\mathbb{F}_2}(K)$ \cite{knotkit,zbMATH06296598}, where $s$ is the classical Rasmussen invariant over $\Q$ and $s_{\mathbb{F}_2}$ is the invariant computed over~$\mathbb{F}_2$. We observe using \verb+khoca+ and \verb+homca+ that the $\Z[G]$-complex $\llbracket K \rrbracket$ decomposes as a sum of a staircase $\text{S}_1$ and finitely many $\symknight(G)$ and $\symknight(G) \otimes \symknight(2)$. Therefore, by \cref{lem:lambdadirectsum}:
\[
\lambda(K) \leq \max (\lambda(S_1), \hspace{0.2cm} \lambda(\symknight(G),0), \hspace{0.2cm} \lambda(\symknight(G) \otimes \symknight(2),0)) = 1.
\]
Since $K \ne U$ it follows that~$\lambda(K)=1$. By \cref{prop:lambdabasicproperties}, given $n \in \Z_{> 0}$, we then have:
\[
\lambda ( K^{\#n} ) \leq n \cdot\lambda ( K ) = n.
\]
On the other hand, it follows from \cref{lem:tensorStaircases} that $\llbracket K^{\#n} \rrbracket \cong S_{n} \oplus C$ for some chain complex~$C$. We know that $\mathfrak{u}_G ( S_{n} \oplus C) = \text{max} (\mathfrak{u}_G ( S_{n}), \mathfrak{u}_G (C))$ (cf.~\cref{lem:propertiesuG}), so
\[
\lambda ( K^{\#n} ) = \lambda ( S_{n} \oplus C ) \geq  \mathfrak{u}_G ( S_{n} \oplus C) \geq \mathfrak{u}_G ( S_{n}) = n.
\]
This proves that $\lambda(K^{\#n})=n$ for all~$n \geq 0$.
\end{proof}
The fact that $\lambda(K^{\# n})=n$ will also follow from \cref{prop:lambdaKand-K}.

\subsection{Further calculations}
\begin{prop} \label{prop:lambdaKand-K}
Let $K_1,\ldots,K_n$ and $J_1, \ldots , J_m$ be knots such that:
\begin{enumerate}
    \item for all $i=1,\ldots ,n$ the complex $\llbracket K_i \rrbracket$ splits as a sum of one staircase $S_1$ and finitely many $\symknight(G)$ and $\symknight(G) \otimes \symknight(2)$ pieces,
    \item for all $j=1,\ldots ,m$ the complex $\llbracket J_j \rrbracket$ decomposes as a sum of one dual staircase $\overline{S_1}$ and finitely many $\symknight(G)$ and $\symknight(G) \otimes \symknight(2)$.
\end{enumerate}
Let the empty $\scalebox{1.3}{$\#$}$ be equal to the unknot. Then
\[
\lambda(\underset{i \leq n}{\scalebox{1.3}{$\#$}} K_i \hspace{0.1cm} \# \hspace{0.1cm} \underset{j \leq m}{\scalebox{1.3}{$\#$}} J_j) = \left\{ \begin{array}{ll}
|n - m| & \text{if $n \neq m$ and $m,n \geq 0$,} \\
1 & \text{if $n = m \neq 0$,} \\
0 & \text{if $n = m = 0$.} \end{array}\right.
\]
\end{prop}

\begin{remark} \label{rem:examplesKand-K}
Using \verb+khoca+ and \verb+homca+ one finds that there are many knots satisfying requirements (1) or (2) of \cref{prop:lambdaKand-K}. For instance, one can take any knot with up to 15 crossings such that $s_{\mathbb{F}_2} \ne s_{\mathbb{F}_3}$. One of those is the above-mentioned knot $14n19265$, and a complete list is given in~\cite{zbMATH07333633,zbMATH07506057}.

We also note that if a knot $K$ satisfies condition (1) of \cref{prop:lambdaKand-K} then its mirror image $-K$ will satisfy condition (2), and vice-versa.
\end{remark}

For the proof of \cref{prop:lambdaKand-K} we will need the following lemmas.

\begin{lem} \label{lem:knighttensorS_n}
Ignoring quantum shifts we have 
\[
\symknight(G) \otimes S_n \cong \symknight(G) \otimes \overline{S_n} \cong ( \symknight(G) \otimes \symknight(2) )^{\oplus n} \oplus \symknight(G)
\]
and
\[
\symknight(2) \otimes S_n \cong \symknight(2) \otimes \overline{S_n} \cong ( \symknight(G) \otimes \symknight(2) )^{\oplus n} \oplus \symknight(2).
\]
\end{lem}

\begin{proof}
We proceed very similarly to the proof of \cref{lem:tensorStaircases}. 
We will only prove that $\symknight(G) \otimes S_n \cong ( \symknight(G) \otimes \symknight(2))^{\oplus n} \oplus \symknight(G)$, as the proofs of the remaining statements are very similar.

The complex $\symknight(G) \otimes S_n$ is isomorphic to 
\begin{equation*}
\centering
\begin{tikzcd}[row sep = 0pt]
	(\Z[G])^{n+1} \arrow{rr}{A} && (\Z[G])^{2n+1} \arrow{rr}{B} && (\Z[G])^n
\end{tikzcd}
\end{equation*}
with
\[
A=
\NiceMatrixOptions{code-for-last-col = \color{darkred} \scriptstyle }
\begin{pNiceArray}{ccccc}[last-col,name=u,margin]
\CodeBefore 
        \rectanglecolor{myred}{1-1}{4-5}
        \rectanglecolor{myblue}{5-1}{9-5}
    \Body
\arrayrulecolor{darkred}
     2 & G &&& & \\
     &\Ddots&\Ddots&& & n \\
     &&&& & \\
     &&& 2 & G  &\\
     \hline
     G &&&& & \\
     &\Ddots&&& & \\
     &&&& & n+1 \\
     &&&& & \\
     &&&& G  & 
\end{pNiceArray}
\hspace{-5mm},\hspace{5mm}
B=
\NiceMatrixOptions{code-for-first-row = \color{darkred} }
\begin{pNiceArray}{cccc|ccccc}[first-row,name=v,margin]
\arrayrulecolor{darkred}
     \Block{1-4}{\color{darkred}\scriptstyle n}&&& & \Block{1-5}{\color{darkred}\scriptstyle n+1} &&&& \\
     G &&& &  -2 & -G &&& \\
     &\Ddots&& & &&&& \\
     &&& & &&&& \\
     &&& G  & &&& -2 & -G
\CodeAfter \tikz{
    \draw[line width=0.3mm, shorten >=2.5mm, shorten <=2.5mm, loosely dotted] (v-1-5)-- (v-4-8);
    \draw[line width=0.3mm, shorten >=2.5mm, shorten <=2.5mm, loosely dotted] (v-1-6)-- (v-4-9);
}
\end{pNiceArray}
\]
and basis given by
\[
\begin{pNiceMatrix}[margin]
\arrayrulecolor{darkred}
    a_1 \otimes b_1 \\
    \Vdots \\
    a_1 \otimes b_{n+1} \\
\end{pNiceMatrix}
, \qquad
\begin{pNiceMatrix}[margin]
\arrayrulecolor{darkred}
    a_1 \otimes b_{n+2} \\
    \Vdots \\
    a_1 \otimes b_{2n+1} \\
    \hline
    a_2 \otimes b_{1} \\
    \Vdots \\
    a_2 \otimes b_{n+1} \\   
\end{pNiceMatrix}
, \qquad
\begin{pNiceMatrix}[margin]
\arrayrulecolor{darkred}
    a_2 \otimes b_{n+2} \\
    \Vdots \\
    a_2 \otimes b_{2n+1} \\
\end{pNiceMatrix}
\]
(here $b_i$ denotes the generator of $R_i$ in $S_n$ and $a_1,a_2$ are respectively the generators of $\Z[G]$ in homological degrees $0$ and $1$ of~$\symknight(G)$).

The statement of the lemma holds if we can find a change of basis $N$ such that 
\begin{equation}
\label{eq:reChangeOfBasisA}
NA=
\begin{pNiceArray}{cccc|c}[name=x,margin]
\arrayrulecolor{darkred}
    G &&& & 0 \\
    2 &&& & \\
    & G && & \\
    & 2 && & \\
    &&& & \Vdots  \\
    &&& & \\
    &&& G  & \\
    &&& 2 & 0 \\
    \hline
    0 &&\Cdots& 0  & G
\CodeAfter \tikz[remember picture] {
    \draw[line width=0.3mm, shorten >=2.5mm, shorten <=2.5mm, loosely dotted] (x-3-2)-- (x-7-4);
    \draw[line width=0.3mm, shorten >=2.5mm, shorten <=2.5mm, loosely dotted] (x-4-2)-- (x-8-4);
    }
\end{pNiceArray}
\end{equation}
and
\begin{equation}
\label{eq:reChangeOfBasisB}
BN^{-1}=
\begin{pNiceArray}{cccccccc|c}[name=y,margin]
\arrayrulecolor{darkred}
    -2 & G &&&&&& & 0 \\
    && -2 & G &&&& & \\ 
    &&&&&&& & \Vdots \\
    &&&&&& -2 & G & 0
\CodeAfter \tikz[remember picture] {
    \draw[line width=0.3mm, shorten >=2.5mm, shorten <=2.5mm, loosely dotted] (y-2-3)-- (y-4-7);
    \draw[line width=0.3mm, shorten >=2.5mm, shorten <=2.5mm, loosely dotted] (y-2-4)-- (y-4-8);
    }
\end{pNiceArray}.
\end{equation}
In order to obtain the right-hand side of \eqref{eq:reChangeOfBasisA} from $A$ we need to perform the following operations: get rid of the $G$ entries in the light red block of $A$ (by subtracting rows of the dark blue block) and then reorder the rows appropriately (let's call $Q$ the matrix expressing this second step). Then the change of basis
\[
N= Q \cdot P^{n}_{2n+1} \cdots P^{1}_{n+2}
\]
clearly satisfies equation \eqref{eq:reChangeOfBasisA}. Some simple calculations show that $N$ also satisfies~\eqref{eq:reChangeOfBasisB}.
\end{proof}

\begin{lem} \label{lem:knighttensorknight}
Let $z \in \Z[G]$ and $a,b \in \Z_{\geq 0}$ with~$a \leq b$. Then
\[
\symknight(z^a) \otimes \symknight(z^b) \cong \symknight(z^a) \oplus \symknight(z^a)
\]
(quantum shifts are omitted).
\end{lem}

\begin{proof}
The complex $\symknight(z^a) \otimes \symknight(z^b)$ is isomorphic to
\begin{equation*}
\centering
\begin{tikzcd}[row sep = 0pt]
	\Z[G] \arrow{rr}{\left( z^b \hspace{0.2cm} z^a \right)^{\mathsf{T}}} && \Z[G] \oplus \Z[G] \arrow{rr}{ \left( z^a \hspace{0.2cm} -z^b \right)} && \Z[G].  \\
\end{tikzcd}
\end{equation*}
Consider the matrix $N= \begin{pmatrix}
1 & -z^{b-a} \\
0 & 1 
\end{pmatrix}$. We have 
\[
N \cdot 
\begin{pNiceMatrix}
	z^b \\
	z^a
 \end{pNiceMatrix}
= \begin{pNiceMatrix}
	0 \\
	z^a
	\end{pNiceMatrix}
, \qquad \begin{pmatrix}
z^a & -z^b
   \end{pmatrix}
\cdot N^{-1} = 
\begin{pmatrix}
z^a & 0
  \end{pmatrix}.
\]
The matrix $N$ is therefore the desired change of basis from $\symknight(z^a) \otimes \symknight(z^b)$ to $\symknight(z^a) \oplus \symknight(z^a)$.
\end{proof}

\begin{lem}
The following are isomorphisms of (ungraded) chain complexes.
\begin{equation} \label{eq:S_1tensordual}
    S_1 \otimes \overline{S_1} \cong (\symknight(G) \otimes \symknight(2))^{\oplus 2} \oplus \sympawn,
\end{equation}
\begin{equation} \label{eq:S_1tensorrank4}
    (\symknight(G) \otimes \symknight(2)) \otimes S_1 \cong (\symknight(G) \otimes \symknight(2)) \otimes \overline{S_1} \cong (\symknight(G) \otimes \symknight(2))^{\oplus 3},
\end{equation}
\begin{equation} \label{eq:rank4tensorrank4}
    (\symknight(G) \otimes \symknight(2)) \otimes (\symknight(G) \otimes \symknight(2)) \cong (\symknight(G) \otimes \symknight(2))^{\oplus 4},
\end{equation}
\begin{equation} \label{eq:Gknighttensorrank4}
    \symknight(G) \otimes (\symknight(G) \otimes \symknight(2)) \cong (\symknight(G) \otimes \symknight(2))^{\oplus 2}.
\end{equation}
\end{lem}

\begin{proof}
The chain complex $S_1 \otimes \overline{S_1}$ can be written as
\begin{equation*}
\centering
\begin{tikzcd}[row sep = 0pt]
	(\Z[G])^2 \arrow{rr}{A} && (\Z[G])^5 \arrow{rr}{B} && (\Z[G])^2,
\end{tikzcd}
\end{equation*}
where 
\[
A= \begin{pmatrix}
G & 0 \\
2 & 0 \\
0 & G \\
0 & 2 \\
2 & G 
\end{pmatrix}
, \qquad 
B= \begin{pmatrix}
2 & 0 & G & 0 & -G \\
0 & 2 & 0 & G & -2
\end{pmatrix}.
\]
Consider the change of basis matrices
\[
L= \begin{pmatrix}
1 & 0 \\
0 & -1
\end{pmatrix}
, \qquad 
M= \begin{pmatrix}
1 & 0 & 0 & 0 & 0 \\
0 & 1 & 0 & 0 & -1 \\
0 & 0 & -1 & 0 & -1 \\
0 & 0 & 0 & -1 & 0 \\
0 & 1 & -1 & 0 & -1 
\end{pmatrix}
\]
and let $A'=M^{-1}AL$ and $B'=BM$. One can verify that $A'$ and $B'$ are the differentials of the chain complex $(\symknight(2) \otimes \symknight(G))^{\oplus 2} \oplus \sympawn \cong (\symknight(G) \otimes \symknight(2))^{\oplus 2} \oplus \sympawn$, which proves equation \eqref{eq:S_1tensordual}.

The first isomorphism of equation \eqref{eq:S_1tensorrank4} is given by \cref{lem:knighttensorS_n}. The following shows that $(\symknight(G) \otimes \symknight(2)) \otimes S_1 \cong (\symknight(G) \otimes \symknight(2))^{\oplus 3}$:
\[
\begin{split}
\symknight(G) \otimes \symknight(2) \otimes S_1 & \cong \symknight(G) \otimes ((\symknight(G) \otimes \symknight(2)) \oplus \symknight(2)) \\
& \cong (\symknight(G) \otimes \symknight(G) \otimes \symknight(2)) \oplus (\symknight(G) \otimes \symknight(2)) \\
& \cong ((\symknight(G) \oplus \symknight(G)) \otimes \symknight(2)) \oplus (\symknight(G) \otimes \symknight(2)) \\
& \cong (\symknight(G) \otimes \symknight(2)) \oplus (\symknight(G) \otimes \symknight(2)) \oplus (\symknight(G) \otimes \symknight(2)) \\
& \cong (\symknight(G) \otimes \symknight(2))^{\oplus 3}
\end{split}
\]
where the first isomorphism is given by \cref{lem:knighttensorS_n} and the third by \cref{lem:knighttensorknight}.
Lastly, equations \eqref{eq:rank4tensorrank4} and \eqref{eq:Gknighttensorrank4} follow easily from \cref{lem:knighttensorknight}.
\end{proof}

We can now turn to the proof of \cref{prop:lambdaKand-K}.
\begin{proof}[Proof of \cref{prop:lambdaKand-K}]
We remind the reader that for two knots $K_1,K_2$ we have $\llbracket K_1 \# K_2 \rrbracket \cong \llbracket K_1 \rrbracket \otimes \llbracket K_2 \rrbracket$ (see \cref{eq:connectedsum}).
Let $L=\underset{i \leq n}{\scalebox{1.3}{\#}} K_i \hspace{0.1cm} \# \hspace{0.1cm} \underset{j \leq m}{\scalebox{1.3}{\#}} J_j$. If $n = m=0$ then $\lambda(L)=~\lambda(U)=0$.

We now consider~$\{n,m\} \ne \{0\}$. Using equations \eqref{eq:S_1tensordual} to \eqref{eq:Gknighttensorrank4}, \cref{lem:knighttensorS_n} and \cref{lem:knighttensorknight} we find that for all $i,j \geq 1$ the complex $\llbracket K_i \# J_j \rrbracket$ splits as a sum of the following pieces:
\[
    \sympawn, \qquad \symknight(G), \qquad \symknight(G)\otimes\symknight(2).
\]
The same pieces also give a decomposition of $\llbracket \underset{i,j \geq 1}{\scalebox{1.3}{\#}} (K_i \# J_j) \rrbracket$.

If $n=m \ne 0$ then $L=\underset{1 \leq i \leq m}{\scalebox{1.3}{\#}} (K_i \# J_i)$. Using \cref{lem:lambdadirectsum} and the fact that $\lambda(\symknight(G),0)=\lambda(\symknight(G)\otimes\symknight(2),0)=1$, one obtains:
\[
\lambda(L)=\lambda(\underset{1 \leq i \leq m}{\scalebox{1.3}{\#}} (K_i \# J_i)) \leq \max (\lambda(\sympawn), \hspace{0.2cm} \lambda(\symknight(G),0), \hspace{0.2cm} \lambda(\symknight(G)\otimes\symknight(2),0)) = 1.
\]
We also have $\lambda(L) \geq 1$ by \cref{prop:lambdadetectsunknot}, since~$L \ne U$. This shows that~$\lambda(L)=1$.

Let now~$n > m \geq 0$. We have
\[ L=\underset{j \leq m}{\scalebox{1.3}{\#}} (K_j \# J_j) \hspace{0.2cm} \# \hspace{0.1cm} \underset{m+1 \leq i \leq n}{\scalebox{1.3}{\#}} K_i.
\]
It is easy to see, using equations \eqref{eq:S_1tensorrank4} to \eqref{eq:Gknighttensorrank4}, \cref{lem:tensorStaircases}, \cref{lem:knighttensorS_n} and \cref{lem:knighttensorknight} that $\llbracket \underset{m+1 \leq i \leq n}{\scalebox{1.3}{\#}} K_i \rrbracket$ splits as a sum of
\[
\symknight(G), \qquad \symknight(G)\otimes\symknight(2), \qquad S_{n-m}.
\]
Now $\llbracket L \rrbracket \cong \llbracket \underset{j \leq m}{\scalebox{1.3}{\#}} (K_j \# J_j) \rrbracket \otimes \llbracket \underset{m+1 \leq i \leq n}{\scalebox{1.3}{\#}} K_i \rrbracket$, therefore equations \eqref{eq:S_1tensorrank4} to \eqref{eq:Gknighttensorrank4}, along with \cref{lem:knighttensorS_n} and \cref{lem:knighttensorknight}, show that the same pieces also give a decomposition of~$\llbracket L \rrbracket$. Thus, in order to prove that $\lambda(L) \leq n-m$ all we have to do is apply \cref{lem:lambdadirectsum}, which yields:
\[
\lambda(L) \leq \max (\lambda(S_{n-m}), \lambda(\symknight(G),0), \lambda(\symknight(G)\otimes\symknight(2),0)) = n-m.   
\]
The inequality $\lambda(L) \geq n-m$ also holds: the complex $\llbracket \underset{j \leq m}{\scalebox{1.3}{\#}} (K_j \# J_j) \rrbracket$ has a $\sympawn$ piece, and $\llbracket \underset{m+1 \leq i \leq n}{\scalebox{1.3}{\#}} K_i \rrbracket$ has a $S_{n-m}$ piece, so there is a piece $S_{n-m} \cong \sympawn \otimes S_{n-m}$ in~$\llbracket L \rrbracket$. Using that $\mathfrak{u}_G(S_{n-m})=n-m$ (cf.~\cref{lem:lambdaStaircase}) and \cref{lem:propertiesuG} one finds $\lambda(L) \geq \mathfrak{u}_G(L) = n-m$. It follows that~$\lambda(L)= n-m$.

Lastly, let~$m > n \geq 0$. Then 
\[ L=\underset{i \leq n}{\scalebox{1.3}{\#}} (K_i \# J_i) \hspace{0.2cm} \# \hspace{0.1cm} \underset{n+1 \leq j \leq m}{\scalebox{1.3}{\#}} J_j
\]
and the only pieces appearing in $\llbracket L \rrbracket$ are 
\[
\symknight(G), \qquad \symknight(G)\otimes\symknight(2), \qquad \overline{S_{m-n}}.
\]
It follows that the pieces appearing in $\llbracket -L \rrbracket = \overline{\llbracket L \rrbracket}$ are 
\[
\symknight(G), \qquad \symknight(G)\otimes\symknight(2), \qquad S_{m-n}.
\]
Hence, by \cref{prop:lambdabasicproperties} and looking at the proof of the case $n > m$ just above, one finds $\lambda(L) = \lambda(-L) = m-n$.
\end{proof}

\begin{remark} \label{rem:u_GofKand-K}
It's easy to see from the above proof that a similar result as \cref{prop:lambdaKand-K} also holds for~$\mathfrak{u}_G$. Namely, if $K_1, \ldots , K_n$, $J_1, \ldots , J_m$ satisfy conditions (1) and (2) of \cref{prop:lambdaKand-K}, we have:
\[
\mathfrak{u}_G(\underset{i \leq n}{\scalebox{1.3}{$\#$}} K_i \hspace{0.1cm} \# \hspace{0.1cm} \underset{j \leq m}{\scalebox{1.3}{$\#$}} J_j) = \left\{ \begin{array}{ll}
n - m & \text{if $n > m \geq 0$} \\
1 & \text{if $n = m \neq 0$ or $m > n \geq 0$} \\
0 & \text{if $n = m = 0$} \end{array}\right.
\]
The partial difference is due to the fact that $\mathfrak{u}_G(\overline{S_{k}})=0$, while $\lambda(\overline{S_{k}})=\lambda(S_{k})=k$ for all~$k \geq 1$.
\end{remark}

	\begin{table}[b]
		\centering
		\begin{tabular}{|c|c|c|}
		\hline
		$K$ & $u(K)$ & $g(K)$ \\
		\hline
		$9_{46}$ & 2 & 1 \\
		$11n_{139}$ & 2 & 1 \\
		$12n_{203}$ & 3 or 4 & 3 \\
		$12n_{260}$ & 2 or 3 & 2 \\
		$12n_{404}$ & 2 or 3 & 2 \\
		$12n_{432}$ & 2 or 3 & 2 \\
		$12n_{554}$ & 3 & 2 \\
		$12n_{642}$ & 3 or 4 & 2 \\
		$12n_{764}$ & 3 or 4 & 3 \\
		$12n_{809}$ & 1, 2 or 3 & 2 \\
		$12n_{851}$ & 3 or 4 & 3 \\
		\hline
		\end{tabular}%
		\caption{Non-quasi-alternating prime knots with up to $12$ crossings for which (possibly) $g < u$ holds.}
		\label{tab:knotslambdagenus}
	\end{table}

\subsection{$\lambda$ of small knots}\label{subsec:lambdasmall}
We start this subsection by computing $\lambda$ for all knots with up to $10$ crossings.

\lambdasmallknots*

\begin{proof}
	By \cref{prop:lambdaofthin}, if a knot is thin then~$\lambda = 1$, so it suffices to look at knots which are not thin. Among the knots with up to $10$ crossings, there are twelve knots that are thick:
	\begin{equation*}
		8_{19},\ 9_{42},\ 10_{124},\ 10_{128},\ 10_{132},\ 10_{136},\ 10_{139},\ 10_{145},\ 10_{152},\ 10_{153},\ 10_{154},\ 10_{161}.
	\end{equation*}
	Using \texttt{khoca} and \texttt{homca}, one can compute that the $\mathbb{Z}[G]$-complex of the knots $9_{42}$, $10_{132}$, $10_{136}$, $10_{145}$, $10_{153}$ decomposes into a sum of a $\sympawn$ and several $\symknight(G)$ pieces, hence $\lambda = 1$ by \cref{lem:lambdaofknights}. The $\mathbb{Z}[G]$-complex of the remaining knots $8_{19}$, $10_{124}$, $10_{128}$, $10_{139}$, $10_{152}$, $10_{154}$, $10_{161}$ decomposes into a sum of a~$\sympawn$, several $\symknight(G)$ pieces and a single $\symknight(G^2)$  piece. Using \cref{lem:lambdadirectsum} and \cref{lem:propertiesuG}, one obtains $\lambda = 2$ for these knots.
\end{proof}

A natural question to ask when introducing a new invariant is how it compares to other already existing invariants. For example, how does $\lambda$ compare to the classical $3$-genus $g$ of a knot~$K$? We know that $\lambda$ is a lower bound for the unknotting number~$u$, while $g$ can be a lower or upper bound for $u$ depending on the knot. For instance, Lee--Lee \cite{zbMATH06212951} showed that for all knots with braid index~$\leq 3$, the inequality $u(K) \leq g(K)$ holds. However, this is no longer true for knots with braid index~$\geq 4$: as pointed out in their work, there are six knots with braid-index $4$ and at most $9$ crossings for which $u > g$ holds. How does $\lambda$ fit into this scheme? For knots up to $12$ crossings, we can provide the following answer.

\begin{prop}\label{prop:lambdagenus}
	For all knots up to $12$ crossings, the $3$-genus $g$ is an upper bound for~$\lambda$.
\end{prop}

\begin{proof}
	Since $\lambda$ is a lower bound for the unknotting number~$u$, it is sufficient to consider knots with up to $12$ crossings for which (possibly) $g < u$ holds. Using that quasi-alternating knots are thin and that for thin knots $\lambda = 1$ (cf.~\cref{prop:lambdaofthin}) there are $11$ non-quasi-alternating knots with at most $12$ crossings with (possibly)~${g < u}$. They were found using Livingston's wonderful KnotInfo \cite{knotinfo} and Jablan's table of quasi-alternating knots for up to $12$ crossings \cite{jablan}. The knots are listed in Table \ref{tab:knotslambdagenus}.

	A computation using \verb+khoca+ and its extension \verb+homca+ showed that the $\mathbb{Z}[G]$-complex of all knots in Table \ref{tab:knotslambdagenus} decomposes into $\sympawn$ and $\symknight(G)$ summands. By \cref{lem:lambdaofknights}, this implies that $\lambda = 1$ for all knots in Table \ref{tab:knotslambdagenus}.
\end{proof}

\cref{prop:lambdagenus} raises the following question.

\begin{question}
	Does $\lambda(K) \leq g(K)$ hold for all knots~$K$?
\end{question}

\section{Rational tangles and the $\lambda$-invariant}\label{sec:rational}
\subsection{The $\Z[G]$-homology of rational tangles}
\begin{dfn}
A 4-ended (oriented or unoriented) tangle $T$ is called \emph{rational} if 
the pair $(B,T)$ is homeomorphic to $(D^2\times [0,1], \{(-\tfrac12,0),(\tfrac12,0)\}\times [0,1])$,
drawn in \cref{fig:trivrational}.
\begin{figure}[ht]
    \centering
    \includegraphics[width=0.15\textwidth]{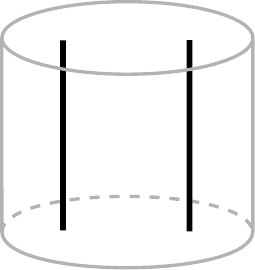}
    \caption{The rational tangle $(D^2\times [0,1], \{(-\tfrac12,0),(\tfrac12,0)\}\times [0,1])$.}
        \label{fig:trivrational}
\end{figure}
\end{dfn}
Let us briefly summarize Conway's famous one-to-one correspondence
\[
R\colon \Q \cup \{\infty\} \to \{\text{unoriented rational tangles}\}/\text{equivalence}.
\]
See e.g.\ \cite{zbMATH02120368} for an introduction to this topic.
Let us work with unoriented tangles in the unit ball $B_0 \subset \mathbb{R}^3\subset S^3$
with the four end points $(\pm 1/\sqrt{2}, \pm 1/{\sqrt{2}}, 0)$,
and base point $(-1/\sqrt{2}, -1/{\sqrt{2}}, 0)$.
Generically, the projection to $D^2 \times \{0\}$
yields tangle diagrams; these are the tangle diagrams we consider in what follows.
Then, $R$ may be defined by the rules in \cref{fig:rational} (where we set $1/\infty = 0$ and $1/0=\infty=\infty+1=-\infty$). By a slight abuse of notation, we denote by $R(x)$ both the tangle and the tangle diagram (both well-defined up to equivalence).

\begin{figure}[ht]
\begin{minipage}[t]{.47\textwidth}%
\begin{align}%
\label{eq:t0}%
\raisebox{-.3\height}{\includegraphics{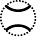}} &\quad=\quad\! R(0) \\
\label{eq:t1}%
\raisebox{-.3\height}{\includegraphics{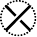}} &\quad=\quad\! R(1) \\
\label{eq:tsum}%
\raisebox{-.3\height}[9ex][5ex]{\mbox{
\begingroup%
  \makeatletter%
  \providecommand\color[2][]{%
    \errmessage{(Inkscape) Color is used for the text in Inkscape, but the package 'color.sty' is not loaded}%
    \renewcommand\color[2][]{}%
  }%
  \providecommand\transparent[1]{%
    \errmessage{(Inkscape) Transparency is used (non-zero) for the text in Inkscape, but the package 'transparent.sty' is not loaded}%
    \renewcommand\transparent[1]{}%
  }%
  \providecommand\rotatebox[2]{#2}%
  \newcommand*\fsize{\dimexpr\f@size pt\relax}%
  \newcommand*\lineheight[1]{\fontsize{\fsize}{#1\fsize}\selectfont}%
  \ifx\svgwidth\undefined%
    \setlength{\unitlength}{48.84769109bp}%
    \ifx\svgscale\undefined%
      \relax%
    \else%
      \setlength{\unitlength}{\unitlength * \real{\svgscale}}%
    \fi%
  \else%
    \setlength{\unitlength}{\svgwidth}%
  \fi%
  \global\let\svgwidth\undefined%
  \global\let\svgscale\undefined%
  \makeatother%
  \begin{picture}(1,0.56771553)%
    \lineheight{1}%
    \setlength\tabcolsep{0pt}%
    \put(0,0){\includegraphics[width=\unitlength,page=1]{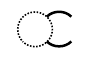}}%
    \put(0.34484228,0.23962483){\color[rgb]{0,0,0}\makebox(0,0)[t]{\lineheight{1.25}\smash{\begin{tabular}[t]{c}$\scriptscriptstyle R(x)$\end{tabular}}}}%
    \put(0,0){\includegraphics[width=\unitlength,page=2]{tanglesum1.pdf}}%
  \end{picture}%
\endgroup%
}} &\quad=\quad\! R(x+1) \\
\label{eq:tinv}%
\parbox{22.5mm}{\footnotesize $R(x)$ mirrored at plane $\langle e_1-e_2, e_3\rangle$}
&\quad=\quad\! R\biggl(\frac1x\biggr)
\end{align}%
\end{minipage}\hspace{.05\textwidth}
\begin{minipage}[t]{.47\textwidth}%
\begin{align}%
\label{eq:tinfty}%
\raisebox{-.3\height}{\includegraphics{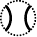}} &\quad=\quad\! R(\infty) \\
\label{eq:t-1}%
\raisebox{-.3\height}{\includegraphics{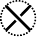}} &\quad=\quad\! R(-1) \\
\label{eq:t*}%
\raisebox{-.4\height}[9ex][5ex]{\mbox{
\begingroup%
  \makeatletter%
  \providecommand\color[2][]{%
    \errmessage{(Inkscape) Color is used for the text in Inkscape, but the package 'color.sty' is not loaded}%
    \renewcommand\color[2][]{}%
  }%
  \providecommand\transparent[1]{%
    \errmessage{(Inkscape) Transparency is used (non-zero) for the text in Inkscape, but the package 'transparent.sty' is not loaded}%
    \renewcommand\transparent[1]{}%
  }%
  \providecommand\rotatebox[2]{#2}%
  \newcommand*\fsize{\dimexpr\f@size pt\relax}%
  \newcommand*\lineheight[1]{\fontsize{\fsize}{#1\fsize}\selectfont}%
  \ifx\svgwidth\undefined%
    \setlength{\unitlength}{27.73159267bp}%
    \ifx\svgscale\undefined%
      \relax%
    \else%
      \setlength{\unitlength}{\unitlength * \real{\svgscale}}%
    \fi%
  \else%
    \setlength{\unitlength}{\svgwidth}%
  \fi%
  \global\let\svgwidth\undefined%
  \global\let\svgscale\undefined%
  \makeatother%
  \begin{picture}(1,1.76144557)%
    \lineheight{1}%
    \setlength\tabcolsep{0pt}%
    \put(0,0){\includegraphics[width=\unitlength,page=1]{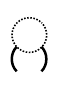}}%
    \put(0.50748722,1.08360011){\color[rgb]{0,0,0}\makebox(0,0)[t]{\lineheight{1.25}\smash{\begin{tabular}[t]{c}$\scriptscriptstyle R(x)$\end{tabular}}}}%
    \put(0,0){\includegraphics[width=\unitlength,page=2]{tanglemult1.pdf}}%
  \end{picture}%
\endgroup%
}} &\quad=\quad\! R\biggl(\frac{x}{x+1}\biggr)\\
\label{eq:m}%
\parbox{20.5mm}{\footnotesize $R(x)$ mirrored at plane $\langle e_1, e_2\rangle$}
&\quad=\quad\! R(-x)\makebox[0pt]{$\phantom{\biggl(\frac1x\biggr)}$}
\end{align}%
\end{minipage}
\caption{The recursive definition of the bijection $R$ between $\Q\cup\{\infty\}$ and equivalence classes of unoriented rational tangles. In (5.4) and (5.8), $e_1,e_2,e_3$ denote the standard basis vectors of~$\mathbb{R}^3$.}
\label{fig:rational}
\end{figure}

As stated, these rules are consistent and determine the correspondence $R$ completely,
but they are somewhat redundant: for example, \cref{eq:tinfty,eq:t-1,eq:t*} can be derived from the other rules.
For simplicity, we will now focus on rational tangles $T$ such that
$R^{-1}(T) \in \mathbb{Q}^+ = \{x\in \Q \mid x > 0\}$
(in particular excluding $R(0)$ and~$R(\infty)$).
The one-to-one correspondence between such rational tangles and $\mathbb{Q}^+$
is completely determined by \cref{eq:t1},\cref{eq:tsum},\cref{eq:tinv},
or by \cref{eq:t1},\cref{eq:tsum},\cref{eq:t*}.

Thompson \cite{thompson} has computed Khovanov homology of all oriented rational tangles.
More precisely, he shows that the complex of a rational tangle
in Bar-Natan's category of cobordisms with dots
is homotopy equivalent to a so-called zigzag complex,
which may be computed recursively.
As a crucial ingredient for the proof of \cref{thm:rationalreplacement}, we require the analog of Thompson's theorem
for Bar-Natan's theory \emph{without} dots (stated below as \cref{thm:thompson}), which is more general than the dotted theory.
Fortunately, Thompson's proof carries over mutatis mutandis to that more general homology theory.

Let us mention that 
Kotelskiy, Watson and Zibrowius have given an elegant way to compute the Bar-Natan complex over $\mathbb{F}_2$
of rational tangles using immersed curves~\cite[Example~6.2]{KWZ}.
Potentially, their methods also work over the integers, which would give an alternative proof of \cref{thm:thompson}.

Let us now have a look at the category $\Cob^{3,\bullet}_{\modl}(4)$ of 4-ended crossingless tangle diagrams with base point, and cobordisms between them. Recall that for two end points, we have found that $\Cob^{3,\bullet}_{\modl}(2)$ is equivalent to $\mathcal{M}_{\Z[G]}$ (see \Cref{sec:zg}). For four end points, a similar strategy of delooping and simplifying cobordisms leads to the following, which is just a reformulation of \cite[Theorem~1.1]{KWZ}:
\begin{theorem}
Consider the $\Z[G]$-enriched category with the two objects \img{tangle0} and~\img{tangleinfty},
and graded $\Z[G]$-morphism modules generated by the identity cobordisms, which we both denote by~$I$,
and the saddle cobordisms $\img{tangle0}\to\img{tangleinfty}$
and $\img{tangleinfty}\to\img{tangle0}$, which we both denote by~$S$, and compositions of these morphisms,
modulo the relations~$S^3 = GS$.
Then the inclusion of the additive graded closure of this category into $\Cob^{3,\bullet}_{\modl}(4)$
is an equivalence of categories.\qed
\end{theorem}
This theorem gives us a compact notation for $\Cob^{3,\bullet}_{\modl}(4)$:
objects are isomorphic to shifted sums of $\img{tangle0}$ and~$\img{tangleinfty}$,
and morphisms are equal to $\Z[G]$-linear combinations of~$I$, $S$ and~$S^2$.
For convenience, we write $D\coloneqq S^2 - G$ (following \cite{KWZ}). Note that
$SD = DS = 0$ and~$D^2 = -GD$.

For the rest of the section, we will for the most part omit homological and quantum gradings without further mention.
\begin{dfn}\label{def:zigag}
A \emph{zigzag complex} is a graded chain complex $(\bigoplus_{i=0}^{n} A_i,\sum_{i=1}^{n} d_i)$ over $\Mat(\Cob^{3,\bullet}_{\modl}(4))$
satisfying the following:
\begin{enumerate}[label=(\roman*)]
\item Each $A_i$ is either \img{tangle0} or \img{tangleinfty} with a quantum and homological degree shift.
\label{eq:2objects}
\item Each $d_i$ has domain and target $A_{i-1}\to A_i$ or $A_i\to A_{i-1}$,
and is one of the following five maps:

\noindent
$S\colon\img{tangleinfty}\to\img{tangle0}$,\hfill
$S^2\colon\img{tangleinfty}\to\img{tangleinfty}$,\hfill
$D\colon\img{tangleinfty}\to\img{tangleinfty}$,\hfill
$S^2\colon\img{tangle0}\to\img{tangle0}$,\hfill
$D\colon\img{tangle0}\to\img{tangle0}$.
\label{eq:lineshapecomplex}
\item Two consecutive differentials $d_i, d_{i+1}$ (no matter what their domain and target are)
are either $S$ and~$D$, or $S^2$ and~$D$.
\label{eq:consecutived}
\item There is at least one differential~$S$.
\label{eq:saddleexists}
\end{enumerate}
The $A_i$ and $d_i$ are considered as part of the data of the zigzag complex.
We say that the zigzag complex $(\bigoplus_{i=0}^{n} A_{n-i}, \sum_{i=1}^{n} d_{n+1-i})$ is obtained from
$(\bigoplus_{i=0}^{n} A_i,\sum_{i=1}^{n} d_i)$ by \emph{reindexing}.
\end{dfn}
Note that reindexing does not change the isomorphism type of the chain complexes.
We would like to depict zigzag complexes using the following type of graphs.
\begin{dfn}
Let us consider a directed finite graph with two \emph{types} of vertices, $\bullet$ and~$\circ$.
Let us call an edge connecting a $\bullet$ and a $\circ$ vertex a \emph{saddle edge}.
Such a graph is called a \emph{zigzag graph} if it satisfies the following conditions.
\begin{enumerate}[label=(\roman*)]
\item \label{eq:lineshape}
The graph has the shape of a line, i.e.~there are exactly two vertices of valency 1 (which we call the \emph{ends}), and all other vertices have valency 2.
\item \label{eq:parity}
There is a partition of edges into \emph{odd} and \emph{even} edges, such
that all saddle edges are odd, and if two edges are adjacent, then one of them is odd and the other one even.
\item \label{eq:orientationsaddles} All saddle edges are directed like this: $\circ\to\bullet$.
\item \label{eq:2vertices} There is at least one saddle edge.
\end{enumerate}
\end{dfn}
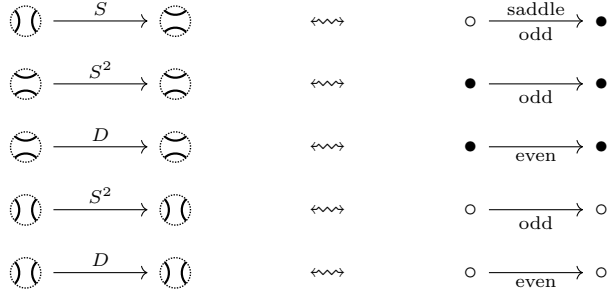
\begin{figure}[b]
\begin{tikzcd}[column sep=large, row sep=tiny]
\img{tangleinfty} \ar[r,"S"] & \img{tangle0} & \leftrightsquigarrow & \circ \arrow{r}{\text{saddle}}[swap]{\text{odd}} & \bullet \\
\img{tangle0} \ar[r,"S^2"] & \img{tangle0} & \leftrightsquigarrow & \bullet \arrow{r}{}[swap]{\text{odd}} & \bullet \\
\img{tangle0} \ar[r,"D"] & \img{tangle0} & \leftrightsquigarrow & \bullet \arrow{r}{}[swap]{\text{even}} & \bullet \\
\img{tangleinfty} \ar[r,"S^2"] & \img{tangleinfty} & \leftrightsquigarrow & \circ \arrow{r}{}[swap]{\text{odd}} & \circ \\
\img{tangleinfty} \ar[r,"D"] & \img{tangleinfty} & \leftrightsquigarrow & \circ \arrow{r}{}[swap]{\text{even}} & \circ 
\end{tikzcd}
\caption{Summary of the correspondence between objects and differentials of zigzag complexes (left column) and vertices and edges of zigzag graphs (right columns).}
\label{table:zigzag}
\end{figure}
Note that because there is at least one saddle edge, the partition of edges into odd and even edges is unique.%
\begin{dfn}
The \emph{graph of a zigzag complex} is the zigzag graph with
a vertex $\bullet$ or $\circ$ corresponding to each $A_i$ that is a shift of \img{tangle0} or \img{tangleinfty}, respectively; and one directed edge corresponding to each~$d_i$.
\end{dfn}
One easily checks that the graph of a zigzag complex really is a zigzag graph.
Moreover, every zigzag graph is the graph of a zigzag complex;
and the graph of a zigzag complex determines the zigzag complex up to reindexing,
and up to global shifts in homological and quantum degree.
The correspondence between zigzag complexes and zigzag graphs is summarized in \cref{table:zigzag}.

Let us now recursively define a zigzag graph $zz(x)$ for all positive rational numbers~$x\in\mathbb{Q}^+$
by the following rules:
\begin{gather}
\label{eq:zz1}
\parbox[t]{.9\textwidth}{$zz(1) = \circ\longrightarrow\bullet$.}\\
\label{eq:zzinv}
\parbox[t]{.9\textwidth}{$zz(1/x)$ is obtained from $zz(x)$ by switching $\circ$ and~$\bullet$, and reversing the directions of all edges.}\\
\label{eq:zz+1}
\parbox[t]{.9\textwidth}{$zz(x + 1)$ is obtained from $zz(x)$ by replacing each edge as shown in \cref{table:x+1}.\footnotemark}
\end{gather}%
\newcounter{disregardrow}\setcounter{disregardrow}{\value{footnote}}
\footnotetext{After the article at hand had appeared as a preprint, it was observed that the left hand side of the fourth row in \cref{table:x+1}
({\footnotesize end}\begin{tikzcd}[ampersand replacement=\&]  {\bullet} \ar[r,"\text{even}",swap] \& {\bullet}\end{tikzcd}{\footnotesize not an end})
does not actually occur in any~$zz(x)$, and so the fourth row may be disregarded~\cite[Lemma~A.1]{arXiv2409.05743}.}
\begin{table}[t]
\begin{tabular}{c|c}
Replace & by  \\\hline 
\raisebox{-4ex}{\rule{0em}{7ex}}%
\begin{tikzcd}{\bullet} \ar[r,"\text{odd}",swap] & {\bullet}\end{tikzcd} &
\begin{tikzcd}\bullet \ar[r,leftarrow] & \circ\ar[r] & \circ \ar[r] & \bullet \end{tikzcd} \\
\raisebox{-4ex}{\rule{0em}{4ex}}%
\makebox[4em][r]{\footnotesize not an end}\begin{tikzcd} {\bullet} \ar[r,"\text{even}",swap] & {\bullet}\end{tikzcd}\makebox[4em][l]{\footnotesize not an end} &
\begin{tikzcd}\bullet \ar[r] & \bullet \end{tikzcd} \\
\raisebox{-4ex}{\rule{0em}{4ex}}%
\makebox[4em][r]{\footnotesize not an end}\begin{tikzcd}  {\bullet} \ar[r,"\text{even}",swap] & {\bullet}\end{tikzcd}\makebox[4em][l]{\footnotesize end} &
\begin{tikzcd}\phantom{\circ}\ar[r,phantom] & \bullet \ar[r] & \bullet \ar[r,leftarrow] & \circ \end{tikzcd} \\
\raisebox{-4ex}{\rule{0em}{4ex}}%
\makebox[4em][r]{\footnotesize end}\begin{tikzcd}  {\bullet} \ar[r,"\text{even}",swap] & {\bullet}\end{tikzcd}\makebox[4em][l]{\footnotesize not an end}&
\begin{tikzcd}\circ \ar[r] & \bullet \ar[r] & \bullet \ar[r,phantom] & \phantom{\circ}\end{tikzcd} \\
\raisebox{-4ex}{\rule{0em}{4ex}}%
\begin{tikzcd}\circ \ar[r] & \bullet\end{tikzcd} &
\begin{tikzcd}\circ \ar[r] & \circ\ar[r] &  \bullet \ar[r,phantom] & \phantom{\circ}\end{tikzcd}  \\
\begin{tikzcd}\circ \ar[r] & \circ\end{tikzcd} &
\begin{tikzcd}\circ \ar[r] & \circ \end{tikzcd}
\end{tabular}
\caption{How to obtain $zz(x+1)$ from~$zz(x)$: each edge $e$ in $zz(x)$ falls into a unique one of the six cases shown in the left column of the table. Apply the rule, i.e.~replace $e$ by the graph $\Gamma_e$ in the right column in the same row. In this way, each of the two vertices $v, w$ adjacent to $e$ in $zz(x)$ are replaced by vertices $v_e, w_e$ in~$zz(x+1)$, namely the leftmost and the rightmost vertex in~$\Gamma_e$. If two edges $e$ and $f$ of $zz(x)$ are adjacent to a common vertex~$v$, identify the vertices $v_e$ and~$v_f$ in~$zz(x+1)$. Note that this is possible since $v_e$ and~$v_f$ always have the same type: in fact, $v_e$ has the same type as $v$ if $v$ (or equivalently,~$v_e$) is not an end.%
\protect\footnotemark[\value{disregardrow}]}\label{table:x+1}
\end{table}
Note that these cases are exhaustive, since two adjacent $\bullet$-vertices cannot both be ends (because there is at least one saddle edge),
and since a saddle edge is always directed from $\circ$ to~$\bullet$.

Let us check that $zz$ is well-defined. Indeed, $\circ\to\bullet$ is a zigzag graph,
and one may verify that \cref{eq:zzinv,eq:zz+1} map zigzag graphs to zigzag graphs.
Every positive rational number can be obtained from $1$ by a sequence of $x\mapsto 1/x$ and~$x\mapsto x + 1$.
Moreover, that sequence is unique up to inserting or removing two consecutive~$x\mapsto 1/x$.
Since applying \cref{eq:zzinv} twice has no effect,
\cref{eq:zz1,eq:zzinv,eq:zz+1} indeed define $zz(x)$ for every positive rational~$x$.%

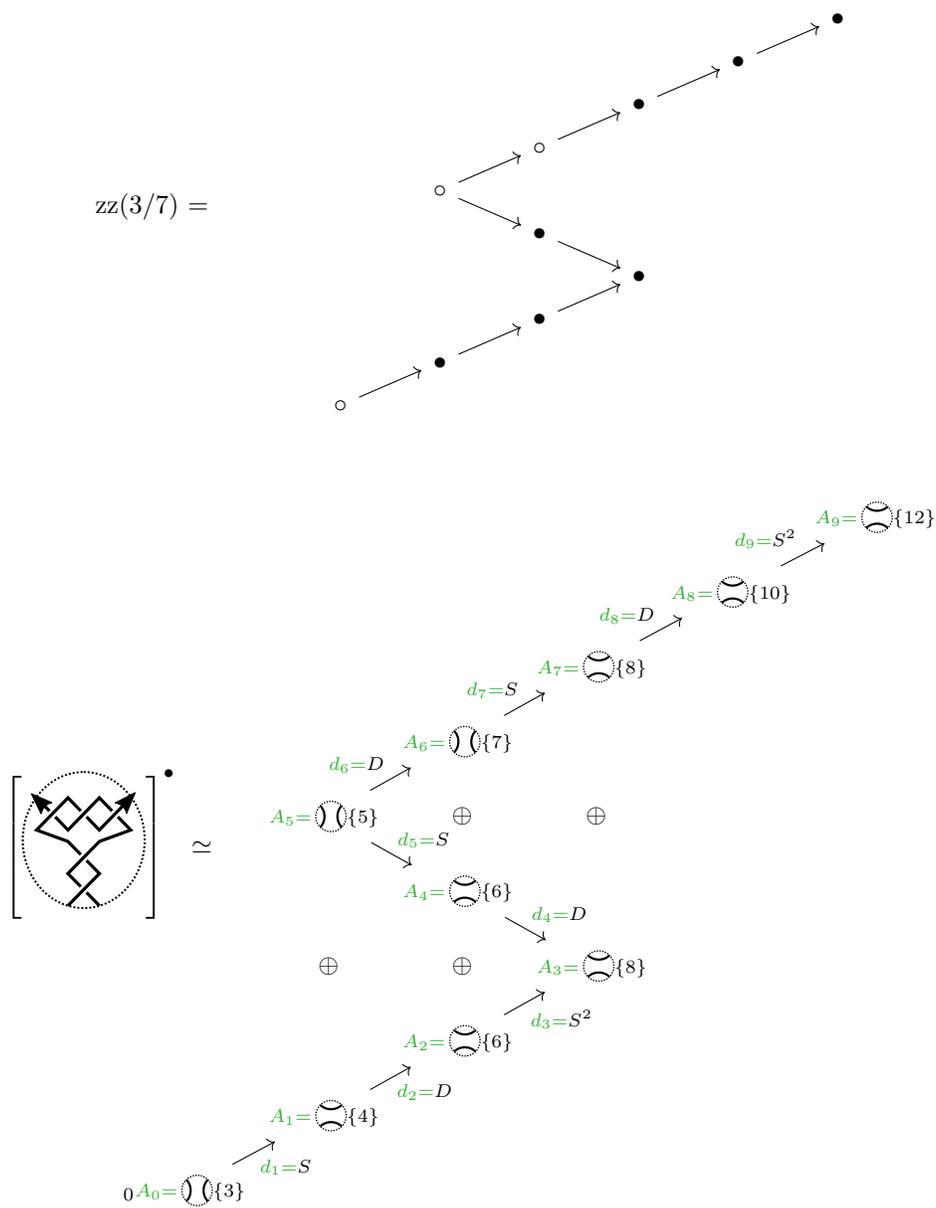
\begin{figure}[p]
zz(3/7) =\hspace{14mm}
\begin{tikzcd}[row sep=tiny]
&&&&& \bullet \\
&&&& \bullet \arrow[ru] & \\
&&& \bullet \arrow[ru] & \\
&& \circ \arrow[ru] && \\
& \circ \arrow[ru]\arrow[rd] &&& \\
&& \bullet \arrow[rd] && \\
&&& \bullet & \\
&& \bullet \arrow[ru] && \\
& \bullet \arrow[ru] &&& \\
\circ \arrow[ru] &&&&   
\end{tikzcd}
\vspace{9mm}

\newcommand{\Aeq}[1]{\scriptstyle\textcolor{annotation}{A_{#1}=}\,}
\newcommand{\deq}[1]{\textcolor{annotation}{d_{#1}=}}
\makebox[0mm][l]{$\left[\raisebox{-.4\height}{\includegraphics{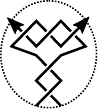}}\right]^{\bullet}\ \simeq$}
\tikzcdset{column sep/tinytiny/.initial=.1ex}
\hfill
\begin{tikzcd}[column sep = tinytiny, row sep=small]
&&&&& \Aeq{9}\img{tangle0}\{12\} \\
&&&&  \Aeq{8}\img{tangle0}\{10\} \arrow{ru}{\deq9S^2} & \\
&&&   \Aeq{7}\img{tangle0}\{8\} \arrow{ru}{\deq8D} & \\
&&    \Aeq{6}\img{tangleinfty}\{7\} \arrow{ru}{\deq7S} && \\
&     \Aeq{5}\img{tangleinfty}\{5\} \arrow{ru}{\deq6D}\arrow{rd}{\deq5S} &~\oplus&~\oplus& \\
&&    \Aeq{4}\img{tangle0}\{6\} \arrow{rd}{\deq4D} && \\
&\ \oplus&~\oplus&   \Aeq{3}\img{tangle0}\{8\} & \\
&&    \Aeq{2}\img{tangle0}\{6\} \arrow[ru,swap,"\deq3S^2"] && \\
&     \Aeq{1}\img{tangle0}\{4\} \arrow[ru,swap,"\deq2D"] &&& \\
      _0\Aeq{0}\img{tangleinfty}\{3\} \arrow[ru,swap,"\deq1S"] &&&&   
\end{tikzcd}
\vspace{8mm}
\caption{As an illustration of the correspondence between zigzag graphs and zigzag complexes,
and of \cref{thm:thompson}:
on the top, the zigzag graph $zz(3/7)$; on the bottom, a zigzag complex that is homotopy equivalent to the
Bar-Natan complex of the rational tangle $R(3/7)$ endowed with some orientation.
The left subscript gives the homological degree.}
\label{fig:37}
\end{figure}
We are now ready to state our generalization of Thompson's theorem.
\begin{theorem}\label{thm:thompson}
Let $R(x)$ be the unoriented rational tangle corresponding to a positive rational number~$x$.
Let $T$ be the tangle $R(x)$ equipped with some orientation~$o$.
Then, the Bar-Natan complex $[T]^{\bullet}$ is homotopy equivalent to a zigzag complex with graph~$zz(x)$.
\end{theorem}
\begin{remark}
Since the Bar-Natan complex of the mirror image of a tangle is isomorphic to the dual of the Bar-Natan complex of that tangle, 
\cref{thm:thompson} yields a rather simple representative of the homotopy equivalence class of the Bar-Natan complex of any rational tangle, up to global shifts in homological and quantum degree.
These shifts depend on the orientation of the tangle; since they do not matter for our work on~$\lambda$, we will neglect them. Thompson computes the shifts in \cite[Theorem~5.1]{thompson}.
\end{remark}
The proof will use Bar-Natan's computation method of delooping and Gaussian elimination \cite{BN2}.
We have described delooping in \cref{fig:delooping} in \cref{subsec:defzgcomplex}.
By Gaussian elimination, we mean the following.
\begin{lem}\label{lem:gaussian}
Assume $(C,d)$ is a chain complex in some additive category taking the following form:
\[
\begin{tikzcd}[row sep=tiny]
 && X \ar[rr,"c"]\ar[rrdd,"d",pos=0.15] && Z \ar[rd,"g"] \\
\ldots \ar[r] & C_{i-1} \ar[ru,"a"]\ar[rd,"b"]
 & \oplus           && \oplus    & C_{i+2} \ar[r] & \ldots \\
 && Y \ar[rr,"e"]\ar[rruu,"f",crossing over,pos=0.25] && W \ar[ru,"h"]
\end{tikzcd}
\]
with $e$ an isomorphism. Then $C$ is homotopy equivalent to
\[
\begin{tikzcd}
\ldots \ar[r] &
C_{i-1} \ar[r,"a"] & X \ar[rr,"c-fe^{-1}d"] && Z \ar[r,"g"] & C_{i+2} \ar[r] & \ldots.
\end{tikzcd}
\myqed
\]
\end{lem}
Using Gaussian elimination as stated in the above lemma, one may eliminate the domain and target of an isomorphism in a chain complex, by paying the price of introducing a new differential $f\circ e^{-1}\circ d$.
\begin{proof}[Proof of \cref{thm:thompson}]
We proceed by induction over the number of transformations
$y\mapsto 1/y$ and $y \mapsto y + 1$ necessary to reach $x$ from~$1$.
For~$x = 1$, observe that $[T]^{\bullet}$ is homotopy equivalent (in fact isomorphic) to a zigzag complex
with graph equal to~$zz(1)$, by the definition of the Bar-Natan complex of a tangle consisting of a single crossing.

So let us now assume that the statement holds
for~$x$. 
To show that the statement also holds for~$1/x$, observe that the tangle diagrams
$R(x)$ and $R(1/x)$ are related by mirroring at the line
$\R\left(\begin{smallmatrix}1\\-1\end{smallmatrix}\right)$.
Equivalently, one may first rotate $R(x)$ by~$\pi / 2$,
getting a tangle diagram~$D_T$, and then switch all crossings in~$D_T$, getting~$R(1/x)$.
For all three tangle diagrams~$R(x)$, $D_T$ and~$R(1/x)$, the base point is the lower left end point
(as it is for all four-ended tangles, by the convention fixed at the beginning of this section).
Let $C$ be the zigzag complex homotopy equivalent to~$[R(x)]^{\bullet}$,
provided by the induction hypothesis.
Since the Bar-Natan complex is defined geometrically, a complex $C'$ homotopy equivalent to $[D_T]^{\bullet}$ may be
obtained from $C$ simply by rotating all crossingless tangle diagrams and all cobordisms in $C$
by~$\pi/2$. This switches \img{tangle0} and \img{tangleinfty}. The effect on cobordisms is a little more subtle: since the base point always remains at the lower left end point, rotation by $\pi/2$ does not commute with the action of~$G$.
Thus one finds (using the $4Tu$-relation, see \cref{fig:relations}) that the rotation sends $I$ to~$I$, $S$ to~$S$, $D\colon \img{tangleinfty}\to\img{tangleinfty}$ to
$D\colon \img{tangle0}\to\img{tangle0}$, but 
$D\colon \img{tangle0}\to\img{tangle0}$ to $-D\colon \img{tangleinfty}\to\img{tangleinfty}$.
However, because of the linear shape of~$C'$, there is a simple change of basis that gets rid of the introduced minus signs, yielding a complex $C''$ isomorphic to~$C'$. Finally, $[R(1/x)]^{\bullet}$ is homotopy equivalent to the dual of~$C''$; naively, the dual of $C''$ is obtained by simply reversing the direction of all differentials. As an upshot of this discussion, the dual of $C''$ is a zigzag complex corresponding to the zigzag graph obtained from $zz(x)$ by switching $\circ$ and $\bullet$ (the effect of the rotation) and redirecting all arrows (the effect of switching all crossings). This shows that the statement holds for~$1/x$.

\begin{figure}[t]
\[
\newcommand{\imgx}[2]{\raisebox{-.3\height}{\includegraphics[scale=#1]{#2}}}
\begin{tikzcd}[column sep=3em]
\textcolor{annotation}{\scriptstyle C_{6,0}} & \textcolor{annotation}{\scriptstyle C_{5,0}} & \textcolor{annotation}{\scriptstyle C_{4,0}} & \textcolor{annotation}{\scriptstyle C_{3,0}} & \textcolor{annotation}{\scriptstyle C_{2,0}} & \textcolor{annotation}{\scriptstyle C_{1,0}} \ar[r,white,"\textcolor{annotation}{\beta_{0,0}}"]& \textcolor{annotation}{\scriptstyle C_{0,0}} \\[-5ex]
  \img{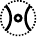}\ar[r,"{\imgx{.2}{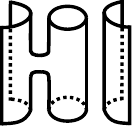} - \imgx{.2}{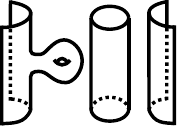}}"]\ar[d,"\img{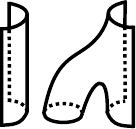}"]
& \img{tangleinftycircle}\ar[r,"\img{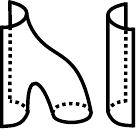}"]\ar[d,"\img{grid4_}"]
& \img{tangleinfty}
\ar[d,"S"]
& \img{tangleinfty}\ar[r,"\img{grid1_}",leftarrow]\ar[d,"S"]
& \img{tangleinftycircle}\ar[r,"{\imgx{.2}{grid2_} - \imgx{.2}{grid3_}}",leftarrow]\ar[d,"\img{grid4_}"]
& \img{tangleinftycircle}\ar[r,"\img{grid2_}",leftarrow]\ar[d,"\img{grid4_}"]
& \img{tangleinftycircle}\ar[d,"\img{grid4_}"] \\
  \img{tangleinfty}\ar[r,"D",swap]
& \img{tangleinfty}\ar[r,"S",swap]
& \img{tangle0}\ar[r,"D",leftarrow,swap]
& \img{tangle0}\ar[r,"S",leftarrow,swap]
& \img{tangleinfty}\ar[r,"D",leftarrow,swap]
& \img{tangleinfty}\ar[r,"S^2",leftarrow,swap]
& \img{tangleinfty} \\[-5ex]
\textcolor{annotation}{\scriptstyle C_{6,1}} & \textcolor{annotation}{\scriptstyle C_{5,1}} & \textcolor{annotation}{\scriptstyle C_{4,1}} & \textcolor{annotation}{\scriptstyle C_{3,1}} & \textcolor{annotation}{\scriptstyle C_{2,1}} & \textcolor{annotation}{\scriptstyle C_{1,1}} \ar[r,white,"\textcolor{annotation}{\beta_{0,1}}",swap]& \textcolor{annotation}{\scriptstyle C_{0,1}}
\end{tikzcd}
\]
\caption{An example of a grid complex appearing in the proof of \cref{thm:thompson}.
The lower row is a zigzag complex corresponding to~$zz(5/2)$.}
\label{fig:ladderexample}
\end{figure}
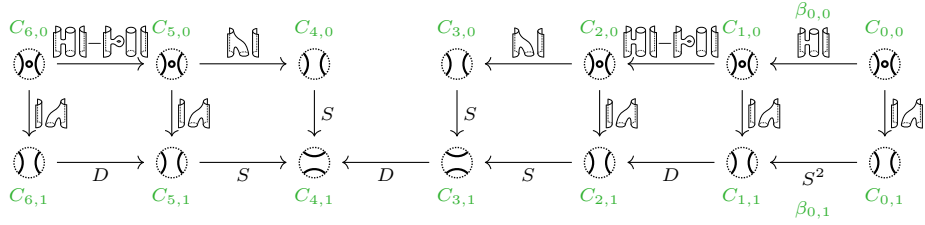

Let us now show that the statement holds for~$x+1$.
Let $T'$ be the tangle $R(x+1)$ equipped with the orientation induced by the orientation of~$T$.
Let $\mathcal{D}$ be the 2-input planar arc diagram
\[
\raisebox{-.4\height}{
\begingroup%
  \makeatletter%
  \providecommand\color[2][]{%
    \errmessage{(Inkscape) Color is used for the text in Inkscape, but the package 'color.sty' is not loaded}%
    \renewcommand\color[2][]{}%
  }%
  \providecommand\transparent[1]{%
    \errmessage{(Inkscape) Transparency is used (non-zero) for the text in Inkscape, but the package 'transparent.sty' is not loaded}%
    \renewcommand\transparent[1]{}%
  }%
  \providecommand\rotatebox[2]{#2}%
  \newcommand*\fsize{\dimexpr\f@size pt\relax}%
  \newcommand*\lineheight[1]{\fontsize{\fsize}{#1\fsize}\selectfont}%
  \ifx\svgwidth\undefined%
    \setlength{\unitlength}{84.58745894bp}%
    \ifx\svgscale\undefined%
      \relax%
    \else%
      \setlength{\unitlength}{\unitlength * \real{\svgscale}}%
    \fi%
  \else%
    \setlength{\unitlength}{\svgwidth}%
  \fi%
  \global\let\svgwidth\undefined%
  \global\let\svgscale\undefined%
  \makeatother%
  \begin{picture}(1,0.69442007)%
    \lineheight{1}%
    \setlength\tabcolsep{0pt}%
    \put(0,0){\includegraphics[width=\unitlength,page=1]{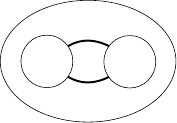}}%
    \put(0.11096711,0.19930223){\makebox(0,0)[lt]{\lineheight{1.25}\smash{\begin{tabular}[t]{l}\textbf{*}\end{tabular}}}}%
    \put(0.58228307,0.19762543){\makebox(0,0)[lt]{\lineheight{1.25}\smash{\begin{tabular}[t]{l}\textbf{*}\end{tabular}}}}%
    \put(0,0){\includegraphics[width=\unitlength,page=2]{arcdiagramclosure2.pdf}}%
    \put(0.02197702,0.12917546){\makebox(0,0)[lt]{\lineheight{1.25}\smash{\begin{tabular}[t]{l}\textbf{*}\end{tabular}}}}%
  \end{picture}%
\endgroup%
}.
\]
Then $T'$ is equivalent to $\mathcal{D}(T, R(1))$, and so
$[T']^{\bullet} \simeq \mathcal{D}([T]^{\bullet}, B)$, for $B$ 
the Bar-Natan complex of a single crossing, which is, up to global shifts, equal to
\[
B_0=\img{tangleinfty}\xrightarrow{\quad S\quad }\img{tangle0}\{1\}=B_1.
\]%
Less formally, $[T']^{\bullet}$ is homotopy equivalent to the
chain complex obtained by
taking the tensor product of the chain complexes
$[T]^{\bullet}$ and~$B$, while at the same time (on the level of tangles)
gluing two pairs of end points.
By the induction hypothesis, $[T]^{\bullet}$ is homotopy equivalent to a zigzag complex
$(\bigoplus_{i=0}^n A_i, \sum_{i=1}^n d_i)$ with graph~$zz(x)$.
Thus the complex $[T']^{\bullet}$ is homotopy equivalent to the complex
\[
\bigoplus_{\substack{i\in \{0, \ldots, n\}\\ j\in\{0,1\}}} C_{ij}
\]
with $C_{ij} = A_i \otimes B_j$
and the following differentials:
$\alpha_i\colon C_{i,0} \to C_{i,1}$ given by $(-1)^i \text{id}_{A_i} \otimes S$, and
$\beta_{i,j}$ a map $C_{i,j}\to C_{i-1,j}$ or $C_{i-1,j}\to C_{i,j}$ given by $d_i \otimes \text{id}_{B_j}$.
One may conveniently depict this complex in a grid with two rows and $n+1$ columns.
An example is shown in \cref{fig:ladderexample}.
The lower row of this grid is the subcomplex $(C_{i,1}, \beta_{i,1})$, which equals the complex $(A_i\{1\}, d_i)$.
In the upper row, on the other hand, $C_{i,0} = \img{tangleinfty}$ if $A_i = \img{tangle0}$
and $C_{i,0} = \img{tangleinftycircle}$ if $A_i = \img{tangleinfty}$; moreover, one easily checks
the correspondence between maps in the lower row and the upper row 
shown in the first two columns of \cref{fig:simplifications}.
Note that for these and the following calculations, it matters where the base point is (in the south west), and where the single crossing is attached to $T$ to form $T'$ (on the right).
Our strategy is now to simplify $C_{ij}$ step by step: first by delooping, then by Gaussian elimination, and finally by basis changes. Thereby we will produce a homotopy equivalence between $C_{ij}$ and a zigzag complex with graph~$zz(x+1)$.
\begin{table}[tb]%
\newcommand{\mysmallmatrix}[2]{(\begin{smallmatrix}#1 & #2\end{smallmatrix})}%
\newcommand{\mysmallmatrixfour}[4]{\bigl(\begin{smallmatrix}#1 & #2 \\ #3 & #4\end{smallmatrix}\bigr)}%
\newcommand{\imgx}[2]{\raisebox{-.3\height}{\includegraphics[scale=#1]{#2}}}
\noindent
\begin{tabular}{c|c|c|c}%
\parbox{3.6em}{\centering Edge\\ in $zz(x)$} & Square in $C_{ij}$ & Square in $C'_{ij}$ & Square in $C''_{ij}$ \\\hline\hline
\begin{tikzcd}[ampersand replacement=\&] {\circ} \ar[r] \& {\bullet}\end{tikzcd} &
\begin{tikzcd}[ampersand replacement=\&]
\img{tangleinftycircle} \ar[d,"{\pm\imgx{.2}{grid4_}}"]\ar[r,"{\imgx{.2}{figures/grid1_}}"] \& \img{tangleinfty}\ar[d,"\mp S"] \\
\img{tangleinfty} \ar[r,"S"] \& \img{tangle0}
\end{tikzcd} &
\begin{tikzcd}[ampersand replacement=\&]\img{tangleinfty}^{\oplus 2} \ar[d,"\pm\mysmallmatrix{D}{I}"]\ar[r,"\mysmallmatrix{0}{I}"] \& \img{tangleinfty}\ar[d,"\mp S"]\\
\img{tangleinfty} \ar[r,"S"] \& \img{tangle0}
\end{tikzcd} &
\begin{tikzcd}[ampersand replacement=\&]\img{tangleinfty} \ar[r,"-D"] \& \img{tangleinfty}\ar[d,"\mp S"]\\
 0 \& \img{tangle0}
\end{tikzcd}
\\\hline
\begin{tikzcd}[ampersand replacement=\&] {\bullet} \ar[r,"\text{odd}",swap] \& {\bullet}\end{tikzcd} &
\begin{tikzcd}[ampersand replacement=\&]
\img{tangleinfty} \ar[d,"\pm S"] \ar[r,"G"] \& \img{tangleinfty}\ar[d,"\mp S"]\\
\img{tangle0} \ar[r,"S^2"] \& \img{tangle0}
\end{tikzcd} &  same as $C_{ij}$ & same as $C_{ij}$
\\\hline
\begin{tikzcd}[ampersand replacement=\&] {\bullet} \ar[r,"\text{even}",swap] \& {\bullet}\end{tikzcd} &
\begin{tikzcd}[ampersand replacement=\&]
\img{tangleinfty} \ar[d,"\pm S"] \& \img{tangleinfty}\ar[d,"\mp S"] \\
\img{tangle0} \ar[r,"D"] \& \img{tangle0} \end{tikzcd}
 &  same as $C_{ij}$ & same as $C_{ij}$
\\\hline
\begin{tikzcd}[ampersand replacement=\&] {\circ} \ar[r,"\text{odd}",swap] \& {\circ}\end{tikzcd} &
\begin{tikzcd}[ampersand replacement=\&]
\img{tangleinftycircle} \ar[d,"{\pm\imgx{.2}{grid4_}}"]\ar[r,"{\imgx{.2}{grid2_}}"] \& \img{tangleinftycircle}\ar[d,"{\mp\imgx{.2}{grid4_}}"] \\
\img{tangleinfty} \ar[r,"S^2"] \& \img{tangleinfty}
\end{tikzcd} &
  \begin{tikzcd}[ampersand replacement=\&]\img{tangleinfty}^{\oplus 2} \ar[d,"{\pm\mysmallmatrix{D}{I}}"]\ar[r,"\mysmallmatrixfour{0}{I}{0}{G}"] \& \img{tangleinfty}^{\oplus 2}\ar[d,"{\mp\mysmallmatrix{D}{I}}"] \\
\img{tangleinfty} \ar[r,"S^2"] \& \img{tangleinfty}
\end{tikzcd} 
&
\begin{tikzcd}[ampersand replacement=\&]\img{tangleinfty} \ar[r,"-D"] \& \img{tangleinfty}\\ 0 \& 0\end{tikzcd}
\\\hline
\begin{tikzcd}[ampersand replacement=\&] {\circ} \ar[r,"\text{even}",swap] \& {\circ}\end{tikzcd} &
\begin{tikzcd}[ampersand replacement=\&]
\img{tangleinftycircle} \ar[d,"{\pm\imgx{.2}{grid4_}}"]\ar[r,outer sep=1ex,"{\imgx{.2}{grid2_} - \imgx{.2}{grid3_}}"] \& \img{tangleinftycircle}\ar[d,"{\mp\imgx{.2}{grid4_}}"] \\
\img{tangleinfty} \ar[r,"D"] \& \img{tangleinfty}
\end{tikzcd} &
  \begin{tikzcd}[ampersand replacement=\&]\img{tangleinfty}^{\oplus 2} \ar[d,"{\pm\mysmallmatrix{D}{I}}"]\ar[r,outer sep=.5ex,"\mysmallmatrixfour{-G}{I}{0}{0}"] \& \img{tangleinfty}^{\oplus 2}\ar[d,"{\mp\mysmallmatrix{D}{I}}"] \\
\img{tangleinfty} \ar[r,"D"] \& \img{tangleinfty}
\end{tikzcd} 
&
\begin{tikzcd}[ampersand replacement=\&]\img{tangleinfty} \ar[r,"-S^2"] \& \img{tangleinfty}\\ 0 \& 0 \end{tikzcd}
\\
\end{tabular}
\caption{Simplifications done in the proof of \cref{thm:thompson}.}
\label{fig:simplifications}
\end{table}

Let us first simplify $C_{ij}$ by delooping. This yields another grid complex $C'_{ij}$
with identical objects, except that $\img{tangleinftycircle}$ is replaced by
$C'_{i0} = \img{tangleinfty}\oplus \img{tangleinfty}$. The lower rows $C'_{i0}$ and $C_{i0}$ are identical,
including the differentials.
The third column of \cref{fig:simplifications} shows how the upper row $C'_{i1}$
is determined by the lower row~$C'_{i0}$.
All of the vertical differentials $\alpha'_i\colon C'_{i1}\to C'_{i0}$ are
either~$(-1)^i S$, or equal to
\[
(-1)^i \begin{pmatrix} D & I\end{pmatrix}.
\]
We leave it to the reader to verify these calculations.

As a second step, let us simultaneously eliminate all the $I$ entries in $\alpha'_i$ by Gaussian elimination,
and denote the resulting grid complex by~$C''_{ij}$. \cref{fig:simplifications} shows how $C''_{ij}$ is determined by 
$C'_{ij}$. Once again, we invite the reader to check these calculations.
Note that if there is a differential $\delta$ in $C'_{ij}$ whose target is one of the $\img{tangleinfty}$ objects
in the lower row, then there are only two possible cases.
The first case is that the domain of $\delta$ is equal to the domain of one of the maps $I$ that are being eliminated.
In that case, Gaussian elimination produces a new differential, which is the reason that the squares of $C''_{ij}$ shown in the last two rows of \cref{fig:simplifications} are $-D$ and~$-S^2$.
The second case is that the domain of $\delta$ is another $\img{tangleinfty}$ object in the
lower row (since there are no edges $\bullet\to\circ$ in~$zz(x)$, it cannot be a $\img{tangle0}$ object). But in that case, the domain itself is being eliminated, so in this case, Gaussian elimination does not produce new differentials.

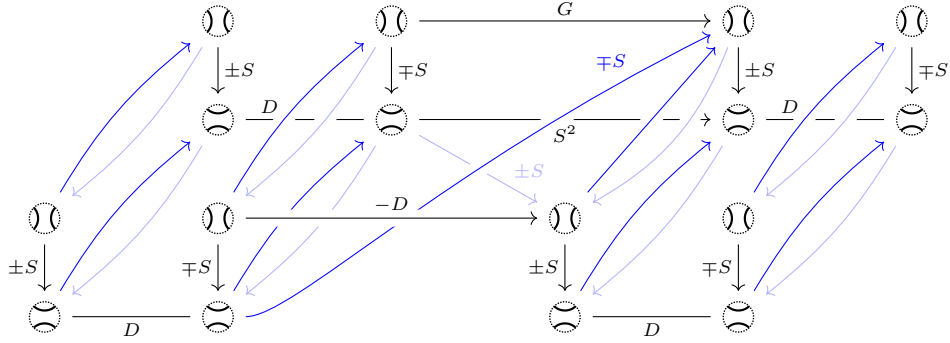
\begin{figure}[tb]%
\begin{tikzcd}[column sep=4.5em]
 & \img{tangleinfty} \ar[ldd,firstmap,leftarrow,bend right=10]\ar[ldd,secondmap,bend left=10]\ar[d,"\pm S"] & \img{tangleinfty} \ar[rr,"G"] \ar[d,"\mp S"]   && \img{tangleinfty}\ar[d,"\pm S"] & \img{tangleinfty} \ar[d,"\mp S"]\\
 & \img{tangle0} \ar[ldd,firstmap,leftarrow,bend right=10]\ar[ldd,secondmap,bend left=10]\ar[r,dash,"D",pos=0.2]   & \img{tangle0} \ar[rr,"S^2",swap]\ar[rd,"\pm S" near end,secondmap]\ar[ldd,firstmap,leftarrow,bend right=10,crossing over]\ar[ldd,secondmap,crossing over,bend left=10] &&   \img{tangle0}\ar[ldd,firstmap,leftarrow,bend right=10]\ar[ldd,secondmap,bend left=10] \ar[r,dash,"D",pos=0.2]& \img{tangle0}\ar[ldd,firstmap,leftarrow,bend right=10]\ar[ldd,secondmap,bend left=10]\\
\img{tangleinfty}\ar[d,"\pm S",swap] & \img{tangleinfty} \ar[d,"\mp S",swap] \ar[firstmap,leftarrow,crossing over,from=ruu,bend right=10]\ar[secondmap,crossing over,bend left=10,from=ruu]  && \img{tangleinfty}\ar[d,"\pm S",swap] \ar[firstmap,leftarrow,crossing over,from=ruu]\ar[secondmap,crossing over,bend left=19.5,from=ruu] & \img{tangleinfty}\ar[d,"\mp S",swap]\ar[firstmap,leftarrow,crossing over,from=ruu,bend right=10]\ar[secondmap,crossing over,bend left=10,from=ruu]\\
\img{tangle0} \ar[r,dash,"D",swap]&\img{tangle0} \ar[rrruuu,firstmap,"\mp S",pos=0.9,controls={+(1,0) and +(-4,-2)},crossing over]
\ar[from=u,to=rru,"-D",crossing over] &&    \img{tangle0} \ar[r,dash,"D",swap] & \img{tangle0}  \\
\end{tikzcd}
\caption{A change of basis required in the proof of \cref{thm:thompson}.
The back face is the complex~$C''_{ij}$, and the front face is the simplified complex~$C'''_{ij}$.
Differentials are drawn in black.
The light and dark blue arrows show mutually inverse isomorphisms between those two complexes.
Unlabeled arrows signify identity maps~$I$. The arrows labeled $D$ on the lower level, on the left,
have no arrowhead, because they may point in either direction (either both point to the right, or both to the left).
The same holds for the arrows labeled $D$ on the lower level, on the right.}
\label{fig:changeofbasis}
\end{figure}%
The third step is to apply changes of basis. At each square $W$ of $C''_{ij}$ arising from an odd edge in $zz(x)$ between $\bullet$ vertices, i.e.~a square
\[
\begin{tikzcd}[ampersand replacement=\&]
\img{tangleinfty} \ar[d,"\pm S"]\ar[r,"G"] \& \img{tangleinfty}\ar[d,"\mp S"] \\
\img{tangle0} \ar[r,"S^2"] \& \img{tangle0}, \end{tikzcd}
\]
we apply the change of basis shown in \cref{fig:changeofbasis}.
Note that \cref{fig:changeofbasis} also shows the two squares adjacent to~$W$. Of course, if $W$ is at one end of the complex, then one of those squares does not exist. The point is that several of those basis changes are compatible, and may be made simultaneously, because the basis change maps on the adjacent objects are the identity~$I$.

This results in a grid complex~$C'''_{ij}$. Since $C'''_{ij}$ is homotopy equivalent to~$C_{ij}$, it just remains to verify that $C'''_{ij}$ is isomorphic to a zigzag complex corresponding to~$zz(x+1)$. The first two columns of \cref{table:fromxtox+1} summarize how $C'''_{ij}$ is determined by~$zz(x)$. From that table, it should be evident that $C'''_{ij}$ is a `linear' complex, i.e.~it is isomorphic to a chain complex of the form $(\bigoplus_i A_i, \sum_i d_i)$, with each $A_i$ either \img{tangleinfty} or \img{tangle0}, and each $d_i$ a map $A_{i-1}\to A_i$ or~$A_i\to A_{i-1}$. It is not quite a zigzag complex, yet, since the differentials are equal to $\pm S$, $\pm S^2$,~$\pm D$. However, due to the linear shape of the complex, one easily finds a basis change that gets rid of all minus signs. So, $C'''_{ij}$ truly is the zigzag complex of some zigzag graph~$\Gamma$. 

How edges of $zz(x)$ determine small subgraphs of $\Gamma$ is shown in the third column of \cref{table:fromxtox+1}.
However, the way these subgraphs fit together is somewhat cumbersome: subgraphs coming from adjacent edges in $zz(x)$ need to be glued together either along a single $\circ$-vertex, or along a vertical downwards edge from $\circ$ to~$\bullet$. As a remedy, we remove the two $\circ$-vertices from the subgraph induced by an even $\bullet\longrightarrow\bullet$ edge, unless those $\circ$-vertices are at one of the ends of the graph. After this reinterpretation, we obtain the subgraphs shown in the fourth column of \cref{table:fromxtox+1} (in the fourth column, we also disregard the grid structure, which was still shown in the third column). To form $\Gamma$ from the subgraphs in the fourth column, one simply glues along a single vertex. Now, the first and fourth column of \cref{table:fromxtox+1} are precisely identical with the rules of \cref{table:x+1}. This shows that $\Gamma = zz(x+1)$, concluding the proof.
\end{proof}
\begin{table}[tb]
\begin{tabular}{c|l|c|c}
\parbox{3.6em}{\centering Edge\\ in $zz(x)$} & \multicolumn{1}{c|}{Square in $C'''_{ij}$} & Subgraph of $\Gamma$ & \parbox{10.6em}{\centering Subgraph of~$\Gamma$,\\ reinterpreted} \\[2ex]\hline\hline
\begin{tikzcd}[ampersand replacement=\&] {\circ} \ar[r] \& {\bullet}\end{tikzcd} &
\begin{tikzcd}[ampersand replacement=\&]\img{tangleinfty} \ar[r,"-D"] \& \img{tangleinfty}\ar[d,"\mp S"]\\
 0 \& \img{tangle0} \end{tikzcd} &
\begin{tikzcd}[ampersand replacement=\&]\circ \ar[r] \& \circ\ar[d]\\
   \& \bullet \end{tikzcd} &
\begin{tikzcd}[ampersand replacement=\&]\circ \ar[r] \& \circ\ar[r] \& \bullet \end{tikzcd} 
\\\hline
\begin{tikzcd}[ampersand replacement=\&] {\bullet} \ar[r,"\text{odd}",swap] \& {\bullet}\end{tikzcd} &
\begin{tikzcd}[ampersand replacement=\&]\img{tangleinfty}\ar[d,"\pm S"] \ar[r,"-D"] \& \img{tangleinfty}\ar[d,"\mp S"]\\
 \img{tangle0} \& \img{tangle0} \end{tikzcd} &
\begin{tikzcd}[ampersand replacement=\&]\circ\ar[d] \ar[r] \& \circ\ar[d]\\
\bullet  \& \bullet \end{tikzcd} &
\begin{tikzcd}[ampersand replacement=\&]\bullet \& \circ\ar[l]\ar[r] \& \circ\ar[r] \& \bullet \end{tikzcd} 
\\\hline
\begin{tikzcd}[ampersand replacement=\&, row sep = -.2ex] {\bullet} \ar[r,"\text{even}",swap] \& {\bullet}\\
\scriptstyle x \& \scriptstyle y\end{tikzcd} &
\begin{tikzcd}[ampersand replacement=\&]\img{tangleinfty}\ar[d,"\pm S"]  \& \img{tangleinfty}\ar[d,"\mp S"]\\
 \img{tangle0} \ar[r,"D"]\& \img{tangle0} \end{tikzcd} &
\begin{tikzcd}[ampersand replacement=\&]\circ\ar[d]  \& \circ\ar[d]\\
\bullet \ar[r] \& \bullet \end{tikzcd} &
\begin{tabular}{@{}l@{}l@{}}
\begin{tikzcd}[ampersand replacement=\&]\bullet \ar[r] \& \bullet \& \circ\ar[l] \end{tikzcd} & {\footnotesize if $y$ is end} \\
\begin{tikzcd}[ampersand replacement=\&]\circ \ar[r] \& \bullet\ar[r] \& \bullet \end{tikzcd} & {\footnotesize if $x$ is end}  \\
\begin{tikzcd}[ampersand replacement=\&]\bullet \ar[r] \& \bullet \end{tikzcd} & {\footnotesize else}
\end{tabular}
\\\hline
\begin{tikzcd}[ampersand replacement=\&] {\circ} \ar[r,"\text{odd}",swap] \& {\circ}\end{tikzcd} &
\begin{tikzcd}[ampersand replacement=\&]\img{tangleinfty} \ar[r,"-D"] \& \img{tangleinfty}\\
 0 \& 0 \end{tikzcd} &
\begin{tikzcd}[ampersand replacement=\&]\circ\ar[r]  \& \circ\\
\phantom{\bullet} \& \end{tikzcd} &
\multirow{2}{*}[-5.5ex]{
\begin{tikzcd}[ampersand replacement=\&]\circ\ar[r]  \& \circ \end{tikzcd}
}
\\\cline{1-3}
\begin{tikzcd}[ampersand replacement=\&] {\circ} \ar[r,"\text{even}",swap] \& {\circ}\end{tikzcd} &
\begin{tikzcd}[ampersand replacement=\&]\img{tangleinfty} \ar[r,"-S^2"] \& \img{tangleinfty}\\
 0 \& 0 \end{tikzcd} &
\begin{tikzcd}[ampersand replacement=\&]\circ\ar[r]  \& \circ\\
\phantom{\bullet} \& \end{tikzcd} &
\\
\end{tabular}
\caption{From $zz(x)$ to $zz(x+1)$ in the proof of \cref{thm:thompson}.}
\label{table:fromxtox+1}
\end{table}

\subsection{The $\lambda$-distance between rational tangles}
Next, let us turn to \cref{thm:rationalreplacement}. Let us break up the proof into a sequence of lemmas.
\begin{lem}\label{lem:bigcasedistinction}
Let~$x\in\mathbb{Q}^+$. For every non-saddle edge $e$ of the zigzag graph~$zz(x)$,
there is a subgraph $\Gamma_e$ of $zz(x)$ as follows for some $n \geq 1$
(in what follows, a missing arrowhead means that the edge's direction is unknown):
\[
\Gamma_e = 
\begin{tikzcd}
A_1 \ar[r,dash]\ar[d,"e",swap] & A_2 \ar[r,dash] & \ldots \ar[r,dash] & A_n \\
B_1 \ar[r,dash] & B_2 \ar[r,dash] & \ldots \ar[r,dash] & B_n,
\end{tikzcd}
\]
such that for each $i$ with~$1 \leq i < n$,
the vertices $A_i$ and $B_i$ are of the same type ($\circ$~or~$\bullet$),
and the edges between $A_i$ and $A_{i+1}$ and between $B_i$ and $B_{i+1}$ are either both directed to the right, or both to the left (in the above drawing of~$\Gamma_e$); and such that
moreover, one of the following statements is true:
\begin{enumerate}[label=(\roman*)]
\item $A_n$ and $B_n$ are of the same type,
$A_n$ has no outgoing \emph{external} edge (i.e.~an edge towards a vertex in~$zz(x) \setminus \Gamma_e$),
and $B_n$ has no incoming external edge.
\item $n\geq 2$, and
\[
\begin{tikzcd}
A_{n-1} \ar[r,dash] & A_n \\
B_{n-1} \ar[r,dash] & B_n
\end{tikzcd}\quad\text{looks like}\quad
\begin{tikzcd}
\circ \ar[r] & \circ \\
\circ \ar[r] & \bullet.
\end{tikzcd}
\]
\item $n\geq 2$, and
\[
\begin{tikzcd}
A_{n-1} \ar[r,dash] & A_n \\
B_{n-1} \ar[r,dash] & B_n
\end{tikzcd}\quad\text{looks like}\quad
\begin{tikzcd}
\bullet &  \ar[l]\circ \\
\bullet &  \ar[l]\bullet.
\end{tikzcd}
\]
\end{enumerate}
\end{lem}
\begin{proof}
We proceed by induction over the number of transformations
$y\mapsto 1/y$ and $y \mapsto y + 1$ necessary to reach $x$ from~$1$.
For~$x = 1$, there is no non-saddle edge in~$zz(x)$, so the statement is trivially true.

Now assume the statement holds for~$zz(x)$, and let us show that it holds for~$zz(1/x)$, too.
Since $zz(1/x)$ arises from $zz(x)$ by reversing edge directions and switching $\circ$ and~$\bullet$,
there is a canonical one-to-one correspondence between edges $e$ of $zz(x)$ and edges $f$ of~$zz(1/x)$.
The subgraph $\Gamma_e$ of $zz(x)$ then corresponds to a subgraph $\Gamma_f$ of~$zz(1/x)$.
Note that in order for $f$ to be pointing downwards,
one needs to turn $\Gamma_f$ upside down.
Then, one sees that $\Gamma_e$ satisfies (i), (ii), (iii) if and only if $\Gamma_f$ satisfies (i), (iii), (ii), respectively.

Finally, assume that the statement holds for~$zz(x)$, and let us prove the statement for~$zz(x+1)$.
Every edge of $zz(x+1)$ comes from some edge $e$ of $zz(x)$ by applying the rules in \cref{table:x+1}.
So let us consider the six rules shown in \cref{table:x+1} case by case.%
\footnote{We invite the reader who is eager to check the following proof in detail to have a printed copy of \cref{table:x+1} handy. Also, they might be well advised to prepare themself a modified version of \cref{table:x+1}, in which all subgraphs are rotated by~$180^{\circ}$, so that they can quickly find out the fate of edges pointing to the left.}
In each case, $e$ gives rise to some number (between one and three) of edges in~$zz(x+1)$, exactly one of which is a non-saddle edge. That edge, which we denote by~$f$, is the one for which we need to check the existence of a subgraph $\Gamma_f$ as in the statement of the lemma. In each case (except when $e$ is a saddle edge),
denote by $\Gamma_e$ a subgraph of $zz(x)$ as in the statement of the lemma, satisfying (i), (ii), or (iii).
\begin{enumerate}
\item \emph{$e$ is an odd edge between $\bullet$-vertices.}
Applying the rules in \cref{table:x+1} to every edge in the subgraph
\[
\Gamma_e = 
\begin{tikzcd}
\bullet \ar[r,dash]\ar[d,"e",swap] & A_2 \ar[r,dash] & \ldots \ar[r,dash] & A_n \\
\bullet \ar[r,dash] & B_2 \ar[r,dash] & \ldots \ar[r,dash] & B_n,
\end{tikzcd}
\]
transforms it into a subgraph $\Gamma'$ of~$zz(x+1)$ that looks as follows:
\[
\Gamma' = 
\begin{tikzcd}
\circ\ar[d,"f",swap]\ar[r] & \bullet \ar[r,dash] & A'_3 \ar[r,dash] & \ldots \ar[r,dash] & A'_k \\
\circ\ar[r]       & \bullet \ar[r,dash] & B'_3 \ar[r,dash] & \ldots \ar[r,dash] & B'_{\ell}.
\end{tikzcd}
\]
We consider the following list of exhaustive subcases:
\newcommand{\zzendabove}[1]{\makebox[0pt][c]{#1}\raisebox{1.5ex}{\makebox[0pt][r]{\footnotesize{} end}}}
\begin{enumerate}
\item \emph{$\Gamma_e$ satisfies (i) with $n$ odd.}
In this case, let us show that $\Gamma_f=\Gamma'$ satisfies (i).
Firstly, since the top and bottom row of $\Gamma_e$ look the same, and the last edge in each row is not an even edge between $\bullet$ vertices,
the top and bottom row of $\Gamma'$ look the same, too.
More precisely, we have~$k = \ell$, $A'_i$ and $B'_i$ are of the same type for $3 \leq i \leq k$,
and horizontal edges in the first and second row of $\Gamma'$ point in the same directions.
Secondly, since $A_n$ has no outgoing external edge, neither does~$A'_k$.
And since $B_n$ has no incoming external edge, and no outgoing external odd edge between $\bullet$ vertices,
$B'_k$ has no incoming external edge, either. So $\Gamma'$ satisfies (i).
\item \emph{$\Gamma_e$ satisfies (i) with $n$ even, and both $A_{n-1}$ and $A_n$ are of type~$\circ$.}
Same as case (a).
\item \emph{$\Gamma_e$ satisfies (i) with $n$ even, $A_{n-1}$ and $A_n$ are both of type~$\bullet$, and $A_n$ and $B_n$ are ends.}
Same as case (a).
\item \emph{$\Gamma_e$ satisfies (i) with $n$ even, $A_{n-1}$ and $A_n$ are both of type~$\bullet$, $A_n$ is an end, and $B_n$ is not.}
Since $B_n$ is not an end, and has no external incoming edge, it must have an external outgoing edge~$h$, which must go to a vertex $B_{n+1}$ of type~$\bullet$. Applying the rules of \cref{table:x+1} to $\Gamma_e\cup \{h, B_{n+1}\}$ yields the following (note that~$k = \ell+1$):

\[
\tikz[
overlay]{
    \node[text=gray,align=right] at (-.3,0) {$\Gamma_e$};
    \filldraw[fill=yellow!50,draw=gray] (0,-1) rectangle (2.2,1);
}
\begin{tikzcd}
\cdots \bullet \ar[r,dash,"\text{even}",swap] & \zzendabove{$\bullet$} \\
\cdots \bullet \ar[r,dash] & \bullet \ar[r,"h",swap] & \bullet
\end{tikzcd}
\quad\rightsquigarrow\qquad
\tikz[
overlay]{
    \node[text=gray,align=right] at (-.3,0) {$\Gamma'$};
    \node at (3.4,-1.0) {$\scriptstyle B'_k$};
    \filldraw[fill=yellow!50,draw=gray] (0,-1) -- (2.2,-1) -- (2.2,0) -- (3.6,0) -- (3.6,1) -- (0,1) -- cycle;
}
\begin{tikzcd}
\cdots \bullet \ar[r,dash] & \bullet & \ar[l] \zzendabove{\ $\circ$} \\
\cdots \bullet \ar[r,dash] & \bullet & \ar[l,"h'"] \circ \ar[r] & \circ \ar[r] & \bullet
\end{tikzcd}
\]
\smallskip

Now, $\Gamma_f = \Gamma' \cup \{h', B'_{k}\}$ satisfies (i).
\item \emph{$\Gamma_e$ satisfies (i) with $n$ even, $A_{n-1}$ and $A_n$ are both of type~$\bullet$, $A_n$ is not an end, and $B_n$ is.}
Since $A_n$ is not an end, and has no external outgoing edge, it must have an external incoming edge $g$ from a vertex~$A_{n+1}$.
Applying the rules of \cref{table:x+1} to $\Gamma_e\cup \{g, A_{n+1}\}$ yields the following (note that~$\ell = k+1$):
\[
\tikz[
overlay]{
    \node[text=gray,align=right] at (-.3,0) {$\Gamma_e$};
    \filldraw[fill=yellow!50,draw=gray] (0,-1) rectangle (2.2,1);
}
\begin{tikzcd}
\cdots \bullet \ar[r,dash,"\text{even}",swap] & \bullet & \ar[l,"g",swap] A_{n+1} \\
\cdots \bullet \ar[r,dash] & \zzendabove{$\bullet$} 
\end{tikzcd}
\quad\rightsquigarrow\qquad
\tikz[
overlay]{
    \node[text=gray,align=right] at (-.3,0) {$\Gamma'$};
    \node at (3.4,.9) {$\scriptstyle A'_{\ell}$};
    \filldraw[fill=yellow!50,draw=gray] (0,1) -- (2.2,1) -- (2.2,0) -- (3.6,0) -- (3.6,-1) -- (0,-1) -- cycle;
}
\begin{tikzcd}
\cdots \bullet \ar[r,dash] & \bullet & \ar[l,"g'",swap] \circ & \ar[l] \circ \cdots \\
\cdots \bullet \ar[r,dash] & \bullet & \ar[l] \zzendabove{$\circ$} 
\end{tikzcd}
\]
\smallskip

Now, $\Gamma_f = \Gamma' \cup \{g', A'_{\ell}\}$ satisfies (i).
\item \emph{$\Gamma$ satisfies (i) with $n$ even, $A_{n-1}$ and $A_n$ are both of type~$\bullet$, and neither $A_n$ nor $B_n$ are ends.}
This is basically the synthesis of the two previous cases. As before, there is an external incoming edge $g$ from some vertex $A_{n+1}$ to~$A_n$, and an external outgoing edge $h$ from $B_n$ to some vertex~$B_{n+1}$.
Applying the rules of \cref{table:x+1} to $\Gamma_e\cup \{g, A_{n+1},h,B_{n+1}\}$ yields the following (note that~$k = \ell$):
\[
\tikz[
overlay]{
    \node[text=gray,align=right] at (-.3,0) {$\Gamma_e$};
    \filldraw[fill=yellow!50,draw=gray] (0,-1) rectangle (2.2,1);
}
\begin{tikzcd}
\cdots \bullet \ar[r,dash,"\text{even}",swap] & \bullet & \ar[l,"g",swap] A_{n+1} \\
\cdots \bullet \ar[r,dash] & \bullet \ar[r,"h",swap] & \bullet
\end{tikzcd}
\quad\rightsquigarrow\qquad
\tikz[
overlay]{
    \node[text=gray,align=right] at (-.3,0) {$\Gamma'$};
    \node at (3.4,.9) {$\scriptstyle A'_{k+1}$};
    \node at (4.6,.9) {$\scriptstyle A'_{k+2}$};
    \node at (3.4,-.9) {$\scriptstyle B'_{k+1}$};
    \node at (4.6,-.9) {$\scriptstyle B'_{k+2}$};
    \filldraw[fill=yellow!50,draw=gray] (0,-1) rectangle (2.2,1);
}
\begin{tikzcd}
\cdots \bullet \ar[r,dash] & \bullet & \ar[l,"g'",swap] \circ & \ar[l] \circ \makebox[0pt][l]{\ $\cdots$} \\
\cdots \bullet \ar[r,dash] & \bullet & \ar[l,"h'"] \circ \ar[r] & \circ \ar[r] & \bullet
\end{tikzcd}
\]
\smallskip

Now, $\Gamma_f = \Gamma' \cup \{g', A'_{k+1}, h', B'_{k+1}\}$ satisfies (i).
\item \emph{$\Gamma$ satisfies (ii).} In this case, we have $\ell = k + 1$ and
\[
\Gamma'= \quad\cdots\!
\begin{tikzcd}
 \ar[r,dash] & A'_{k-2} \ar[r,dash] & \circ \ar[r] & \circ \\
 \ar[r,dash] & B'_{k-2} \ar[r,dash] & \circ \ar[r] & \circ \ar[r,"h'"] & \bullet
\end{tikzcd}
\]
If $A_n$ does not have an external outgoing edge, then neither does~$A'_k$, and thus $\Gamma_f = \Gamma' \setminus \{B'_{k+1}, h'\}$ satisfies (i).
If $A_n$ has an external outgoing edge $g$ towards a vertex~$A_{n+1}$, then $A'_k$ also has an external outgoing edge $g'$ towards a vertex~$A'_{k+1}$, which must be of type~$\circ$. Then $\Gamma_f = \Gamma' \cup \{g', A'_{k+1}\}$ satisfies (ii).
\item \emph{$\Gamma_e$ satisfies (iii).} In this case, we have $\ell = k + 1$ and
\[
\Gamma'= \quad\cdots\!
\begin{tikzcd}
 \ar[r,dash] & A'_{k-3} \ar[r,dash] & \bullet & \ar[l] \circ & \ar[l] \circ \\
 \ar[r,dash] & B'_{k-3} \ar[r,dash] & \bullet & \ar[l] \circ & \ar[l] \circ \ar[r,"h'"] & \bullet
\end{tikzcd}
\]
We may proceed exactly as in case (g) above.
\end{enumerate}
\item \emph{$e$ is an even edge between $\bullet$-vertices, none of which are an end.}
Applying the rules in \cref{table:x+1} to every edge in the subgraph
\[
\Gamma_e = 
\begin{tikzcd}
\bullet \ar[r,dash]\ar[d,"e",swap] & A_2 \ar[r,dash] & \ldots \ar[r,dash] & A_n \\
\bullet \ar[r,dash] & B_2 \ar[r,dash] & \ldots \ar[r,dash] & B_n,
\end{tikzcd}
\]
transforms it into a subgraph $\Gamma'$ of~$zz(x+1)$ that looks as follows:
\[
\Gamma' = 
\begin{tikzcd}
\bullet\ar[d,"f",swap]\ar[r,dash] &  A'_2 \ar[r,dash] & \ldots \ar[r,dash] & A'_k \\
\bullet\ar[r,dash]       & B'_2 \ar[r,dash] & \ldots \ar[r,dash] & B'_{\ell}.
\end{tikzcd}
\]
Now, one may proceed exactly as in case~(1).

\item \emph{$e$ is an even edge between $\bullet$-vertices, directed towards an end.}
In~$\Gamma_e$, $B_1$ is an end. Thus we must have $n = 1$ and $\Gamma$ satisfying (i).
So in~$zz(x)$, there is an edge between $A_1$ and a vertex~$A_2$, directed towards~$A_1$.
The rules of \cref{table:x+1} transform $A_2\to A_1\to B_1$ into a subgraph $\Gamma'$ of $zz(x+1)$ that looks as follows:
\[
\Gamma' =  \quad
\begin{tikzcd}
\bullet \ar[d,"f",swap] & \ar[l] \circ & \ar[l] \circ \cdots \\
\bullet & \ar[l] \circ \makebox[0pt][l]{\footnotesize{} end.}
\end{tikzcd}
\]
Depending on the type of~$A_2$, the first row may have a fourth column, or not.
But either way, the first two columns of $\Gamma'$ form a subgraph $\Gamma_f$ of $zz(x+1)$ satisfying (i) with~$n = 2$.
\item \emph{$e$ is an even edge between $\bullet$-vertices, directed away from an end.}
In~$\Gamma$, $A_1$ is an end. Thus we must have $n = 1$ and $\Gamma$ satisfying (i).
So in~$zz(x)$, there is an edge between $B_1$ and a vertex~$B_2$, directed away from~$B_1$.
Since $B_1$ is a $\bullet$-vertex, so is~$B_2$.
Let us inspect the transformation given by \cref{table:x+1}:
\[
\begin{tikzcd}[ampersand replacement=\&]
\bullet \ar[d,"e",swap] \makebox[0pt][l]{\footnotesize{} end} \\
\bullet \ar[r] \& \bullet
\end{tikzcd}
\quad\leadsto\quad
\begin{tikzcd}[ampersand replacement=\&]
\bullet \ar[d,"f",swap] \& \ar[l] \circ \makebox[0pt][l]{\footnotesize{} end} \\
\bullet  \& \ar[l] \circ \ar[r] \& \circ\ar[r] \& \bullet
\end{tikzcd}
\]
The first two columns of the right-hand side form a subgraph $\Gamma_f$ of~$zz(x+1)$, satisfying (i) with n = 2.
\item \emph{$e$ is a saddle edge.} Note that we do not have a graph $\Gamma_e$ in this case.
Nevertheless, let us denote by $A_1$ the $\circ$-vertex adjacent to~$e$.
Let us inspect the transformation given by \cref{table:x+1} in each of the three cases that (a) $A_1$ is an end, (b) $A_1$ has an incoming edge and (c) $A_1$ has an outgoing edge:
\begin{align*}
(a)\quad
\begin{tikzcd}[ampersand replacement=\&]
\circ \ar[d,"e",swap] \makebox[0pt][l]{\footnotesize{} end} \ar[r,phantom] \& \phantom{A_1} \\
\bullet
\end{tikzcd}
&\quad\leadsto\quad
\begin{tikzcd}[ampersand replacement=\&]
\circ \ar[d,"f",swap] \makebox[0pt][l]{\footnotesize{} end} \\
\circ \ar[r] \& \bullet
\end{tikzcd}
\\[1ex]
(b)\quad
\begin{tikzcd}[ampersand replacement=\&]
\circ \ar[d,"e",swap] \& \ar[l] \circ \\
\bullet
\end{tikzcd}
&\quad\leadsto\quad
\begin{tikzcd}[ampersand replacement=\&]
\circ \ar[d,"f",swap] \& \ar[l] \circ \\
\circ \ar[r] \& \bullet
\end{tikzcd}
\\[1ex]
(c)\quad
\begin{tikzcd}[ampersand replacement=\&]
\circ \ar[d,"e",swap] \ar[r] \& \circ \\
\bullet
\end{tikzcd}
&\quad\leadsto\quad
\begin{tikzcd}[ampersand replacement=\&]
\circ \ar[d,"f",swap] \ar[r] \& \circ \\
\circ \ar[r] \& \bullet
\end{tikzcd}
\end{align*}
In cases (a) and (b), the edge $f$ and its end points form a subgraph $\Gamma_f$ of $zz(x+1)$ satisfying (i), with n = 1.
In case (c), the whole right-hand side is a subgraph $\Gamma_f$ of~$zz(x+1)$, satisfying (ii) with n = 2.
\item \emph{$e$ is an edge between two $\circ$-vertices.}
Treat this case similarly as~(1) and~(2).\myqed
\end{enumerate}%
\end{proof}
\begin{lem}\label{lem:homotopy2}
Let $x\in\mathbb{Q}^+$, and let $(C,d) = (\bigoplus_{i=0}^n A_i, \sum_{i=1}^n d_i)$ be a zigzag complex corresponding to~$zz(x)$.
Then, for every $i\in\{1,\ldots,n\}$, there exists a homotopy $h\colon C\to C$ such that
$h\circ d + d \circ h = f\cdot (\id_{A_i} + \id_{A_{i-1}})$ with $f = S^2$ if $d_i$ is odd (i.e.~$d_i=S$ or~$d_i=S^2$),
and $f = D$ if $d_i$ is even, i.e.~$d_i=D$.
\end{lem}
\begin{proof}
By reindexing $C$ if necessary, we assume without loss of generality that $d_i$ is a map $A_{i-1}\to A_i$.

If $d_i$ is~$S$, then let $h$ be given by $S\colon A_i\to A_{i-1}$. Then, for
$h\circ d_j$ to be non-zero, the target of $d_j$ and the domain of $h$ must match;
this happens only if~$j = i$, or $j = i + 1 \leq n$ and $d_{i+1}$ is a map $A_{i+1}\to A_i$.
In the latter case, we nevertheless have $h\circ d_{i+1} = 0$, since $h$ is $S$ and $d_{i+1}$ is~$D$.
Similarly, one sees that $d_j\circ h = 0$ unless~$j = i$. Overall, we find
\[
h\circ d + d \circ h
= \sum_{j=1}^n h\circ d_j + d_j \circ h
= h\circ d_i + d_i \circ h = S^2\cdot(\id_{A_i} + \id_{A_{i-1}})
\]
as desired.

If $d_i$ is not~$S$, denote by $e$ the edge in $zz(x)$ corresponding to~$d_i$.
Since $e$ is not a saddle edge, by the previous \cref{lem:bigcasedistinction}
there is a subgraph $\Gamma$ satisfying (i), (ii) or~(iii). The part of $C$ corresponding to $\Gamma$ is the following (drawn in black):
\begin{equation}\label{eq:homotopy}
\begin{tikzcd}
A_{i-1} \ar[r,dash,"d_{i-1}"]\ar[d,"d_{i}",swap]
& A_{i-2} \ar[r,dash] & \ldots \ar[r,dash,"d_{i-k+1}"] & A_{i-k} \\
A_i \ar[r,dash,"d_{i+1}",swap] \ar[u,homotopy,dashed,bend right=30,swap,"h_0"]
& A_{i+1} \ar[r,dash] \ar[u,homotopy,dashed,swap,"h_1"]
& \ldots \ar[r,dash,"d_{i+k-1}",swap] & A_{i+k-1}.
 \ar[u,homotopy,dashed,swap,"h_{k-1}"]
\end{tikzcd}
\end{equation}
Let the homotopy~$h$, drawn in \cref{eq:homotopy} in red and dashed, be defined as the sum
\[
h = \sum_{j=0}^{k-1} h_j\colon A_{i+j}\to A_{i-j-1}
\]
with $h_j$ equal to $(-1)^{i-j}$ times the identity cobordism if the domain and target of $h_j$ are both
\img{tangle0} or both \img{tangleinfty}; and $h_j$ equal to $(-1)^{i-j}$ times $S$ if one of the domain and target of $h_j$ is \img{tangle0}, and the other \img{tangleinfty}. Note that the latter case only happens if $\Gamma$ satisfies (ii) or (iii) and $j = k-1$. Now, $h\circ d + d \circ h$ is equal to the sum of the following terms $\alpha, \beta_j, \gamma_j, \delta$
(all other compositions of $h_j$ and $d_k$ vanish because target and domain do not match):
\begin{multline*}
\underbrace{d_i\circ h_0 + h_0\circ d_i}_{\alpha}
\quad+\quad 
\sum_{j=1}^{k-1} \underbrace{d_{i-j} \circ h_{j-1} + h_j \circ d_{i+j}}_{\beta_j}
\quad+\quad \\
\sum_{j=1}^{k-1} \underbrace{d_{i-j} \circ h_j + h_{j-1} \circ d_{i+j}}_{\gamma_j}
\quad+\quad
\underbrace{d_{i-k} \circ h_{k-1}
+
h_{k-1} \circ d_{i+k}}_{\delta}.
\end{multline*}
Now, observe that $\alpha$ equals
$f\cdot (\id_{A_i} + \id_{A_{i-1}})$ with $f = S^2$ if $d_i$ is $S^2$ and $f = D$ if $d_i$ is~$D$.
So it just remains to show that the terms $\beta_j, \gamma_j$ and $\delta$ are 0.
For each~$j$, one of $\beta_j$ and $\gamma_j$ is 0 because targets and domains do not match;
and the other term is 0 because the squares in \cref{eq:homotopy} anticommute (remember that
$d_{i-j}$ and $d_{i+j}$ both point to the left, or both to the right). This is also true for the last square in case that $\Gamma$ satisfies (ii) or (iii), in which case that square respectively looks like
\[
\begin{tikzcd}
\img{tangleinfty} \ar[r,"S^2"] & \img{tangleinfty} \\
\img{tangleinfty} \ar[r,"S",swap] \ar[u,homotopy,dashed,"\pm 1"] & \img{tangle0}\ar[u,homotopy,dashed,"\mp S",swap]
\end{tikzcd}\qquad\text{or}\qquad
\begin{tikzcd}
\img{tangle0} \ar[r,leftarrow,"S"] & \img{tangleinfty} \\
\img{tangle0} \ar[r,leftarrow,"S^2",swap] \ar[u,homotopy,dashed,"\pm 1"] & \img{tangle0}.\ar[u,homotopy,dashed,"\mp S",swap]
\end{tikzcd}
\]
Finally, in case $\Gamma$ satisfies (i), $\delta$ is 0 because targets and domains mismatch.
If $\Gamma$ satisfies (ii) or (iii), $\delta$ is 0 either for the same reason, or because $h_{k-1}$ is $S$
and $d_{i-j}$ and $d_{i+j}$ are~$D$.
\end{proof}
The following lemma is well-known (see e.g.~\cite{zbMATH03960559,KWZ2}), and can easily be checked inductively.
\begin{lem}\label{lem:parity}
For $i\in \{1,2\}$, let $p_i$ and $q_i$ be coprime integers.
Then $R(p_1/q_1)$ and $R(p_2/q_2)$ have the same connectivity if and only if
$p_1 \equiv p_2\pmod{2}$ and $q_1 \equiv q_2\pmod{2}$.\qed
\end{lem}
Let us call an end of a zigzag graph \emph{even} or \emph{odd} depending on whether the unique edge adjacent to it is even or odd.
\begin{lem}\label{lem:goodzigzag}
Let $x = p/q$ with $p,q$ positive and coprime. Then the following hold.
\begin{enumerate}[label=(\roman*)]
\item $zz(x)$ has an even end iff $p$ or $q$ is even.
\item $zz(x)$ has an odd $\circ$-end iff $p$ is odd.
\item $zz(x)$ has an odd $\bullet$-end iff $q$ is odd.
\end{enumerate}
\end{lem}
\begin{proof}
Let us show this by induction over
the number of transformations
$y\mapsto 1/y$ and $y \mapsto y + 1$ necessary to reach $x$ from~$1$.
For~$x = 1$, $zz(x)$ has both a $\circ$-end a $\bullet$-end, and $p=q=1$ are both odd, so the statement holds.
Let us now assume the statement holds for~$x$.
The zigzag graph $zz(1/x)$ is obtained from $zz(x)$ by switching $\circ$ and $\bullet$ and reversing all edges.
This switches $\circ$- and $\bullet$-ends, and does not change the parity of edges.
This corresponds to the fact that $x\mapsto 1/x$ switches the parity of $p$ and~$q$. Thus, the statement holds for~$1/x$.

To check the statement for~$x + 1$, one needs to analyze the effect of \cref{eq:zz+1} on parity and the type of ends ($\bullet$ or~$\circ$).
The possible configurations for $zz(x)$ are listed in the first two columns of the following table:\medskip\\
\begin{tabular}{l|l||l|l}
parities of $p,q$ & ends of $zz(x)$ & parities of $p+q,q$ & ends of $zz(x+1)$ \\\hline
odd,  odd         & odd $\circ$,    odd $\bullet$ & even, odd  & even $\circ$, odd $\bullet$ \\
even, odd         & even $\circ$,   odd $\bullet$ & odd, odd   & odd $\circ$,  odd $\bullet$ \\
even, odd         & even $\bullet$, odd $\bullet$ & odd, odd   & odd $\circ$,  odd $\bullet$ \\
odd, even         & even $\bullet$, odd $\circ$   & odd, even  & even $\circ$, odd $\circ$ \\
odd, even         & even $\circ$,   odd $\circ$   & odd, even  & even $\circ$, odd $\circ$ \\
\end{tabular}\medskip\\
The third and fourth column show the respective resulting configurations for~$zz(x+1)$.
The third column is straight-forward. To verify the fourth column, one needs to refer to \cref{table:x+1}.
From the above table one sees that the induction statement holds for~$x+1$. That concludes the proof.
\end{proof}

The next lemma is the heart of the proof. It is the analog of Lemmas~3.1 and 3.2\ in \cite{zbMATH07005602}.
\begin{lem}\label{lem:fg}
Let $p/q\in\mathbb{Q}^+$ with both $p$ and $q$ odd, and let $C$ be a zigzag complex corresponding to~$zz(p/q)$.
Let $C'$ be the complex
\[
C'_0=\img{tangle0}\xrightarrow{\quad S\quad }\img{tangleinfty}\{1\}=C'_1,
\]
which is (up to global shifts) the Bar-Natan complex of $R(-1)$ equipped with some orientation.
Then there are ungraded chain maps $f\colon C\to C'$ and $g\colon C' \to C$,
such that $g\circ f \simeq G\cdot\id_C$ and $f\circ g\simeq G\cdot\id_{C'}$.
\end{lem}
\begin{figure}[b]%
\newcommand{\abgd}[1]{\makebox[0pt]{\raisebox{-\baselineskip}{%
\textcolor{annotation}{$#1$}}}}
\begin{tikzcd}[column sep=1.4em,row sep=huge,execute at end picture={;}]
\makebox[0pt][r]{\textcolor{annotation}{$\scriptstyle A_0$\ }}\img{tangleinfty}\ar[r,swap,"S"]\ar[r,homotopy,leftarrow,dashed,bend left=30,"S"]
\ar[rrrrr,homotopy,leftarrow,dashed,rounded corners,to path={-- ([yshift=7ex]\tikztostart.north) -- node[text=red,above]{$\scriptstyle 1$}([yshift=7ex]\tikztotarget.north) -- (\tikztotarget.north) }] &
\img{tangle0}    \ar[r,swap,"D"]\ar[r,homotopy,leftarrow,dashed,bend left=30,"-1"]
\ar[rrr,homotopy,leftarrow,dashed,rounded corners,to path={-- ([yshift=4ex]\tikztostart.north) -- node[text=red,above]{$\scriptstyle -1$}([yshift=4ex]\tikztotarget.north) -- (\tikztotarget.north) }] & 
\img{tangle0}    \ar[r,swap,"S^2"]\ar[r,homotopy,leftarrow,dashed,bend left=30,"1"]
\ar[rrr,homotopy,dashed,rounded corners,to path={-- ([yshift=-2ex]\tikztostart.south) -- node[text=red,below]{$\scriptstyle S$}([yshift=-2ex]\tikztotarget.south) -- (\tikztotarget.south) },"S"] & 
\img{tangle0}    \ar[r,swap,leftarrow,"D"]\ar[r,homotopy,dashed,bend left=30,"-1"]
\ar[rrrrr,homotopy,leftarrow,dashed,rounded corners,to path={-- ([yshift=-7ex]\tikztostart.south) -- node[text=red,below]{$\scriptstyle -1$}([yshift=-7ex]\tikztotarget.south) -- (\tikztotarget.south) },"-1"] & 
\img{tangle0}    \ar[r,swap,leftarrow,"S"]\ar[r,homotopy,dashed,bend left=30,"S"]
\ar[rrr,homotopy,leftarrow,dashed,rounded corners,to path={-- ([yshift=-4ex]\tikztostart.south) -- node[text=red,below]{$\scriptstyle 1$}([yshift=-4ex]\tikztotarget.south) -- (\tikztotarget.south) },"1"] & 
\img{tangleinfty}\ar[r,swap,"D"]\ar[r,leftarrow,homotopy,dashed,bend left=30,"-1"] & 
\img{tangleinfty}\ar[r,swap,"S"]\ar[r,leftarrow,homotopy,dashed,bend left=30,"S"] & 
\img{tangle0}    \ar[r,swap,"D"]\ar[r,leftarrow,homotopy,dashed,bend left=30,"-1"] & 
\img{tangle0}    \ar[r,swap,"S^2"]\ar[r,leftarrow,homotopy,dashed,bend left=30,"1"] & 
\img{tangle0} \makebox[0pt][l]{\textcolor{annotation}{\ $\scriptstyle A_9$}}\\
\makebox[0pt][r]{\textcolor{annotation}{$\scriptstyle C_1'$\ }}\img{tangleinfty}
\ar[rrrrrrrrr,leftarrow,"S"]
\ar[rrrrrrrrr,homotopy,dashed,bend right=10,swap,"S"]
\ar[u,leftarrow,firstmap,bend left,"1\abgd{\alpha\ \ }"]
\ar[u,secondmap,bend right,swap,"-D\abgd{\!\!\!\!\!\!\beta}"]
&&&&&&&&& \img{tangle0}\makebox[0pt][l]{\textcolor{annotation}{\ $\scriptstyle C_0'$}}
\ar[u,leftarrow,firstmap,bend left,"-D\abgd{\gamma\ \ }"]
\ar[u,secondmap,bend right,swap,"1\abgd{\!\!\delta}"] 
\end{tikzcd}
\caption{Illustration of the proof of \cref{lem:fg}. In the top row, a zigzag complex $C$ with graph $zz(3/7)$
(compare \cref{fig:37}).
On the bottom row, the complex~$C'$. In light and dark blue, the ungraded chain maps $f\colon C\to C'$
(going down) and $g\colon C'\to C$ (going up).
Red and dashed, the required homotopies.
Homological and quantum degree shifts are omitted from the diagram.}
\label{fig:lemfg}
\end{figure}
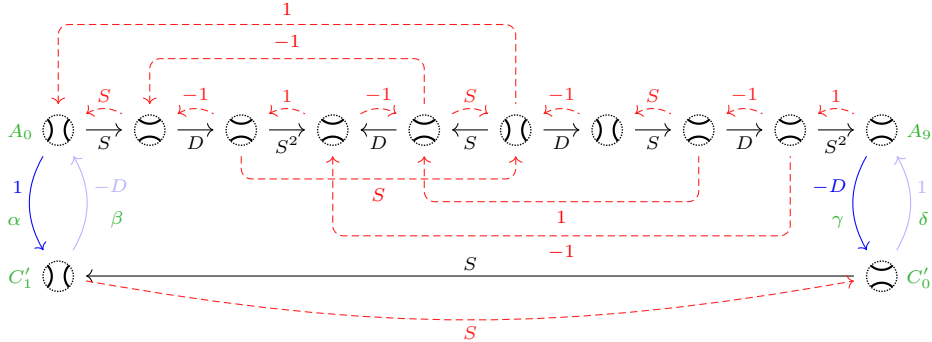
\begin{proof}
Let us write
\[
(C,d) = \Bigl(\bigoplus_{i=0}^n A_i, \sum_{i=1}^n d_i\Bigr).
\]
By reindexing this zigzag if necessary, we may assume that the end of $\Gamma$ corresponding to $A_0$
is the $\circ$-end.
We are going to define $f$ as sum of two ungraded chain maps $\alpha\colon A_0 \to C'_1$ and $\gamma\colon A_n \to C'_0$,
and $g$ as sum of two ungraded chain maps $\beta\colon C'_1\to A_0$ and $\delta\colon C'_0\to A_n$.
See \cref{fig:lemfg} for an example.
Namely, let $\beta$ and $\gamma$ be~$-D$, and $\alpha$ and $\delta$ be~$1$.
One may check inductively using \cref{table:x+1} (see \cite[Remark~12.14]{iltgenphd} for details) that the $\circ$-end of $\Gamma$ has an outgoing differential, and the $\bullet$-end has an incoming differential; in other words, $d_1$ is a map~$A_0 \to A_1$, and $d_n$ is a map $A_{n-1} \to A_n$. Combined with the fact that $d_1$ and $d_n$ are both $S$ or~$S^2$, this implies that $f$ and $g$ are ungraded chain maps.
Now, one calculates that
\begin{align*}
f\circ g & = \alpha\circ \beta + \gamma\circ\delta \\
         & = (-D\colon C'_0\to C'_0) + (-D\colon C'_1\to C'_1).
\end{align*}
Let $h'\colon C'\to C'$ be $S\colon C'_1\to C'_0$. Then $d'\circ h' + h'\circ d' = G\cdot \id_{C'}  - f\circ g$,
so $f\circ g$ is homotopic to~$G$, as desired.

Similarly, one finds $g\circ f = (-D\colon A_0\to A_0) + (-D\colon A_n\to A_n)$.
By \cref{lem:homotopy2}, for each~$i$, there exists a homotopy $h_i\colon C\to C$ such that
$h_i\circ d + d\circ h_i$ equals $u\cdot (\id_{A_i} + \id_{A_{i-1}})$ with $u = S^2$ if $d_i$ is odd
and $u = D$ if $d_i$ is even. Let $h = \sum_i (-1)^{i+1} h_i$.
Now, one sees that
\[
h\circ d + d\circ h = G\cdot \id_C - g\circ f,
\]
which concludes the proof.
\end{proof}
We now need to examine rational replacements (first seen in \cref{dfn:ratioreplacement}) more closely.
\begin{dfn}\label{dfn:ratioreplacement2}
Two unoriented link $L, L'\subset S^3$ are related by a
\emph{rational replacement} if, after an isotopy, there exists
a ball $B\subset S^3$ whose boundary sphere intersects $L$ and $L'$ transversely,
such that $L \setminus B^{\circ} = L' \setminus B^{\circ}$,
and the two tangles $T = L\cap B$ and $T' = L'\cap B$ are rational.
If $T$ and $T'$ have the same connectivity, we say that the rational replacement is \emph{proper}.
If there is a homeomorphism between $B$ and the unit ball that sends $T$ to $R(x)$
and $T'$ to $R(y)$ for some $x,y\in\Q\cup\{\infty\}$,
we speak of an \emph{$x$ by $y$ rational replacement}.
\end{dfn}
It is a frequently used fact that a crossing change may be seen as $-1$ by $1$ rational replacement,
but also as $0$ by $2$ rational replacement.
The following lemma generalizes this.
\begin{lem}\label{lem:changerationalcoeff}
Let $S, T$ be rational tangles in a ball~$B$, let $x,y\in\Q\cup \{\infty\}$,
and let $\varphi$ be a homeomorphism of $B$ to the unit ball $B_0$ such that $\varphi(S) = R(x)$
and $\varphi(T) = R(y)$. Then there exists $z\in \{-1\}\cup [0,\infty)$,
and a homeomorphism $\varphi'\colon B\to B_0$ such that $\varphi'(S) = R(-1)$ and $\varphi'(T) = R(z)$.
\end{lem}
\begin{proof}
Let $\psi_1$ be a homeomorphism of $B_0$ with $\psi_1(R(x)) = R(\infty)$.
Then $\psi_1(R(y)) = R(y')$ for some~$y'$.
Let $\psi_2$ be a homeomorphism of $B_0$
such that $\psi_2(R(\infty)) = R(\infty)$
and $\psi_2(R(y')) = R(y'')$ with $y'' \in (0,1]\cup\{\infty\}$.
Such a $\psi_2$ may be constructed by adding a certain number of twists to the right side of the ball.
Finally, let $\psi_3$ be the homeomorphism of $B_0$ that sends $R(w)$ to $R(1/w - 1)$ for all $w\in\Q\cup\{\infty\}$.
Then $\psi_3\circ\psi_2\circ\psi_1\circ\varphi(S) = R(-1)$ and
$\psi_3\circ\psi_2\circ\psi_1\circ\varphi(T) = R(z)$ for $z\in \{-1\}\cup [0,\infty)$, as desired.
\end{proof}
\propdiscrete*
\begin{proof}
By \cref{lem:changerationalcoeff}, there exists $x \in \{-1\}\cup [0,\infty)$ and a homeomorphism
that sends $S$ to $R(-1)$ and $T$ to~$R(x)$.
Since $S$ and $T$ are not equivalent, we have~$x\neq -1$.
Since the connectivities of $S$ and $T$ are the same, \cref{lem:parity} implies that
$x = p/q$ with both $p$ and $q$ odd (in particular,~$p/q \neq 0$).
By \cref{prop:lambdaequivariant}, $\lambda$ is equivariant under homeomorphisms,
and so we have $\lambda(S, T) = \lambda(R(-1),R(x))$. So it will be sufficient to show that $\lambda(R(-1),R(x)) = 1$.

By \cref{thm:thompson}, $[R(x)]^{\bullet}$ is homotopy equivalent to a zigzag complex $C$ with graph~$zz(x)$.
By \cref{lem:fg}, there are ungraded chain maps $f\colon [R(-1)]^{\bullet} \to C$
and $g\colon C \to [R(-1)]^{\bullet}$ with $g\circ f \simeq G\cdot \id_{[R(-1)]^{\bullet}}$ and 
$f\circ g \simeq G\cdot \id_{C}$, showing $\lambda(R(-1),R(x)) \leq 1$.

Let $\mathcal{D}$ be the following 2-input planar arc diagram:
\[
\begingroup%
  \makeatletter%
  \providecommand\color[2][]{%
    \errmessage{(Inkscape) Color is used for the text in Inkscape, but the package 'color.sty' is not loaded}%
    \renewcommand\color[2][]{}%
  }%
  \providecommand\transparent[1]{%
    \errmessage{(Inkscape) Transparency is used (non-zero) for the text in Inkscape, but the package 'transparent.sty' is not loaded}%
    \renewcommand\transparent[1]{}%
  }%
  \providecommand\rotatebox[2]{#2}%
  \newcommand*\fsize{\dimexpr\f@size pt\relax}%
  \newcommand*\lineheight[1]{\fontsize{\fsize}{#1\fsize}\selectfont}%
  \ifx\svgwidth\undefined%
    \setlength{\unitlength}{84.58745894bp}%
    \ifx\svgscale\undefined%
      \relax%
    \else%
      \setlength{\unitlength}{\unitlength * \real{\svgscale}}%
    \fi%
  \else%
    \setlength{\unitlength}{\svgwidth}%
  \fi%
  \global\let\svgwidth\undefined%
  \global\let\svgscale\undefined%
  \makeatother%
  \begin{picture}(1,0.69442001)%
    \lineheight{1}%
    \setlength\tabcolsep{0pt}%
    \put(0,0){\includegraphics[width=\unitlength,page=1]{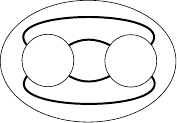}}%
    \put(0.11748526,0.20573563){\makebox(0,0)[lt]{\lineheight{1.25}\smash{\begin{tabular}[t]{l}\textbf{*}\end{tabular}}}}%
    \put(0.58788111,0.2028397){\makebox(0,0)[lt]{\lineheight{1.25}\smash{\begin{tabular}[t]{l}\textbf{*}\end{tabular}}}}%
  \end{picture}%
\endgroup%

\]
Then $\mathcal{D}(R(-1), R(2))$ is the unknot,
and $\mathcal{D}(R(x), R(2))$ is the two-bridge knot $K$ corresponding to $x + 2 = (p + 2q) / q$.
Since~$x+2 > 1$, this is a non-trivial knot, and so we have $\lambda(K) > 0$ because $\lambda$ detects the unknot (see \cref{prop:lambdadetectsunknot}).%
\footnote{Here, we do not even need Kronheimer--Mrowka's theorem that Khovanov homology detects the unknot, but only the (much easier) theorem that Khovanov homology detects the unknot among two-bridge knots (in fact, already the Jones polynomial can be seen to accomplish that).}
Overall, using \cref{lem:lambdainputdiagram}, we get
\[
\lambda(R(-1),R(x)) \geq \lambda(\mathcal{D}(R(-1), R(2)), \mathcal{D}(R(x), R(2))) = \lambda(K) > 0.
\]
This concludes the proof.
\end{proof}
\begin{proof}[Proof of \cref{thm:rationalreplacement}]
To show $\lambda(K) \leq u_q(K)$, it is sufficient to show the following:
if two knots $K$ and $J$ are related by a proper rational replacement, then $\lambda(K,J) \leq 1$.
So let knots~$K$, $J$ related by a proper rational replacement be given.
By definition, there exists a 4-ended tangle~$T$, such that $K$ is the union of $T$ with a rational tangle~$S$,
and $J$ is the union of $T$ with a another rational tangle~$S'$.
Since the replacement is proper, $S$ and $S'$ have the same connectivity.
So $\lambda(S, S') \leq 1$ by \cref{prop:discrete},
and thus $\lambda(K, J) \leq 1$ by \cref{prop:gluing}.
\end{proof}

\begin{appendix}
\section{Proof of \refinsec{Proposition}{prop:equivcat}}\label{sec:appendix}
This appendix is devoted to the proof of the following proposition.
\rpropequivcat*

\begin{proof}
	We need to check that $B$ is faithful, full and dense. As explained before, density of $B$ follows directly from delooping. To show that $B$ is faithful and full, we are going to look at the morphism spaces
	\begin{equation*}
	\begin{split}
		& \Hom_\mathcal{E}(D_{T_0}, D_{T_0}),\\ &
		\Hom_{\Cob^3(2)}(B(D_{T_0}), B(D_{T_0})),\\ &
		\Hom_{\Cob^3_{\modl}(2)}(B(D_{T_0}), B(D_{T_0})),
	\end{split}
	\end{equation*}
	where $D_{T_0}$ is the diagram of the trivial $2$-ended tangle~$T_0$.
	
	Let $G, \Sigma_i$, $i \in \Z_{\geq 0}$ be formal variables. We introduce a grading on $\Z[G]$ and $\Z[G, \Sigma_0, \Sigma_1, \Sigma_2, \dots]$ by setting $\deg G = -2$ and $\deg \Sigma_i = 2-2i$. Then there is an isomorphism of graded Abelian groups\footnotemark
	\begin{equation*}
		\Hom_\mathcal{E}(D_{T_0}, D_{T_0}) \cong \Z[G]
	\end{equation*}
	given by mapping a connected cobordism of genus $k$ to~$G^k$. Similarly, there is an isomorphism of graded Abelian groups\footnotemark[\value{footnote}]\footnotetext{This is in fact an isomorphism of graded rings if we declare multiplication in $\Hom_{\mathcal{E}}$ (resp.~$\Hom_{\Cob^3(2)}$) as composition of cobordisms.}
	\begin{equation*}
		\Hom_{\Cob^3(2)}(B(D_{T_0}), B(D_{T_0})) \cong \Z[G, \Sigma_0, \Sigma_1, \Sigma_2, \dots],
	\end{equation*}
	given by mapping a cobordism, which consists of the marked component with genus $k$ and a disjoint union of $n_i$ many closed surfaces of genus~$i$, to the product $G^k\prod_{i=0}^\infty \Sigma_i^{n_i}$. In order to determine $\Hom_{\Cob^3_{\modl}(2)}(B(D_{T_0}), B(D_{T_0}))$, we need to understand how the local relations~$S$, $T$, and $4Tu$ in $\Cob^3_{\modl}(2)$ affect the ring $\Z[G, \Sigma_0, \Sigma_1, \Sigma_2, \dots]$. Introducing $S$ and $T$ translates to $\Sigma_0 = 0$ and~$\Sigma_1 = 2$. Next, we replace $4Tu$ with the equivalent $3S_2$ relation (cf.~\cite[Section 11.4]{BN1}, and \cref{fig:3S2} below) which is easier to handle as there are at most three surfaces involved. We name the surfaces in the relation $A, B, C$ in a clockwise manner starting top left. Suppose that $g(A) = a$, $g(B) = b$,~$g(C) = c$. Let $M_n$ be the curtain with genus~$g(M_n) = n$.
	
	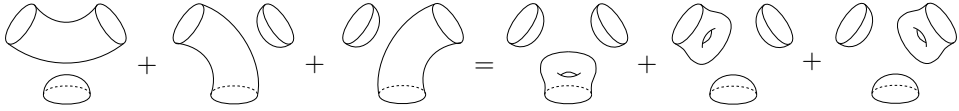
\begin{figure}[ht]
		\centering
		\def\svgwidth{\textwidth}
               {
\begingroup%
  \makeatletter%
  \providecommand\color[2][]{%
    \errmessage{(Inkscape) Color is used for the text in Inkscape, but the package 'color.sty' is not loaded}%
    \renewcommand\color[2][]{}%
  }%
  \providecommand\transparent[1]{%
    \errmessage{(Inkscape) Transparency is used (non-zero) for the text in Inkscape, but the package 'transparent.sty' is not loaded}%
    \renewcommand\transparent[1]{}%
  }%
  \providecommand\rotatebox[2]{#2}%
  \newcommand*\fsize{\dimexpr\f@size pt\relax}%
  \newcommand*\lineheight[1]{\fontsize{\fsize}{#1\fsize}\selectfont}%
  \ifx\svgwidth\undefined%
    \setlength{\unitlength}{851.03311813bp}%
    \ifx\svgscale\undefined%
      \relax%
    \else%
      \setlength{\unitlength}{\unitlength * \real{\svgscale}}%
    \fi%
  \else%
    \setlength{\unitlength}{\svgwidth}%
  \fi%
  \global\let\svgwidth\undefined%
  \global\let\svgscale\undefined%
  \makeatother%
  \begin{picture}(1,0.10709594)%
    \lineheight{1}%
    \setlength\tabcolsep{0pt}%
    \put(0,0){\includegraphics[width=\unitlength,page=1]{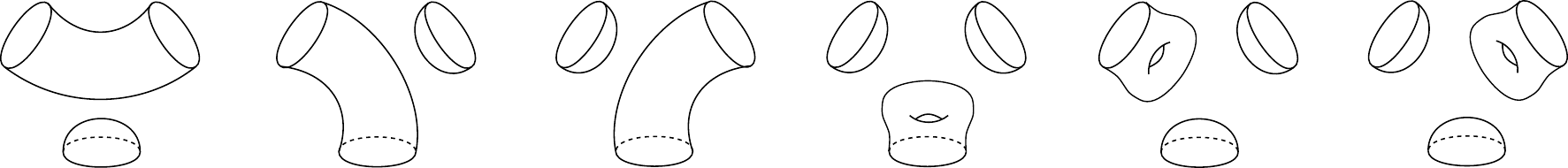}}%
    \put(0.13839834,0.03476824){\makebox(0,0)[lt]{\lineheight{1.25}\smash{\begin{tabular}[t]{l}+\end{tabular}}}}%
    \put(0.31466373,0.03565843){\makebox(0,0)[lt]{\lineheight{1.25}\smash{\begin{tabular}[t]{l}+\end{tabular}}}}%
    \put(0.49270951,0.033878){\makebox(0,0)[lt]{\lineheight{1.25}\smash{\begin{tabular}[t]{l}=\end{tabular}}}}%
    \put(0.66363351,0.03521331){\makebox(0,0)[lt]{\lineheight{1.25}\smash{\begin{tabular}[t]{l}+\end{tabular}}}}%
    \put(0.83722811,0.0383291){\makebox(0,0)[lt]{\lineheight{1.25}\smash{\begin{tabular}[t]{l}+\end{tabular}}}}%
  \end{picture}%
\endgroup%
}
		\caption{The $3S_2$-relation.}
		\label{fig:3S2}
	\end{figure}
	
	Suppose first that $A = B = M_0 \neq C$. In this case, the $3S_2$ relation translates to
	\begin{equation*}
		G\Sigma_c + G^c + G^c = \Sigma_{c+1} + G\Sigma_c + G\Sigma_c \iff \Sigma_{c+1} = 2G^c - G\Sigma_c.
	\end{equation*}
	By the $S$ relation, we have $\Sigma_0 = 0$ and thus $\Sigma_1 = 2G^0 - G\cdot 0 = 2$, which coincides with the $T$ relation. By induction, we therefore obtain the relation
	\begin{equation*}\tag{$*$}
		\Sigma_c = \begin{cases}
			2G^{c-1}, & c \textrm{ odd},\\
			0, & c \textrm{ even},
		\end{cases}
	\end{equation*}
	giving us a surjection $\Hom_{\Cob^3_{\modl}(2)}(B(D_{T_0}), B(D_{T_0})) \twoheadrightarrow \Z[G]$. We claim that there are no other relations introduced, i.e.\ that this surjection is an isomorphism. For this, we check all possible general cases of the $3S_2$ relation.
	\newline
	\newline
	\textbf{Case 1}: $A, B, C$ are three different closed surfaces (i.e.\ none of them is a curtain~$M_n$). In this case, $3S_2$ translates to
	\begin{equation*}
		\Sigma_{a+b}\Sigma_c + \Sigma_{a+c}\Sigma_b + \Sigma_{b+c}\Sigma_a = \Sigma_{a+1}\Sigma_b\Sigma_c + \Sigma_a\Sigma_{b+1}\Sigma_c + \Sigma_a\Sigma_b\Sigma_{c+1}.
	\end{equation*}
	If all $a \equiv b\equiv c \mod 2$ or if $a$ is odd and $b, c$ even, both sides of the equation vanish after applying~$(*)$. On the other hand, if $a$ is even and $b, c$ odd, then $(*)$ gives us
	\begin{equation*}
		4G^{a+b+c-2} + 4G^{a+b+c-2} + 0 = 8G^{a+b+c-2} + 0 + 0,
	\end{equation*}
	showing that there is no new relation introduced. Since $3S_2$ is symmetric in $A, B, C$, no other parities of $a, b, c$ need to be checked in this case.
	\newline
	\newline
	\textbf{Case 2}: $A = B \neq C$, none of them is a curtain~$M_n$. In this case we have
	\begin{equation*}
		\Sigma_{a+1}\Sigma_c + \Sigma_{a+c} + \Sigma_{a+c} = \Sigma_a\Sigma_{c+1} + \Sigma_{a+1}\Sigma_c + \Sigma_{a+1}\Sigma_c.
	\end{equation*}
	If $a \equiv c \mod 2$ both sides of the equation vanish after applying~$(*)$. If $a$ is even and $c$ odd, we obtain
	\begin{equation*}
		4G^{a+c-1} + 0 + 0 = 0 + 2G^{a+c-1} + 2G^{a+c-1},
	\end{equation*}
	showing that there is no new relation introduced. Again by symmetry of the $3S_2$ relation, no other cases $a$ and $c$ need to be checked.
	\newline
	\newline
	\textbf{Case 3}: $A = B = C$, none of them is a curtain~$M_n$. In this case we have
	\begin{equation*}
		3\Sigma_{a+1} = 3\Sigma_{a+1},
	\end{equation*}
	showing immediately that there are no new relations introduced.
	\newline
	\newline
	\textbf{Case 4}: $A = M_a$, $B$, $C$ are three different surfaces. The $3S_2$ relations translates to
	\begin{equation*}
		G^{a+b}\Sigma_c + G^{a+c}\Sigma_b + G^a\Sigma{b+c} = G^{a+1}\Sigma_b\Sigma_c + G^a\Sigma_{b+1}\Sigma_c + G^a\Sigma_b\Sigma_{c+1}.
	\end{equation*}
	If $b, c$ are even, both sides vanish after applying~$(*)$. If $b$ is even and $c$ is odd, we obtain
	\begin{equation*}
		2G^{a+b+c-1} + 0 + 2G^{a+b+c-1} = 0 + 4G^{a+b+c-1} + 0,
	\end{equation*}
	and if $b$ is odd and $c$ even, we get
	\begin{equation*}
		0 + 2G^{a+b+c-1} + 2G^{a+b+c-1} = 0 + 0 + 4G^{a+b+c-1}.
	\end{equation*}
	In both cases, no new relations are introduced.
	\newline
	\newline
	\textbf{Case 5}: $A = B = M_a \neq C$. In this case, we get
	\begin{equation*}
		G^{a+1}\Sigma_c + G^{a+c} + G^{a+c} = G^{a+1}\Sigma_c + G^{a+1}\Sigma_c + G^a\Sigma_{c+1}.
	\end{equation*}
	If $c$ is odd, applying $(*)$ yields
	\begin{equation*}
		2G^{a+c} + G^{a+c} + G^{a+c} = 2G^{a+c} + 2G^{a+c} + 0,
	\end{equation*}
	and if $c$ is even,
	\begin{equation*}
		0 + G^{a+c} + G^{a+c} = 0 + 0 + 2G^{a+c}.
	\end{equation*}
	In both cases, no new relations are introduced.
	\newline
	\newline
	\textbf{Case 6}: $A = M_a$, $B = C = M_c$, $M_a \neq M_c$. We have
	\begin{equation*}
		G^{a+c} + G^{a+c} + G^a\Sigma_{c+1} = G^{a+1}\Sigma_c + G^a\Sigma_{c+1} + G^a\Sigma_{c+1}.
	\end{equation*}
	If $c$ is odd, we obtain after applying $(*)$
	\begin{equation*}
		G^{a+c} + G^{a+c} + 0 = 2G^{a+c} + 0 + 0,
	\end{equation*}
	and if $c$ is even,
	\begin{equation*}
		G^{a+c} + G^{a+c} + 2G^{a+c} = 0 + 2G^{a+c} + 2G^{a+c}.
	\end{equation*}
	In both cases, no new relations are introduced.
	\newline
	\newline
	\textbf{Case 7}: $A = B = C = M_a$. As in the third case, we have
	\begin{equation*}
		3G^a = 3G^a,
	\end{equation*}
	showing immediately that there are no new relations introduced.
	
	The above shows that there are isomorphisms
	\begin{equation*}
	\begin{split}
		& \varphi: \Hom_\mathcal{E}(D_{T_0}, D_{T_0}) \overset{\cong}\rightarrow \Z[G] \\ &
		\psi: \Hom_{\Cob^3_{\modl}(2)}(B(D_{T_0}), B(D_{T_0})) \overset{\cong}\rightarrow \Z[G].
	\end{split}
	\end{equation*}
	Consider the diagram
	\begin{center}
	\begin{tikzcd}
		& \Z[G] \arrow[ld, swap, "\varphi"] \arrow{rd}{\psi} & \\
		\Hom_\mathcal{E}(D_{T_0}, D_{T_0}) \arrow{rr}{B} & & \Hom_{\Cob^3_{\modl}(2)}(B(D_{T_0}), B(D_{T_0})).
	\end{tikzcd}
	\end{center}
	By construction this diagram commutes, i.e.\ $B \circ \varphi = \psi$. Since both $\varphi$ and $\psi$ are isomorphisms, $B$ has to be an isomorphism as well.
\end{proof}

\end{appendix}
\bibliographystyle{myamsalpha}
\bibliography{References}
\end{document}